\documentclass[11pt,reqno,oneside,a4paper]{amsart}
\usepackage[utf8]{inputenc}
\usepackage{amsfonts, amsmath, amssymb}
\usepackage{amsthm}
\usepackage{mathtools}
\usepackage{fullpage}
\usepackage{xcolor}
\usepackage{tikz-cd}
\usepackage{scalerel}
\usepackage{setspace}
\usetikzlibrary{calc}
\usetikzlibrary{shapes.geometric}
\usepackage[all]{xy}
\usepackage[shortlabels]{enumitem}
\usepackage{bm}
\usepackage[normalem]{ulem} 
\usepackage{mathrsfs}
\usepackage{comment}
\usepackage{hyperref}

\newtheorem{theorem}[equation]{Theorem}
\newtheorem{lemma}[equation]{Lemma}

\newtheorem{proposition}[equation]{Proposition}
\newtheorem{corollary}[equation]{Corollary}

\newtheorem*{theorem*}{Theorem}

{\theoremstyle{definition}\newtheorem{notation}[equation]{Notation}}
{\theoremstyle{definition}\newtheorem{remark}[equation]{Remark}}
{\theoremstyle{definition}\newtheorem{defn}[equation]{Definition}}
{\theoremstyle{definition}\newtheorem{example}[equation]{Example}}

\newcommand{\DMO}{\DeclareMathOperator}

\newcommand{\beq}{\begin{equation}}
\newcommand{\eeq}{\end{equation}}

\newcommand{\ang}[1]{\langle #1 \rangle}

\newcommand\restr[2]{{
  \left.\kern-\nulldelimiterspace 
  #1 
  \right|_{#2} 
  }}
  
\renewcommand{\leq}{\leqslant}
\renewcommand{\geq}{\geqslant}

\DeclareMathOperator{\id}{{id}}
\DeclareMathOperator{\Hom}{{Hom}}

\DeclareMathOperator{\End}{{End}}
\DeclareMathOperator{\Ext}{{Ext}}

\DeclareMathOperator{\Tor}{Tor}

\DeclareMathOperator{\Proj}{Proj}
\DeclareMathOperator{\Spec}{Spec}

\DeclareMathOperator{\pd}{pd}

\DeclareMathOperator{\coker}{coker}
\DeclareMathOperator{\im}{im}

\DeclareMathOperator{\gr}{gr}

\newcommand{\Qgr}[1][]{Q^{#1}_{\operatorname{gr}}}

\DeclareMathOperator{\spann}{span}

\DeclareMathOperator{\hilb}{hilb}

\DeclareMathOperator{\Perf}{Perf}

\newcommand{\mc}{\mathcal}
\newcommand{\mf}{\mathfrak}

\newcommand{\hash}{\hspace{1pt} \# \hspace{1pt}}

\DeclareMathOperator{\shHom}{\mathcal{H}\!\mathit{om}}
\DeclareMathOperator{\shExt}{\mathcal{E}\!\mathit{xt}}

\newcommand{\kk}{{\Bbbk}}

\newcommand{\ZZ}{{\mathbb Z}}
\newcommand{\PP}{{\mathbb P}}
\newcommand{\NN}{{\mathbb N}}

\renewcommand{\AA}{{\mathbb A}} 

\newcommand{\mb}{\mathbb}
\newcommand{\CC}{\mb C}
\newcommand{\DD}{\mb D}
\newcommand{\EE}{\mb E}

\newcommand{\KK}{\mb K}
\newcommand{\LL}{\mb L}
\newcommand{\RR}{\mb R}

\newcommand{\vC}{\check{C}}

\newcommand{\sA}{\mc{A}}
\newcommand{\sB}{\mc{B}}
\newcommand{\sC}{\mc{C}}

\newcommand{\sF}{\mc{F}}

\newcommand{\sJ}{\mc{J}}

\newcommand{\sM}{\mc{M}}
\newcommand{\sN}{\mc{N}}
\newcommand{\sO}{\mc{O}}

\newcommand{\sT}{\mc{T}}

\newcommand{\sX}{\mc{X}}

\newcommand{\ms}{\mathscr}

\newcommand{\sMi}{\mc{M}^{(i)}}
\newcommand{\sNi}{\mc{N}^{(i)}}
\newcommand{\sMj}{\mc{M}^{(j)}}
\newcommand{\sNj}{\mc{N}^{(j)}}

\newcommand{\msOi}{\ms{O}(U_i)}

\newcommand{\msOk}{\ms{O}(U_k)}
\newcommand{\msOn}{\ms{O}(U_n)}

\renewcommand{\a}{\overline{a}}

\newcommand{\bs}{\bm{s}}
\newcommand{\ba}{\bm{a}}
\newcommand{\bb}{\bm{b}}

\DeclareMathOperator{\Qcoh}{Qcoh}
\DeclareMathOperator{\coh}{coh}
\DeclareMathOperator{\rgr}{gr-\!}
\DeclareMathOperator{\lgr}{\!-gr}
\DeclareMathOperator{\lGr}{\!-Gr}
\DeclareMathOperator{\rGr}{Gr-\!}
\DeclareMathOperator{\rmod}{mod-\!}
\DeclareMathOperator{\lmod}{\!-mod}
\DeclareMathOperator{\lMod}{\!-Mod}
\DeclareMathOperator{\rMod}{Mod-\!}

\DeclareMathOperator{\rQgr}{Qgr-\!}
\DeclareMathOperator{\rqgr}{qgr-\!}

\DeclareMathOperator{\rTors}{Tors-\!}

\DeclareMathOperator{\rtors}{tors-\!}

\DeclareMathOperator{\lMOD}{\!-\mf{Mod}}
\DeclareMathOperator{\rMOD}{\mf{Mod}-\!}

\DeclareMathOperator{\GKdim}{GKdim}
\DMO{\wt}{wt}
\DMO{\gldim}{gldim}
\DeclareMathOperator{\Xyz}{Xyz}
\DeclareMathOperator{\xyz}{xyz}

\newcommand{\uHom}{{\underline{\Hom}}}
\newcommand{\uExt}{{\underline{\Ext}}}

\newcommand{\RHom}{{\mathbb R}\!\Hom}

\newcommand{\hra}{\hookrightarrow}
\newcommand{\ssm}{\smallsetminus}

\newcommand{\tikzAngleOfLine}{\tikz@AngleOfLine}
\def\tikz@AngleOfLine(#1)(#2)#3{%
\pgfmathanglebetweenpoints{%
\pgfpointanchor{#1}{center}}{%
\pgfpointanchor{#2}{center}}
\pgfmathsetmacro{#3}{\pgfmathresult}%
}

\newcommand{\X}{\mf X}

\newcommand{\T}{T}
\newcommand{\Tpres}{\T_{\text{\normalfont{pres}}}}
\newcommand{\Ttoric}{\T_{\text{\normalfont{toric}}}}

\newcommand{\bGamma}{\boldsymbol{\Gamma}}
\newcommand{\bpi}{\boldsymbol{\pi}}

\renewcommand{\H}{\mathrm{H}}

\newcommand{\RomanOne}{\text{\normalfont{I}}}

\newcommand{\RomanThree}{\text{\normalfont{I\hspace{-1.1pt}I\hspace{-1.1pt}I}}}

\newcommand{\qb}{{q^\bullet}}

\DMO{\supp}{supp}
\renewcommand{\geq}{\geqslant}
\renewcommand{\leq}{\leqslant}
\DMO{\Rep}{Rep}
\DeclareMathOperator{\SL}{SL}

\makeatletter
\patchcmd{\@setaddresses}{\indent}{\noindent}{}{}
\patchcmd{\@setaddresses}{\indent}{\noindent}{}{}
\patchcmd{\@setaddresses}{\indent}{\noindent}{}{}
\patchcmd{\@setaddresses}{\indent}{\noindent}{}{}
\makeatother

\newcommand{\opi}{\boldsymbol{\overline{\bpi}}}

\allowdisplaybreaks 

\title{Resolutions of Type $\mathbb{A}$ Quantum Surface Singularities}

\address{The University of Manchester, Alan Turing Building, Oxford Road, Manchester, M13 9PL, United Kingdom\vspace{-5pt}}
\address{Heilbronn Institute for Mathematical Research, Bristol, United Kingdom}
\email{simon.crawford@manchester.ac.uk\vspace{5pt}}

\address{School of Mathematics, University of Edinburgh, Edinburgh EH9 3FD, United Kingdom}
\email{s.sierra@ed.ac.uk}

\author{Simon Crawford and Susan J. Sierra}
\date{\today}

\begin{document}

\begin{abstract}
Let $B = \kk_q[u,v]^{C_{n+1}}$ be a Type $\mathbb{A}_n$ quantum Kleinian singularity, which is an example of a noncommutative surface singularity. 
This singularity is known to have a noncommutative quasi-crepant resolution $\Lambda$, which is an ``algebraic'' resolution of $B$. We construct a category $\sX$ which serves as a ``geometric'' resolution of $B$ by adapting techniques from quiver GIT and show that $\sX$ and $\rmod \Lambda$ are derived equivalent.
Furthermore, we show that the intersection arrangement of lines in the exceptional locus of $\sX$ corresponds to a Type $\mathbb{A}_n$ Dynkin diagram.
This generalises the geometric McKay correspondence for classical Kleinian singularities. 
\end{abstract}

\maketitle


\setcounter{tocdepth}{1}
\let\oldtocsection=\tocsection
\let\oldtocsubsection=\tocsubsection
\let\oldtocsubsubsection=\tocsubsubsection
\renewcommand{\tocsection}[2]{\hspace{0em}\oldtocsection{#1}{#2}}
\renewcommand{\tocsubsection}[2]{\hspace{1.5em}\oldtocsubsection{#1}{#2}}
\renewcommand{\tocsubsubsection}[2]{\hspace{3em}\oldtocsubsubsection{#1}{#2}}

\tableofcontents


\numberwithin{equation}{section}

\section{Introduction}\label{INTRO}

Throughout, let $\Bbbk$ be an algebraically closed field of characteristic $0$. 
Let $q \in \kk^\times$, and let 
\[ 
A_q = \kk_q[u,v] :=\frac{ \kk \ang{u,v}} {\langle vu - quv \rangle}
\]
be the $q$-quantum plane. 
The cyclic group $C_{n+1}$ of order $n+1$ generated by
\begin{align*}
g \coloneqq 
\begin{pmatrix}
\zeta & 0 \\ 0 & \zeta^{-1}
\end{pmatrix},
\end{align*}
where $\zeta \in \Bbbk$ is a primitive $(n+1)$th root of unity, acts on $A_q$ via  graded automorphisms.
Let ${B_q = A_q^{C_{n+1}}}$ be the invariant ring of this action.
This is a well-known {\em quantum Kleinian singularity}, as studied in \cite{crawfordthesis,ckwz1,ckwz2} and others.  
In particular, $B_q$ is singular, i.e.~it has infinite global dimension, and this singularity can be resolved, in an appropriate sense, by the smash product $\Lambda_q := A_q \hash G$.
Further, there is a noncommutative McKay correspondence linking the representation theory of $G$, of $B_q$, and of $\Lambda_q$; see \cite{ckwz1,ckwz2}.

In this paper we extend the McKay correspondence for $B_q$ by establishing a {\em (derived) geometric  McKay correspondence}.  
That is, we construct an $\NN$-graded ring $T_q$ using methods inspired by geometric invariant theory, and a further construct a category $\mc X_q$ from $T_q$ by generalising standard techniques of noncommutative geometry. 
We show that $\mc X_q$  is also a resolution of (the module category of) $B_q$, again in an appropriate sense and, further, that $\mc X_q$ and $\rmod \Lambda_q$ are derived equivalent. This establishes an analogue of the derived geometric 
McKay correspondence, as given in \cite{BKR}, for the quantum Kleinian singularity $B_q$.  
Thus we have quantised the full classical McKay correspondence for type $\AA$ Kleinian surface singularities.

In order to describe our results in more detail, we first review the classical (commutative) McKay correspondence for Kleinian surface singularities. 
The finite subgroups of $\SL_2(\kk)$ were famously classified by Klein in the case $\kk = \CC$.
Any such subgroup $G$ acts on $\AA^2$ by construction, and the quotient singularity
\[ X = \AA^2/G := \Spec \kk[u,v]^G\]
is known as a {\em Kleinian (surface) singularity}.
These were studied by Du Val \cite{DuVal}, Artin \cite{Artin66}, and many others from a geometric perspective, but in 1981 McKay \cite{McKay} established a beautiful relationship between the geometry of $X$ and its minimal resolution and the representation theory of $G$, later extended and deepened by many others. We shall now give an ahistorical survey of the McKay correspondence for Kleinian singularities.

A finite subgroup $G$ of $\SL_2(\kk)$ acts naturally on $V = \kk^2$ and thus on $A := S(V) = \kk[u,v]$.
Set $B = A^G$, and let $X = \AA^2/G = \Spec B$ be the associated Kleinian singularity.
Let $W_0, \dots, W_n$ be the irreducible representations of $G$, up to isomorphism, where $W_0 = \kk$ is the trivial representation.
As $\Bbbk$ has characteristic $0$, there are non-negative integers $m_{ij} $ so that 
\[ V\otimes W_j \cong \bigoplus_i W_i^{m_{ij}}.\]
The {\em McKay quiver} $Q$ associated to the action of $G$ on $V$ has vertex set $Q_0 = \{ 0, \dots, n\}$ and $m_{ij}$ arrows from vertex $i$ to vertex $j$. (Since $G \leqslant \SL_2(\Bbbk)$, we in fact have $m_{ij}=m_{ji}$ \cite[Lemma 6.25]{LW}.)
Auslander had observed \cite{aus} that 
$A\hash G \cong \End_{B}(A)$, and Reiten and Van den Bergh later showed (see in particular the proof of \cite[Lemma 2.13]{RV}) that $A\hash G $ is Morita equivalent to the {\em preprojective algebra} $\Pi(Q)$ of the McKay quiver, which is the path algebra of $Q$ factored by certain natural relations; see \cite{cbh} for the definition. 
In fact, if $G$ is cyclic, then $A\hash G \cong \Pi(Q)$.
There are many  relationships between the representation theory of $G$, of $\Lambda:=\Pi(Q)$, and of $A\hash G$; see \cite[Chapter~5]{LW} for a discussion.

In 1981, McKay \cite{McKay} observed that the McKay quiver can also be obtained from the geometry of $X$.
Let $\phi: \widetilde{X} \to X$ be the minimal resolution of $X$, which exists and is unique for surface singularities. 
The exceptional divisor of $\phi$ is a union of irreducible curves, each of which is a $\PP^1$ of self-intersection $-2$, and it was known that there are $n$ of these (recall that $G$ has $n+1$ irreducible representations).
McKay showed:
\begin{theorem}\label{ithm:mckay}
{\rm (\cite{McKay})}
Let $L_1, \dots, L_n$ be the irreducible components of the exceptional locus of $\phi$.
Form the dual graph $\Gamma$ of the exceptional divisor by replacing each curve by a vertex and drawing  $\#(L_i \cap L_j)$ edges  between the  vertices corresponding to $L_i$ and $L_j$ if $i \neq j$.
Then $\Gamma$ is a Dynkin diagram of type $\AA_n$, $\DD_n$, or $\EE_n$, and is equal to the graph obtained from the McKay quiver of $G$ by removing the vertex $0$ and replacing each pair of opposing arrows by an edge.
\end{theorem}

McKay's result is essentially combinatorial, but it was later extended, originally by Kapranov and Vasserot, to a derived correspondence between $\widetilde{X}$ and $\Lambda$:
\begin{theorem}\label{ithm:derivedmckay}
{\rm (\cite{KV})}
Let $\widetilde{X}$ and $\Lambda$ be as above.
The categories $\coh \widetilde{X}$ and $\rmod \Lambda$ are derived equivalent.
\end{theorem}

It is well-known that the resolution $\phi: \widetilde{X} \to X$ is {\em crepant}:  that is, $\phi^* \omega_X \cong \omega_{\widetilde{X}}$, or equivalently the relative dualising complex $\phi^{!} \sO_X$ is trivial.
Further, the algebra $A\hash G$ is Van den Bergh's original example \cite[Example~1.1]{VdBNCCRoriginal} of a {\em noncommutative crepant resolution}.  
We do not give Van den Bergh's full definition here, which relies on the assumption that the singularity being resolved is commutative. 
Instead we note that $A \hash G$ is 2-Calabi-Yau (see Remark~\ref{rem:CY2}) and $A^G$ is Gorenstein.  
The dualising complexes of $A \hash G$ and of $A^G$ are both trivial and so one expects the ``relative dualising complex'' to be trivial: in other words, it makes sense to regard $A \hash G$ as a crepant resolution of $A^G$.
See \cite{VdBNCCR} for more discussion.

We briefly mention one more aspect of the commutative McKay correspondence: it is possible to construct the minimal resolution $\widetilde{X}$ of $X $ from the preprojective algebra $\Lambda$, as in \cite{BKR}.
This is done by using quiver GIT to construct a graded ring $T$ so that $\widetilde{X}= \Proj T$.
See the beginning of Section~\ref{REPRING} for details.

We now discuss how to quantise the above.  
The appropriate analogue of a finite subgroup of $\SL_2(\kk)$ in the noncommutative context is a {\em  quantum binary polyhedral group:} a pair $(H,A)$ where  $A$ is a two-dimensional Artin-Schelter regular algebra and $H$ is a semisimple finite-dimensional Hopf algebra, acting on $A$ inner faithfully, homogeneously, and with trivial homological determinant.  
These pairs were classified in \cite{ckwz0}; the cyclic group $C_{n+1}$ acting on the quantum plane $A_q $ as above is an example.  
If $(H, A)$ is a quantum binary polyhedral group, then it makes sense to call the invariant ring $A^H$ a {\em quantum Kleinian singularity}.

It is known that many aspects of the classical McKay correspondence extend to quantum Kleinian singularities.
Let $(H, A)$ be a quantum binary polyhedral group.
By \cite[Theorem~0.3]{ckwz1} the noncommutative version of \emph{Auslander's theorem} holds: the natural map
\[ \gamma: A \hash  H \to \End_{A^H} (A), \quad \gamma(a \hash h)(b) = a (h \cdot b) \]
is an isomorphism of rings. It makes sense to view the nonsingular
ring  $A \hash  H$ (or, equivalently, $\End_{A^G} (A)$) as a resolution of singularities of $A^G$; this is called a {\em noncommutative quasiresolution} in \cite{nqr}. 

Furthermore,  one may  construct the McKay quiver $Q$ of the pair $(H, A)$ (see \cite{ckwz2}),
and $A \hash H$ is Morita equivalent to a quotient of $\Lambda$ of $\kk Q$ \cite{crawfordsuperpotentials}.
The relationships between the representation theory of $H$ and of the preprojective algebra in the commutative context carry over to our situation; see \cite{ckwz2}.

In the case of the pair $(C_{n+1}, A_q)$, the McKay quiver $Q$ is the quiver corresponding to an affine $\AA_n$ Dynkin diagram. 
Further, $A_q \hash {C_{n+1}} $ is itself a basic algebra and may be viewed as a quotient of $\kk Q$.
We set $\Lambda_q := A_q \hash  C_{n+1}$.
We refer to this as an {\em algebraic resolution} of $B_q = A_q^{C_{n+1}}$, as it clearly parallels the (algebraic) NCCR of a commutative Kleinian singularity.
We show in Section~\ref{OURSINGULARITY} that $\Lambda_q$ is also a crepant resolution of $B_q$, suitably defined.

In the commutative case the Kleinian singularity $X = \Spec(A^G)$ has two resolutions:  the algebraic resolution $\Lambda$ and the geometric resolution $\widetilde{X}$.
In this paper, for the case that $(H, A) = ( C_{n+1}, A_q)$, we use techniques of noncommutative algebraic geometry to construct a category $\sX_q$ which gives a second resolution of $B_q = A_q^{C_{n+1}}$.  
That is, there is a pair of adjoint functors
\begin{equation}
\begin{tikzcd}[>=stealth,ampersand replacement=\&,row sep=35pt,every label/.append style={font=\normalsize}] 
\sX_q \arrow[d, "\,\phi_*" right,shift left=2] \\
\rMod B_q \arrow[u, "\phi^*\," left,shift left=2]
\end{tikzcd} \label{adjpair}
\end{equation}
(These functors are the appropriate noncommutative analogue of a morphism of schemes; see \cite{smith} and Section~\ref{MAPS}.)
We show:
\begin{theorem}\label{ithm:main1}
 The category $\sX_q$ is a resolution of $\rmod B_q$.
 That is: 
\begin{enumerate}[{\normalfont (1)},topsep=0pt,itemsep=0pt,leftmargin=*]
\item \emph{(Theorem~\ref{newthm:smooth})} $\sX_q$ has finite homological dimension;
\item \emph{(Theorem~\ref{thm:SF}(1))} $\sX_q $ is proper over $\rmod B_q$  in the sense that $\phi_*$ takes noetherian objects to noetherian objects; and
\item \emph{(Proposition~\ref{prop:june})} away from the singularity of $\rmod B_q$ the functors \eqref{adjpair} become inverse equivalences.
\end{enumerate}
\end{theorem}

We also show that $\sX_q$ contains exceptional objects $L_1, \dots, L_n$ whose noncommutative intersection theory gives an $\AA_n$ Dynkin diagram, extending the classical McKay correspondence  Theorem~\ref{ithm:mckay}.  See Theorem~\ref{thm:inttheory}.

The construction of $\sX_q$ comes from a graded ring $T_q$ which generalises (in fact deforms) the construction of $\widetilde{X}$ via quiver GIT.  
In particular, we also deform the representation variety of $\Lambda$ to construct a ``representation ring'' of the algebraic resolution $\Lambda_q$, which we show (Theorem~\ref{thm:universal}) is universal for a certain class of representations of $\Lambda_q$.
However, to prove most of Theorem~\ref{ithm:main1} we must approach $\sX_q$ in another way: we show that $\sX_q$ may also be viewed as a noncommutative toric variety (Theorem~\ref{newthm:catequiv}).  
This allows us to access $\sX_q$ via charts and local techniques and in particular is needed to prove Theorem~\ref{ithm:main1}(1): that $\sX_q$ is nonsingular.

Our other main theorem is an extension of Theorem~\ref{ithm:derivedmckay} to show that the geometric and algebraic resolutions are derived equivalent.

\begin{theorem} \label{ithm:main2}
{\rm (Theorem~\ref{thm:tilting})}
The derived categories of $\sX_q$ and $\rmod \Lambda_q$ are equivalent.
\end{theorem}

When  $q=1$, the proof of Theorem~\ref{ithm:main2} recovers the proof of Theorem~\ref{ithm:derivedmckay} given in \cite{NCCR}.

Since $\Lambda_q$ is a crepant resolution of $B_q$ in an appropriate sense, as discussed above, and $\sX_q$ and $\Lambda_q$ are derived equivalent, one expects that $\sX_q$ is also a crepant resolution of $B_q$.  We prove this at the end of Section~\ref{TILTING}.

In \cite{kawamata2024}, Kawamata considers a closely related family of noncommutative singularities, namely the deformations of type $\mathbb{A}$ Kleinian singularities defined in \cite{cbh}, and constructs noncommutative schemes which serve as resolutions of these singularities.  See also \cite{kawamata2025}. We note that the singularities Kawamata considers, as well as the techniques employed, are distinct from those in this paper.

By adapting the methods in this paper, it should be possible to construct an analogue of the category $\sX_q$ for other quantum Kleinian singularities $A^H$; in work in progress, we consider the case of an $\mathbb{L}_1$ singularity. \\

\noindent{\bf Organisation of the paper.} In Section \ref{PRELIM}, we give some definitions and basic results, and recall the notion of a noncommutative crepant resolution and a generalisation to noncommutative singularities. Section \ref{OURSINGULARITY} introduces the singularities $B_q$ of interest, constructs algebraic and categorical resolutions of these singularities, and discusses their crepancy.

The remainder of the paper is dedicated to constructing and studying a geometric resolution of $B_q$.
Motivated by quiver GIT, in Section \ref{REPRING} we define the {\em representation ring} of $\Lambda_q$ and show that it is universal in an appropriate sense. 
We then use this in Section \ref{GEOMRESOLUTION} to construct the ring $T_q$ which gives rise to the geometric resolution, and outline some of its properties. 
In Section \ref{QGR T}, we define the category $\sX_q$ which is our geometric resolution of $B_q$ and calculate some related cohomology groups.
To establish further properties of $\sX_q$, in Section \ref{LOCAL} we show that $\sX_q$ is equivalent to a noncommutative toric variety, i.e.~the category of coherent sheaves over a noncommutative ringed space $\X$ with a suitable action of a quantum torus.
This allows us to, in Section \ref{DERIVEDEQUIV}, establish a derived equivalence between our algebraic and geometric resolutions.
Finally, in Section \ref{INTTHEORY} we compute the intersection arrangement of the lines in the ``exceptional locus'' in $\sX_q$.
\\

\noindent{\bf Acknowledgements.} 
The second named author is partially supported by the EPSRC grant EP/T018844/1, 
and thanks the EPSRC for their support.

We thank the Simons Laufer Mathematical Sciences Institute for hosting us both in the spring of 2024 as part of  the program ``Noncommutative Algebraic Geometry''.  We thank Jenny August, Arend Bayer, Ellen Kirkman, Jon Pridham, Michael Wemyss, and James Zhang for helpful discussions.
We also thank Colin Ingalls for sharing his unpublished manuscript \cite{ingalls} on noncommutative toric varieties with us; it gave us helpful inspiration, although our construction of $\sX_q$ as a noncommutative toric variety uses a different method.

\section{Preliminaries}\label{PRELIM}

\numberwithin{equation}{subsection}

\subsection{Conventions}\label{CONVENTIONS}
Throughout $\Bbbk$ will denote an algebraically closed field of characteristic 0. Let $R$ be a $\ZZ$-graded ring. We write $\rGr R$ (respectively, $R\lGr$) for the category of $\ZZ$-graded right (respectively, left) $R$-modules with degree-preserving morphisms, and $\rgr R$ (respectively, $R\lgr$) for the full subcategory of noetherian objects. More generally, if $\Xyz$ is an abelian category, then $\xyz$ is the full subcategory of noetherian objects. Given $M \in \rGr R$, we define $M(i)$ to be the graded module which is isomorphic to $M$ as an ungraded module, but which satisfies $M(i)_n = M_{i+n}$. 
The categories $\rMod R$ and $R \lMod$ are the ungraded right and left module categories of $R$.\\
\indent If $M,N \in \rGr R$, then we write $\Hom_{\rGr R}(M,N)$ for the space of graded (i.e.~degree-preserving) morphisms. 
By $\Hom_R(M,N)$ we mean $\Hom_{\rMod R}(M,N)$.
We have
\[\uHom_R(M, N) = \bigoplus_{s\in\ZZ} \Hom_{\rGr R}(M, N(s)).\]
If $M_R$ is finitely generated, then $\Hom_R(M, N) \cong \uHom_R(M, N)$.

We write $\otimes$ for $\otimes_\kk$.
The set of non-negative integers is denoted by $\NN$.

\subsection{Definitions and basic results}
We begin by recalling some definitions and basic results that we will use without mention throughout the paper.

\begin{defn}\label{def2.2.1}
Suppose that $R$ is a $\Bbbk$-algebra. We say that $R$ is \emph{connected graded} if it is $\NN$-graded with $R_0 = \Bbbk$. We say that $M \in \rGr R$ is \emph{locally finite} if $\dim_\Bbbk M_i < \infty$ for all $i \in \mathbb{Z}$. The graded module $M$ is called \emph{left bounded} if $M_i = 0 $ for $i \ll 0$, \emph{right bounded} if $M_i = 0$ for $i \gg 0$, and \emph{bounded} if it is left and right bounded. If $M \in \rGr R$ is locally finite,  the \emph{Hilbert series} of $M$ is the formal Laurent series
\begin{align*}
\hilb M = \sum_{i \in \ZZ} (\dim_\Bbbk M_i) t^i.
\end{align*}
\end{defn}

The following notion provides a suitable concept of  dimension for a noncommutative $\Bbbk$-algebra:

\begin{defn}
Suppose that $R$ is a $\Bbbk$-algebra and $M$ is a right $R$-module. 
The \emph{Gelfand-Kirillov (GK-) dimension} of $M$ is
\begin{align*}
\GKdim M \coloneqq \sup_{M_0,V} \limsup_{n \to \infty} \log_n( \dim_\Bbbk M_0 V^n)
\end{align*}
where the supremum is taken over all finite-dimensional subspaces $M_0$ of $M$ and  $V$ of $A$. 

We will freely use standard results about GK-dimension, proofs of which can be found in \cite{lenagan}, for example.
If $R$ is a connected graded $\Bbbk$-algebra which is locally finite and $M \in \rgr A$ (which forces $M$ to be left bounded) then, by \cite[Proposition 6.6]{lenagan}, we have
\begin{align*}
\GKdim M = \limsup_{n \to \infty} \log_n(\dim_\Bbbk M_n) + 1.
\end{align*}
\end{defn}

\begin{defn}
Let $q \in \Bbbk^\times$. The \emph{quantum plane} is the ring
\begin{align*}
A_q = \Bbbk_q[u,v] \coloneqq \frac{\Bbbk \langle u,v \rangle}{\langle vu-quv \rangle}.
\end{align*}
This ring is a connected graded noetherian domain. 
We will sometimes suppress the $q$ and let $A = A_q$.
Throughout this paper, we will reserve the letter $A$ for the quantum plane as presented above.

 More generally, if $\mathbf{q} = (q_{ij}) \in M_n(\Bbbk)$ is a multiplicatively antisymmetric
 matrix, then to this data we can associate a \emph{quantum polynomial ring}, or \emph{quantum $n$-space}, as follows:
\begin{align*}
\Bbbk_\mathbf{q}[x_1, \dots, x_n] \coloneqq \frac{\Bbbk \langle x_1, \dots, x_n \rangle}{\langle x_j x_i - q_{ij} x_i x_j \mid 1 \leqslant i,j \leqslant n \rangle}.
\end{align*}
As before, this ring is always a connected graded noetherian domain.
\end{defn}

We will frequently work with rings which are (factor rings of) quantum polynomial rings. When performing computations in such rings, the following notation will be convenient: if $q \in \kk^\times$, we will use $q^\bullet$ to mean an unspecified power of $q$. For example, in 
$\kk_q[u,v]$, the monomials $u^i$ and $v^j$ $q^\bullet$-commute. Moreover, we will sometimes write $\Bbbk_\qb[x_1, \dots, x_n]$ for a quantum polynomial ring where the generators $x_i, x_j$ commute up to some power of $q$ (possibly depending on $i$ and $j$) which can be calculated in principle, but whose specific value does not affect the result.

We will frequently need to invert elements of noetherian domains; the next definition mainly establishes notation.

\begin{defn}
Suppose that $R$ is a noetherian domain. Then the set $X = R \ssm \{0\}$ forms a left and right Ore set in $R$, so we can form the \emph{quotient ring} $Q(R) \coloneqq RX^{-1} = X^{-1} R$. If, in addition, $R$ is graded by some semigroup $\Sigma$, then the set $Y$ of nonzero $\Sigma$-homogeneous elements of $R$ forms a left and right Ore set in $R$ \cite[Theorem 5(a)]{goodearlstafford}, so we can form the \emph{graded quotient ring} $\Qgr[\Sigma](R) \coloneqq RY^{-1} = Y^{-1} R$. When the grading is clear from context, we will simply write $\Qgr(R)$.
\end{defn}

We remark that the hypotheses imposed on $R$ in the above definition are stronger than required, but all rings to which we need to apply this construction will be noetherian domains. \\
\indent We will be interested in rings which have various good homological properties. We recall the relevant definitions below:

\begin{defn}
Let $R$ be a ring and $M$ a right $R$-module. The \emph{grade} of $M$ is defined to be
\begin{align*}
j(M) \coloneqq \inf \{ i \mid \Ext_R^i(M,R) \neq 0\} \in \NN \cup \{\infty\},
\end{align*}
with the definition for left modules being similar. We say that $M$ satisfies the \emph{Auslander condition} if for all $i \geqslant 0$ and all left submodules $N \subseteq \Ext^i_R(M,R)$ we have $j(N) \geqslant i$.

We say that $R$ is \emph{Gorenstein} if $R$ has finite injective dimension when viewed as a right or left $R$-module, and that that $R$ is \emph{Auslander--Gorenstein} if it is Gorenstein and every right or left $R$-module satisfies the Auslander condition. 
If $R$ is Auslander-Gorenstein and has finite global dimension, it is {\em Auslander regular}.
\end{defn}

In many cases, the properties of rings of infinite global dimension are encoded by properties of their so-called maximal Cohen--Macaulay modules, whose definition we now recall:

\begin{defn}
Let $R$ be a Gorenstein ring. We say that a finitely generated $R$-module $M$ is \emph{maximal Cohen--Macaulay (MCM)} if it satisfies $\Ext^i_R(M,R) = 0$ for all $i \geqslant 1$.
\end{defn}

In \cite{ckwz2}, the authors give a more general definition which specialises to the above definition when $R$ is Gorenstein.

We will assume that the reader is familiar with the language of sheaves and derived categories; \cite{Ha} and \cite{Weibel} are standard references for the definitions and results we will use. We will, however, set some notation. 
If $\sC$ is an abelian category, we will write $D(\sC)$ and $D^b(\sC)$ for its derived and bounded derived categories, respectively. 
If  $R$ is a noetherian ring, we write $D(R)$ for $D(\rMod R)$ and $\Perf(R)$ for the full subcategory  of perfect complexes. We  write $\mathscr{D}(R)$ for $D^b(\rmod R)$, following \cite{VdBNCCR}.

\subsection{Noncommutative crepant resolutions}\label{NCCR}
Suppose that $X$ is a normal variety with resolution of singularities $\pi : \widetilde{X} \to X$. This resolution is said to be \emph{crepant} if the canonical sheaves $\omega_X$ and $\omega_{\widetilde{X}}$ satisfy $\pi^* \omega_X = \omega_{\widetilde{X}}$. In some sense, crepant resolutions are the ``best'' nonsingular approximation of a singular variety. In particular, if $X$ is Gorenstein and we want $\widetilde{X}$ to also be Calabi--Yau, then the resolution must be crepant. 

In \cite[Theorem 4.4.2]{vdb2004}, Van den Bergh proved the following result, which we paraphrase:

\begin{theorem}\label{thm:classical}
Let $R$ be a (commutative)  Gorenstein $\Bbbk$-algebra which is a normal domain, 
and let $\pi : \widetilde{X} \to X = \Spec R$ be a crepant resolution of singularities. Assume that the fibres of $\pi$ have dimension at most $1$. Then there exists an MCM $R$-module $M$ such that $\Lambda \coloneqq \End_R(M)$ is MCM as an $R$-module and has finite global dimension. Moreover, $\widetilde{X}$ and $\Lambda$ are derived equivalent.
\end{theorem}

This motivates the following definition, which is equivalent to the original one due to Van den Bergh:

\begin{defn}\label{def:NCCR}
Let $R$ be a commutative CM  normal  domain. A \emph{noncommutative crepant resolution (NCCR)} of $R$ (or of $\Spec R$) is a ring of the form $\Lambda \coloneqq \End_R(M)$ for some reflexive $R$-module $M$, such that $\Lambda$ has finite global dimension and is an MCM $R$-module.
\end{defn}

For example, by \cite[Example~1.1]{VdBNCCRoriginal}, if $G$ is a finite subgroup of $\SL_2(\kk)$ then $\kk[x,y]\hash G = \End_{\kk[x,y]^G}(\kk[x,y])$ is an NCCR of the Kleinian singularity $\kk[x,y]^G$.
If $R$ is a CM local normal domain of dimension $\leqslant 3$, then it is known that all crepant resolutions and NCCRs of $R$ are derived equivalent \cite[Theorem 1.5]{iyamawemyss}.

If we replace $R$ by a noncommutative ring, then there are several ways to generalise the definition of an NCCR.  
One such way is the \emph{noncommutative quasi-crepant resolution} defined in \cite{nqr}.  
We begin by recalling some more definitions.

\begin{defn}
We say that a $\Bbbk$-algebra $R$ is \emph{GK-Cohen--Macaulay} if
\begin{align*}
\GKdim(M) + j(M) = \GKdim(R) < \infty,
\end{align*}
for all finitely generated nonzero right $R$-modules $M$.
\end{defn}

If $R$ is commutative and noetherian, then \cite[Theorem 4.8]{brown} shows that $R$ is GK-Cohen--Macaulay if and only if it is Cohen--Macaulay and equicodimensional (i.e.~all of its maximal ideals have the same height).

\begin{defn} \label{def:nisomorphic}
Suppose that $R$ and $S$ are locally finite noetherian $\NN$-graded $\Bbbk$-algebras with finite GK-dimension. We say that two graded $(R,S)$-bimodules $M$ and $N$ are \emph{$n$-isomorphic}, denoted $M \cong_n N$, if there exist a graded $(R,S)$-bimodule $P$ and bimodule morphisms
\begin{align*}
f : M \to P \quad \text{and} \quad g : N \to P
\end{align*}
such that the kernels and cokernels of $f$ and $g$ have GK-dimension at most $n$.
\end{defn}

We are now in a position to define the notion of a \emph{noncommutative quasi-crepant resolution}, as given in \cite{nqr}.

\begin{defn}
Suppose that $R$ is a locally finite noetherian $\NN$-graded $\Bbbk$-algebra and that $\GKdim R = d < \infty$. A \emph{noncommutative quasi-resolution (NQR)} of $R$ is a triple $(S,M,N)$ where $S \in \mathcal{A}$ is Auslander regular and GK-Cohen--Macaulay with $\GKdim S = d$, and ${}_S M_R$ and ${}_R N_S$ are bimodules which satisfy
\begin{align*}
M \otimes_R N \cong_0 S \quad \text{and} \quad N \otimes_S M \cong_0 R.
\end{align*}
If, in addition, $R$ is Auslander--Gorenstein and ${}_S M_R$ and ${}_R N_S$ are reflexive on both sides, then we call this triple a \emph{noncommutative quasi-crepant resolution (NQCR)} of $R$.
\end{defn}

By \cite[Example 8.5]{nqr}, if $S$ is a (commutative) Type $\mathbb{A}_n$ Kleinian singularity, then an NCCR of $S$ is also an NQR.

There is also a notion of a {\em categorical resolution}, due to Kuznetsov \cite{kuznetsov}, which may be weakly or strongly crepant.  This is discussed in Section~\ref{CATRES}.

\subsection{Maps of noncommutative spaces}\label{MAPS}
It will be useful to be able to use the language of maps of noncommutative spaces to discuss functors between categories.
Following \cite{smith}, let $\mathbf{A}, \mathbf{B}$ be (possibly triangulated) categories.
A {\em map} $\Phi: \mathbf{A} \to \mathbf{B}$ is a pair of (triangulated) functors
\begin{equation*}
\begin{tikzcd}[>=stealth,ampersand replacement=\&,row sep=35pt,every label/.append style={font=\normalsize}] 
\mathbf{A} \arrow[d, "\,\,\Phi_*" right,shift left=2] \\
\mathbf{B} \arrow[u, "\Phi^*\,\," left,shift left=2]
\end{tikzcd}
\end{equation*}
so that $(\Phi^*, \Phi_*)$ is an adjoint pair.
 This is by analogy with the fact that, if $\pi: Y\to X$ is a morphism of  schemes, then there is an adjoint pair
\begin{equation*}
\begin{tikzcd}[>=stealth,ampersand replacement=\&,row sep=35pt,every label/.append style={font=\normalsize}] 
\Qcoh Y \arrow[d, "\,\pi_*" right,shift left=2] \\
\Qcoh X \arrow[u, "\pi^*\," left,shift left=2]
\end{tikzcd}
\end{equation*}

For example, if $R, S$ are rings, then an $(R, S)$-bimodule $M$ induces a map $\Phi:  \rMod S \to \rMod R$ given by
\begin{equation*}
\begin{tikzcd}[>=stealth,ampersand replacement=\&,row sep=35pt,every label/.append style={font=\normalsize}] 
\rMod S \arrow[d, "\,\,\Phi_*= \Hom_S(M{,}-)" right,shift left=2] \\
\rMod R \arrow[u, "\Phi^* = - \otimes_R M\,\," left,shift left=2]
\end{tikzcd}
\end{equation*}

\section{A noncommutative singularity} \label{OURSINGULARITY}

\subsection{The singularity of interest} \label{sec:OurSingularity}

We now introduce the main protagonist of this paper.

\begin{notation} \label{notation:setup}
Fix $n \geqslant 2$ and $q \in \Bbbk^\times$. 
The cyclic group $G$ of order $n+1$ generated by
\begin{align*}
g \coloneqq 
\begin{pmatrix}
\zeta & 0 \\ 0 & \zeta^{-1}
\end{pmatrix},
\end{align*}
where $\zeta \in \Bbbk$ is a primitive $(n+1)$th root of unity, acts on $A_q = \kk_q[u,v]$ as a group of graded automorphisms. Explicitly, we have
\begin{align*}
g \cdot u = \zeta u, \quad g \cdot v = \zeta^{-1} v.
\end{align*}

To this data we can associate the \emph{skew group ring} $A_q \hash G$. As an abelian group this ring is equal to $A_q \otimes \Bbbk G$, and we write $a \hash g$ for the simple tensor $a \otimes g$. On simple tensors, the multiplication in $A \hash G$ is given by
\begin{align*}
(a \hash g)(b \hash h) \coloneqq a (g \cdot b) \hash gh,
\end{align*}
which is then extended linearly to all of $A \hash G$.

We write $B_q \coloneqq A_q^G$ for the invariant ring of the action of $G$ on $A_q$. By \cite[Table 3]{ckwz1}, $B_q$ is generated by the elements
\begin{align*}
x \coloneqq u^{n+1}, \quad y \coloneqq v^{n+1}, \quad z \coloneqq uv,
\end{align*}
and can be presented as a quotient of a quantum polynomial ring by the ideal generated by a single homogeneous, normal, nonzerodivisor:
\begin{align}
B_q \cong \frac{\Bbbk_\qb[x,y,z]}{\big\langle xy - q^{-\frac{1}{2}n(n+1)} z^{n+1} \big\rangle},
\quad \text{where} \quad
\Bbbk_\qb[x,y,z] \coloneqq \frac{\Bbbk\langle x,y,z \rangle}{\left\langle\!
\begin{array}{c}
yx - q^{(n+1)^2} xy \\
zx - q^{n+1} xz \\
zy - q^{-(n+1)} yz
\end{array}\!
\right\rangle
}. \label{eqn:BPresentation}
\end{align}
Following \cite{ckwz1}, we call $B_q$ a \emph{quantum Kleinian singularity of type} $\mathbb{A}_n$ or a \emph{quantum type $\mathbb{A}$ singularity}. 
The singularity category of $B_q$ is that of a commutative type $\mathbb{A}$ singularity by \cite[Theorem 8.2.7]{crawfordthesis}.
\end{notation}

We now recall various properties of $A_q$, $B_q$, and related rings. We begin by noting that $B_q$, although singular, has good ring-theoretic and homological properties:

\begin{lemma}
The ring $B_q = A_q^G$:
\begin{enumerate}[{\normalfont (1)},topsep=0pt,itemsep=0pt,leftmargin=*]
\item is a noetherian domain;
\item is Auslander--Gorenstein of injective dimension $2$;
\item is GK-Cohen--Macaulay;
\item has GK-dimension $2$; and
\item has infinite global dimension.
\end{enumerate}
\end{lemma}
\begin{proof}
These are all established in \cite[Lemma 6.1.3, Proposition 6.1.6]{crawfordthesis}.
\end{proof}

\indent The ring $B_q$ has an important family of modules $M^{(i)}$, where $0 \leqslant i \leqslant n$, which we now define. We first define (left and right) $B_q$-modules $M^{i,G}$ for $0 \leqslant i \leqslant n$ as follows:
\begin{align*}
M^{i,G} \coloneqq \{ a \in A_q \mid g \cdot a = \zeta^{-i} a \}.
\end{align*}
As right $B_q$-modules,  these have the explicit form
\begin{align*}
M^{0,G} = B_q, \qquad M^{i,G} = u^{n+1-i} B_q + v^i B_q, \quad 1 \leqslant i \leqslant n.
\end{align*}
In terms of the isomorphism \eqref{eqn:BPresentation}, we can write these modules as
\begin{align*}
M^{(0)} = B_q, \qquad M^{(i)} = x B_q + z^i B_q, \quad 1 \leqslant i \leqslant n.
\end{align*}
Under the grading inherited from $A_q$, there are graded isomorphisms
\begin{align*}
M^{(i)} \cong M^{i,G}(i) \quad 1 \leqslant i \leqslant n,
\end{align*}
induced by multiplication by $u^i$. We are particularly interested in these modules because they are maximal Cohen--Macaulay and they form essentially a complete list of such modules:

\begin{proposition}{\rm {(\cite[Theorem C]{ckwz2}})}
The modules $M^{(i)}$ are indecomposable and maximal Cohen--Macaulay. Moreover, any indecomposable MCM $B_q$-module is isomorphic (in the ungraded category) to some $M^{(i)}$. 
\end{proposition}

We now determine the Hilbert series of $B_q$ and the $M^{(i)}$. Note that these Hilbert series are independent of $q$ so, in particular, we can restrict attention to the case $q=1$ (i.e.~when $A_q$ and $B_q$ are commutative). In this case, by \cite[Theorem 6.4.8]{brunsherzog} we have 
\begin{align*}
\hilb M^{i,G} = \frac{1}{n+1} \sum_{j=0}^n \frac{\zeta^{ij}}{(1-\zeta^j t)(1-\zeta^{-j}t)}.
\end{align*}
A tedious calculation, similar to \cite[Example 6.19]{taftactions}, gives a closed form for this Hilbert series:
\begin{align*}
\hilb M^{i,G} = \frac{t^i + t^{n+1-i}}{(1-t^2)(1-t^{n+1})},
\end{align*}
and hence
\begin{align}
\hilb M^{(i)} = \frac{t^{2i} + t^{n+1}}{(1-t^2)(1-t^{n+1})}. \label{eqn:hilbMi}
\end{align}

\begin{lemma} \label{lem:MiSES}
For $0 \leqslant i \leqslant n$, there is a short exact sequence of right $B_q$-modules given by
\begin{align*}
0 \to M^{(n+1-i)}(-2i) \xrightarrow{\begin{psmallmatrix} x^{-1} z^i \\ -1 \end{psmallmatrix} \cdot} B_q(-(n+1)) \oplus B_q(-2i) \xrightarrow{\begin{psmallmatrix} x & z^i \\[-4pt] & \end{psmallmatrix} \cdot} M^{(i)} \to 0,
\end{align*}
where the maps are given by left multiplication by the indicated matrices.
\end{lemma}
\begin{proof}
It is straightforward to check that the given maps are well-defined and that the sequence forms a complex. For exactness, it suffices to verify that the Euler characteristic of the sequence is $0$; indeed, by \eqref{eqn:hilbMi}, the Euler characteristic is
\begin{align*}
- t^{2i} \frac{t^{2(n+1-i)} + t^{n+1}}{(1-t^2)(1-t^{n+1})} + (t^{n+1}+t^{2i}) \frac{1+t^{n+1}}{(1-t^2)(1-t^{n+1})} - \frac{t^{2i} + t^{n+1}}{(1-t^2)(1-t^{n+1})} = 0,
\end{align*}
as desired.
\end{proof}

We will also require the following fact:

\begin{lemma} \label{lem:HomMi}
There is an isomorphism
\begin{align*}
\Hom_{B_q}(M^{(i)}, B_q) \cong M^{(n+1-i)}
\end{align*}
given by $\eta : M^{(n+1-i)} \to \Hom_{B_q}(M^{(i)}, B_q), \hspace{3pt} \eta(m)(m') = x^{-1} mm'$.
\end{lemma}
\begin{proof}
By the proof of \cite[Corollary 2.6]{ckwz2}, there is an equivalence of categories
\begin{align*}
\{ \mbox{graded projective $A_q \hash G$-modules} \} \to \rgr B_q, \quad N \mapsto N^G.
\end{align*}
Using this, and letting $V_i$ denote the $\Bbbk G$-module with character $\chi_i$ satisfying $\chi_i(g) = \zeta^i$, we obtain a chain of isomorphisms
\begin{align*}
\Hom_{B_q}(M^{i,G},B_q) &\cong \Hom_{A_q^G}\big( (V_i \otimes A_q)^G, A_q^G \big) \\
&\cong \Hom_{A_q \hash G}(V_i \otimes A_q, A_q) \\
&\cong \Hom_{A_q}(V_i \otimes A_q, A_q)^G \\
&\cong \Hom_{A_q}(A_q, V_i^* \otimes A_q)^G \\
&\cong \Hom_{A_q}(A_q, V_{n+1-i} \otimes A_q)^G \\
&\cong (V_{n+1-i} \otimes A_q)^G \\
&\cong M^{n+1-i,G}.
\end{align*}
Identifying $M^{i,G}$ with $M^{(i)}$ and chasing through the isomorphisms shows that the map giving rise to this isomorphism is as claimed.
\end{proof}

\subsection{An algebraic resolution of \texorpdfstring{$B$}{B}}
Write $M = \bigoplus_{i=0}^n M^{(i)}$ for the direct sum of the indecomposable MCM $B_q$-modules. We in fact have $M \cong A_q$ as objects of $\rmod B_q$; we obtain an isomorphism of \emph{graded} modules if we replace each $M^{(i)}$ by $M^{i,G}$. The endomorphism ring of this module will play an important role. By \cite[Theorem 0.3]{ckwz1}, the \emph{Auslander map}
\begin{align}
\gamma : A_q \hash G \to \End_{B_q}(A_q), \quad \gamma(a \hash g)(b) = a (g \cdot b) \label{eqn:AusMap}
\end{align}
is an isomorphism, and there is a similar isomorphism when we view $A_q$ as a left $B_q$-module. By \cite[Example 6.6]{crawfordsuperpotentials}, there is a ring isomorphism
\begin{align}
A_q \hash G \cong \Lambda_q \label{eqn:CKWZiso}
\end{align}
where $\Lambda_q$  is the quotient of the path algebra of the following quiver by the given relations:

\beq\label{quiverdiagram}
\begin{tikzpicture}[->,>=stealth,thick,scale=1,baseline=(current  bounding  box.center)]
\def\ra{2.2cm}
\node[regular polygon,regular polygon sides=6,minimum size=0.8cm] (0) at (4*360/6:  \ra) {\phantom{0}};
\node[regular polygon,regular polygon sides=6,minimum size=0.8cm]  (1) at (3*360/6:  \ra) {\phantom{0}};
\node[regular polygon,regular polygon sides=6,minimum size=0.8cm]  (2) at (2*360/6:  \ra) {\phantom{0}};
\node[regular polygon,regular polygon sides=6,minimum size=0.8cm]  (3) at (1*360/6:  \ra) {\phantom{0}};
\node[regular polygon,regular polygon sides=6,minimum size=0.8cm]  (4) at (0*360/6:  \ra) {\phantom{0}};
\node[regular polygon,regular polygon sides=6,minimum size=0.8cm]  (5) at (-1*360/6:  \ra) {\phantom{0}};

\node at (3*360/6:  \ra) {$n$};
\node at (2*360/6:  \ra) {$0$};
\node at (1*360/6:  \ra) {$1$};
\node at (0*360/6:  \ra) {$2$};

\draw (0.100) to (1.-40) ;
\draw (1.40) to (2.260) ;
\draw (2.-20) to (3.200);
\draw (3.-80) to (4.140) ;
\draw (4.220) to (5.80) ;

\draw (1.280) to (0.140);
\draw (2.220) to (1.80);
\draw (3.160) to (2.20) ;
\draw (4.100) to (3.-40);
\draw (5.40) to (4.260);

\draw[-,dash pattern={on 1pt off 2pt}] (0) to (5);

\node at (30+1*360/6: \ra - 0.78cm) {$\alpha_0$};
\node at (30+0*360/6: \ra - 0.8cm) {$\alpha_1$};
\node at (30+5*360/6: \ra - 0.8cm) {$\alpha_2$};
\node at (-1,-0.72) {$\alpha_{n{-}1}$};
\node at (30+2*360/6: \ra - 0.8cm) {$\alpha_{n}$};

\node at (30+1*360/6: \ra + 0.18cm) {$\overline{\alpha}_0$};
\node at (30+0*360/6: \ra + 0.2cm) {$\overline{\alpha}_1$};
\node at (30+5*360/6: \ra + 0.2cm) {$\overline{\alpha}_2$};
\node at (-2.2,-1.22) {$\overline{\alpha}_{n{-}1}$};
\node at (30+2*360/6: \ra + 0.2cm) {$\overline{\alpha}_n$};

\node at (30+4*360/6: \ra + 0.36cm) {\phantom{$\overline{\alpha}_0$}};

\node at (6.2,0) {$\overline{\alpha}_{i} \alpha_{i} = q \alpha_{i+1} \overline{\alpha}_{i+1} $ for $0 \leqslant i \leqslant n$.};

\end{tikzpicture}
\eeq

\noindent (Note that our convention is to compose arrows in quivers from left to right.)

Moreover, we have $e_0 \Lambda_q e_0 \cong B_q$ and, under this isomorphism, we can identify the right $B_q$-module $M^{i,G}$ with the right $e_0 \Lambda_q e_0$-module $e_i \Lambda_q e_0$ and identify $M \cong A_q$ with $\Lambda_q e_0$. 
Under this identification, we can rewrite the isomorphism of \eqref{eqn:AusMap} as
\begin{align}
\gamma' : \Lambda_q \to \End_{e_0 \Lambda_q e_0}(\Lambda_q e_0), \quad \gamma'(x)(y e_0) = xy e_0. \label{eqn:CKWZiso2}
\end{align}

The algebra $\Lambda_q$ is of particular interest to us because it is an NQCR of $B_q$.

\begin{lemma}\label{lem:NCQR}
The triple $(\Lambda_q,e_0 \Lambda_q, \Lambda_q e_0)$ is an NQCR of $B_q = e_0 \Lambda_q e_0$.
\end{lemma}
\begin{proof}
That $(\Lambda_q,e_0 \Lambda_q, \Lambda_q e_0)$ is an NQR of $B_q$ is \cite[Example 8.5]{nqr} (where they show that the isomorphic algebra $A_q \hash G$ is an NQR). Since $B_q$ is Auslander--Gorenstein, it remains to show that $e_0 \Lambda_q$ and $\Lambda_q e_0$ are reflexive on both sides; by symmetry, it suffices to show this for $\Lambda_q e_0$. \\
\indent Reflexivity of $\Lambda_q e_0$ as a left $\Lambda_q$-module follows from the fact that $\Lambda_q e_0$ is a direct summand of $\Lambda_q$, hence projective. 
To prove that $\Lambda_q e_0$ is reflexive as a right $e_0 \Lambda_q e_0$-module, we use \eqref{eqn:CKWZiso2}.
Multiplying this isomorphism on the left by $e_0$ gives $\Hom_{e_0 \Lambda_q e_0}(\Lambda_q e_0, e_0 \Lambda_q e_0) \cong e_0 \Lambda_q$, and similarly we have $\Hom_{e_0 \Lambda_q e_0}(e_0 \Lambda_q, e_0 \Lambda_q e_0) \cong \Lambda_q e_0$. Therefore
\begin{align*}
\Hom_{e_0 \Lambda_q e_0}(\Hom_{e_0 \Lambda_q e_0}(\Lambda_q e_0, e_0 \Lambda_q e_0), e_0 \Lambda_q e_0) \cong \Hom_{e_0 \Lambda_q e_0}(e_0 \Lambda_q, e_0 \Lambda_q e_0) \cong \Lambda_q e_0, 
\end{align*}
so that $\Lambda_q e_0$ is reflexive as a right $e_0 \Lambda_q e_0$-module.
\end{proof}

Another good property of  $\Lambda_q$ is that it is Morita equivalent to $B_q$ ``away from the singularity''. 
This term is imprecise for a noncommutative ring which has no associated geometry, but note that $B_q[x^{-1}], B_q[y^{-1}], B_q[z^{-1}]$ are all localisations of quantum planes and so have finite global dimension.  
On the other hand, $B_q$ has infinite global dimension, and by the previous sentence it makes some sense to view the singularity of $B$ as being concentrated where $x=y=z=0$.

\begin{lemma}\label{lem:MoritaEquiv}
Let $w \in \{ x,y,z\}$. Then the algebras $B_q[w^{-1}]$ and $\Lambda_q[w^{-1}]$ are Morita equivalent.
\end{lemma}
\begin{proof}
First let $w=x$. It is clear that $M^{(i)}[x^{-1}] = B[x^{-1}]$ for all $i$, and hence $M[x^{-1}] \cong B[x^{-1}]^{n+1}$. Therefore, combining \eqref{eqn:CKWZiso} with the fact that the Auslander map is an isomorphism and localising, we obtain
\begin{align*}
\Lambda_q[x^{-1}] = \End_{B_q}(M)[x^{-1}] \cong \End_{B_q[x^{-1}]}(M[x^{-1}]) \cong \End_{B_q[x^{-1}]}(B_q[x^{-1}]^{n+1}) \cong M_{n+1}(B_q[x^{-1}),
\end{align*}
and hence $B_q[x^{-1}]$ and $\Lambda_q[x^{-1}]$ are Morita equivalent.
The same proof will work to show that $B_q[y^{-1}]$ and $\Lambda_q[y^{-1}]$ are Morita equivalent, provided we can show that $M^{(i)}[y^{-1}] \cong B_q[y^{-1}]$ for $1 \leqslant i \leqslant n$ (the claim for $i=0$ is clear). Indeed, this follows from the fact that the maps 
\begin{align*}
\alpha_i : M^{(i)}[y^{-1}] \to B_q[y^{-1}], \quad \alpha_i(m) = z^{-1} y m
\end{align*}
are easily seen to be well-defined maps of right $B_q[y^{-1}]$-modules which are invertible. Finally, if $w=z$ then the relation $xy=z^{n+1}$ means that we have also inverted $x$ and $y$, and so the previous arguments apply.
\end{proof}

\begin{remark}\label{rem:CY2}
By \cite[Theorem 4.1(c)]{rrz}, $\Lambda_q$ is \emph{twisted 2-Calabi-Yau} (sometimes called \emph{skew 2-Calabi-Yau}). While we do not give a precise definition, in particular this says that $\Lambda_q$ is homologically smooth, in the sense that $\Lambda_q$ has a finite length projective resolution in the category of $\Lambda_q$-bimodules such that each term is finitely generated. 
By \cite[Lemma 4.3(b)]{yekutielizhang}, this implies that $\Lambda_q$ has finite global dimension, although not all rings of finite global dimension are homologically smooth. 
Since $\Lambda_q$ is twisted $2$-CY, it is reasonable to view $\rmod \Lambda_q$ as being ``smooth''. 
\end{remark}

\subsection{The algebraic resolution is categorical} \label{CATRES}
Let $\Lambda = \Lambda_q$ and $B = B_q$.
As in Section~\ref{MAPS}, the functors
\begin{equation*}
\begin{tikzcd}[>=stealth,ampersand replacement=\&,row sep=35pt,every label/.append style={font=\normalsize}] 
\rMod \Lambda \arrow[d, "\,\,\psi_* = \Hom_\Lambda(e_0 \Lambda{,}-)" right,shift left=2] \\
\rMod B \arrow[u, "\psi^* = - \otimes_B e_0 \Lambda\,\," left,shift left=2]
\end{tikzcd}
\end{equation*}
give a map of noncommutative spaces $\psi : \rMod \Lambda \to \rMod B$.

Since $\Lambda$ has finite global dimension, we think of the category $\rMod \Lambda$ as being nonsingular, and $\psi : \rMod \Lambda \to \rMod B$ may be viewed as a resolution of singularities, in a sense we will now make more precise. Let $\widetilde{D}$ be a triangulated category and suppose that there is a map of noncommutative spaces $\Psi: \widetilde{D} \to \ms D(B)$, that is, an adjoint pair of triangulated functors 
\begin{align*}
\Psi^* : \Perf(B) \to \widetilde{D}, \qquad \Psi_* : \widetilde{D} \to \mathscr{D}(B).
\end{align*}
Following \cite{kuznetsov}, $\widetilde{D}$ is a {\em categorical resolution} of $\ms D(B)$ if $\widetilde{D}$ is ``smooth'' and the natural morphism of functors $\id_{\Perf(B)} \to \Psi_* \Psi^*$ is an isomorphism.
The term ``smooth'' is somewhat open to interpretation, but it is reasonable to consider the derived category of an abelian category of a  homologically smooth algebra to be smooth. 
(Recall from Remark~\ref{rem:CY2} that $\Lambda$ is homologically smooth.)
A categorical resolution is said to be \emph{weakly crepant} if $(\Psi_*,\Psi^*)$ also form an adjoint pair. 
If $X$ is a variety with rational singularities then a resolution of singularities $\Psi:  \widetilde{X} \to X$ is crepant exactly when $D^b(\widetilde{X})$ is a weakly crepant categorical resolution of $D^b(X)$.

We will see that $\ms D(\Lambda)$ is  a weakly crepant categorical resolution of  $\mathscr{D}(B)$.  
We need some preliminary results first.

\begin{lemma}\label{lem:star}
Let $R$ be a right noetherian ring, let $\mc D $ be a triangulated category, and let $\alpha, \beta: \ms D(R) \to \mc D$ be triangulated functors. 
Suppose that
\begin{enumerate}[{\normalfont (a)},topsep=0pt,itemsep=0pt,leftmargin=*]
\item there is an isomorphism $\theta: \alpha(R) \to \beta(R)$; and
\item for all $r \in R$, the following diagram commutes:
\begin{equation*}
\begin{tikzcd}[>=stealth,ampersand replacement=\&,row sep=30pt,every label/.append style={font=\normalsize},column sep=45pt] 
\alpha(R) \arrow[r, "\alpha(r \cdot -)" {above}] \arrow[d, "\theta" {swap}] \& \alpha(R) \arrow[d, "\theta"] \\
\beta(R) \arrow[r, "\beta(r\cdot-)" {swap}] \& \beta(R)
\end{tikzcd}
\end{equation*}
\end{enumerate}
Then $\alpha\cong \beta$ as functors from $\Perf(R) \to \mc D$.
\end{lemma}
\begin{proof}
Consider the full subcategory of $\ms D(R)$ with one object, $R$.
The hypothesis is that $\alpha,\beta$ are naturally isomorphic as functors from this subcategory to $\mc D$.
The result follows since $\Perf(R)$ is the triangulated envelope of this subcategory.
\end{proof}

Consider now the  two other natural functors between $\rMod \Lambda$ and $\rMod B$, namely
\begin{equation*}
\begin{tikzcd}[>=stealth,ampersand replacement=\&,row sep=35pt,every label/.append style={font=\normalsize}] 
\rMod \Lambda \arrow[d, "\psi_! = - \otimes_\Lambda \Lambda e_0\,\," left,shift right=2] \\
\rMod B \arrow[u, "\,\,\psi^! = \Hom_B(\Lambda e_0{,}-)" right,shift right=2]
\end{tikzcd}
\end{equation*}
given by the isomorphism $\Lambda \cong \End_B(\Lambda e_0)$.
Note that  $(\psi_!,\psi^!)$ are also  an adjoint pair. 

\begin{lemma}\label{lem:shriek}
There is a natural isomorphism $\psi_* \cong \psi_!$, and the functors $\psi^*$ and $\psi^!$ are equivalent upon restriction to the full subcategory of projective $B$-modules.
\end{lemma}
\begin{proof}
We first consider the claimed natural isomorphism $\psi_* \cong \psi_!$. A straightforward adaptation of \cite[Proposition 20.11]{andersonfuller} shows that there is a natural isomorphism
\begin{align*}
\Hom_\Lambda(\Lambda,-) \otimes_\Lambda \Lambda e_0 \cong \Hom_\Lambda(\Hom_\Lambda(\Lambda e_0, \Lambda), -).
\end{align*}
However, the left hand functor is naturally isomorphic to $\psi_!$ and, since $\Hom_\Lambda(\Lambda e_0,\Lambda) \cong e_0 \Lambda$, the right hand functor is naturally isomorphic to $\psi_*$, giving the claimed result. Explicitly, the natural isomorphism $\eta: \psi_! \to \psi_*$ arising this way is given by, for $N \in \rMod \Lambda$,
\begin{align*}
\eta_N : N \otimes_\Lambda \Lambda e_0 \to \Hom_\Lambda(e_0 \Lambda, N), \quad \eta_N(n \otimes x e_0)(e_0 y) = nxe_0y.
\end{align*}
\indent For the second claim, it is straightforward to verify that $\theta: \psi^* \to \psi^!$ defined on $K \in \rMod B$ via
\begin{align*}
\theta_K : K \otimes_B e_0 \Lambda \to \Hom_B(\Lambda e_0, K), \quad \theta_K(k \otimes e_0 x)(y e_0) = k e_0 xy e_0,
\end{align*}
is a natural transformation. To show that this is an equivalence upon restricting to projective modules, by \cite[Lemma 20.9]{andersonfuller}, it suffices to check that $\psi^*$ and $\psi^!$ agree on $B = e_0 \Lambda e_0$. Indeed,
\begin{align*}
\psi^*(B) = B \otimes_B e_0 \Lambda \cong e_0 \Lambda \cong \Hom_B(\Lambda e_0, e_0 \Lambda e_0) = \Hom_B(\Lambda e_0, B) = \psi^!(B),
\end{align*}
where the isomorphism $e_0 \Lambda \cong \Hom_B(\Lambda e_0, e_0 \Lambda e_0)$ follows from \eqref{eqn:CKWZiso2}.
\end{proof}

\begin{proposition} \label{prop:PsiCrepantCategoricalRes}
Since $\psi^*$ and $\psi_*$ are left exact and right exact, respectively, we may derive them to obtain functors $\Perf(B) \to D(\Lambda)$ and $D(\Lambda) \to \mathscr{D}(B)$, which we will also call $\psi^*$ and $\psi_*$. This data gives a weakly crepant categorical resolution of $\mathscr{D}(B)$.
\end{proposition}
\begin{proof}
The functors $\psi^* = - \otimes_B e_0 \Lambda$ and $\psi_* = \Hom_\Lambda(e_0 \Lambda,-)$ form an adjoint pair, and the same is true for their derived analogues.
There is an isomorphism
\begin{align*}
\psi_* \psi^*(B) = \Hom_\Lambda(e_0 \Lambda, B \otimes_B e_0 \Lambda) \cong \End_\Lambda(e_0 \Lambda) \cong e_0 \Lambda e_0 = B,
\end{align*}
from which it follows that the natural morphism $\id_{\rMod B} \to \psi_* \psi^*$ restricts to an equivalence on projective $B$-modules. Lemma \ref{lem:star} then shows that $\psi^*$ and $\psi_*$ give rise to a categorical resolution of $\mathscr{D}(B)$. Moreover, if $N \in \rMod \Lambda$ and $P$ is a projective $B$-module then
\begin{align*}
\Hom_B(\psi_*(N),P) \cong \Hom_B(\psi_!(N),P) \cong \Hom_\Lambda(N,\psi^!(P)) \cong \Hom_\Lambda(N,\psi^*(P)),
\end{align*}
where we have used Lemma~\ref{lem:shriek} and the fact that $(\psi_!,\psi^!)$ form an adjoint pair. Thus $(\psi_*,\psi^*)$ are an adjoint pair on the level of modules and hence the same is true for their derived counterparts, and so this categorical resolution is weakly crepant.
\end{proof}

There is also a notion of a \emph{strongly crepant} categorical resolution \cite{VdBNCCR}, corresponding to triviality of the relative dualising complex; we believe this resolution to be strongly crepant, but have not been able to prove it. 
By \cite[Lemma~2.2.1]{VdBNCCR}, NCCRs of normal Gorenstein algebraic varieties are strongly crepant.

There is another sense in which $\Lambda= \End_B(M)$ is a ``crepant'' resolution of $B$:  the module $M_B$ is reflexive and the inclusion $B = A^{C_{n+1}} \subseteq A \hash C_{n+1} \cong  \Lambda$ makes $\Lambda$ an MCM left and right $B$-module. Indeed, as a (left or right) $B$-module we have $A \hash C_{n+1} \cong A^{n+1}$, and $A$ is an MCM $B$-module by \cite[Theorem A]{ckwz0}.
Thus $\Lambda$ satisfies the definition of an NCCR if we allow the singular ring to be noncommutative.

\section{The representation ring}\label{REPRING} 

Lemmata~\ref{lem:NCQR} and \ref{lem:MoritaEquiv} show that it is reasonable to think of $\Lambda_q$ as an algebraic resolution of singularities of $B_q$.
The rest of the paper will be devoted to constructing a more ``geometric'' resolution of $B_q$, studying it,  and comparing this resolution to $\Lambda_q$.

We begin by reviewing the geometric resolution $\widetilde{X}$ of $X = \AA^2/G = \Spec B_1$.
This can be constructed via quiver GIT (geometric invariant theory), since  $G = C_{n+1} \leqslant \SL_2(\Bbbk)$ is cyclic. 
This is essentially due to \cite{BKR}, but we follow \cite[Theorem~6.3.1]{VdBNCCR}, as explained in \cite[Section 3]{NCCR}. 

Let $q=1$ and let $\Lambda = \Lambda_1$ be the ring defined by \eqref{quiverdiagram}, which we have seen is an NCCR of  the type $\AA$ singularity $B = B_1$.
Let $Y = \operatorname{Rep}(\Lambda, (1, 1, \dots, 1))$ be the variety parametrising representations of $\Lambda$ with dimension vector $(1,1,1, \dots, 1)$.
Such a representation is given by a choice of $2(n+1)$ scalars, by which the generators of $\Lambda$ act, satisfying the relations of $\Lambda$.
Equivalently we have
\[  \kk[Y] = \kk[a_0, \dots, a_n, \a_0, \dots, \a_n]/\ang{a_0 \a_0 = a_1\a_1 = \dots a_n \a_n}.\]

Let $R:=\kk[Y]$.
There is a natural action of $\mb T^{n+1} = (\kk^\times)^{n+1}$ on $Y$, given by scaling the coordinates in such a way as to preserve $a_0 \a_0$, and two representations are isomorphic if and only if they are in the same $\mb T^{n+1}$-orbit. 
Equivalently, there is a natural $\ZZ^{n+1}$-grading on $R$ given by setting $\deg a_i = e_i, \hspace{3pt} \deg \a_i = -e_i$.
Note that the invariant ring $R^{\mb T^{n+1}}= R_{\vec{0}}$ is isomorphic to the singularity $A^{C_{n+1}} = \kk[X] = B_1$; that is, $X$ is isomorphic to the the categorical quotient $Y /\!/ \mb T^{n+1}$.
Moreover, the GIT quotient $Y /_\chi\, \mb T^{n+1}$, for a suitable weight $\chi$, gives the minimal resolution $\widetilde{X}$ of $X$.
Explicitly, let $\chi = (-n, 1,\dots, 1)$ and define 
\[ \widetilde{X} := Y /_\chi\, \mb T^{n+1} = 
\Proj \Bigl ( \bigoplus_{n \geq 0} R_{n \chi}  \Bigr). \]
Thus it is  natural  to attempt to construct an alternative resolution of our noncommutative singularity $B_q$ is by quantising the  representation variety $Y$, or at least its coordinate ring $R$.

\subsection{Defining the representation ring}\label{DEFREPRING}

Let $Y = \Rep(\Lambda_1) := \Rep(\Lambda_1, (1,1, \dots, 1))$ be the representation variety of $\Lambda_1$, where we have suppressed the dimension vector in the notation. We refer to the coordinate ring of $Y$, that is 
\begin{align} \label{eqn:DefOfR}
R_1 = \kk[Y] = \kk[\Rep(\Lambda_1)] = \frac{\Bbbk[a_i, \a_i \mid 0 \leqslant i \leqslant n]}{\langle a_0 \a_0 = a_1 \a_1 = \dots = a_n \a_n \rangle} ,
\end{align}
as the {\em representation ring} of $\Lambda_1$. 

Now let $q \in \Bbbk^\times$ be arbitrary.
As usual, we cannot expect a geometric object to parametrise representations of $\Lambda_q$, but we can still define a useful deformation of the representation ring $R_1$.
\begin{defn}\label{def:repring}
The {\em representation ring} of $\Lambda_q$ is 
\begin{align*}
R^{(n)}_q \coloneqq \frac{\Bbbk \langle a_i, \a_i \mid 0 \leqslant i \leqslant n \rangle}{\left \langle \hspace{-3pt}
\begin{array}{c|c}
\begin{array}{cc}
a_j a_i = a_i a_j & \a_j \a_i = \a_i \a_j  \\
\a_j a_i = q a_i \a_j & a_0 \a_0 = a_i \a_i
\end{array}
&
\;0 \leqslant i,j \leqslant n
\end{array}
\right \rangle}.
\end{align*}
When the values of $n$ or $q$ are clear from context or immaterial, we will simply write $R$ (or sometimes $R_q$) for this ring. \end{defn}

 We begin by recording some basic properties of $R^{(n)}_q$.

\begin{lemma}\label{lem:R-basics}
Let $R^{(n)} = R^{(n)}_q$.  
Then $R^{(n)}$:
\begin{enumerate}[{\normalfont (1)},topsep=0pt,itemsep=0pt,leftmargin=*]
\item 
 is connected graded with  Hilbert series
\begin{align*}
h_{R^{(n)}}(t) = \frac{(1-t^2)^n}{(1-t)^{2(n+1)}}
\end{align*}
if we set $\deg a_i = 1 = \deg \a_i$ for all $i$; 
\item has GK-dimension $ n+2$;
\item  is a domain;
\item is noetherian.
\end{enumerate}
\end{lemma}
\begin{proof}\mbox{}
\begin{enumerate}[{\normalfont (1)},wide=0pt,topsep=0pt,itemsep=0pt]
\item Consider the monomials of $R^{(n)}$ under the graded lexicographic ordering with
\begin{align*}
a_0 < \a_0 < a_1 < \a_1 < \dots < a_n < \a_n.
\end{align*}
It is then straightforward to check that the given generators for the ideal of relations defining $R^{(n)}$ form a Gr\"obner basis. Therefore a basis for the degree $d$ piece of $R^{(n)}$ is
\begin{align*}
\{ a_0^{i_0} \a_0^{j_0} x_1^{i_1} x_2^{i_2} \dots x_n^{i_n} \mid 0 \leqslant i_k \leqslant d,\hspace{2pt} 0 \leqslant j_0 \leqslant d, \hspace{2pt} i_0 + j_0 + i_1 + \dots + i_n = d, \hspace{2pt} x_k \in \{ a_k, \a_k \} \}.
\end{align*}
In particular, it follows that $a_0$ is a nonzerodivisor which is also normal, and therefore we have
\begin{align*}
h_{R^{(n)}}(t) = \frac{h_{R^{(n)}/\langle a_0 \rangle}(t)}{1-t}.
\end{align*}
\mbox{}\hspace{12pt}Now define an auxiliary algebra
\begin{align*}
S^{(n)} \coloneqq \frac{\Bbbk \langle a_i, \a_i \mid 1 \leqslant i \leqslant n \rangle}{\left \langle \hspace{-3pt}
\begin{array}{c|c}
\begin{array}{cc}
a_j a_i = a_j a_i & \a_j \a_i = \a_i \a_j  \\
\a_j a_i = q a_i \a_j & a_i \a_i = 0 
\end{array}
&
1 \leqslant i,j \leqslant n
\end{array}
\right \rangle},
\end{align*}
and note that
\begin{align*}
R^{(n)}/\langle a_0 \rangle \cong S^{(n)}[\a_0; \theta].
\end{align*}
is a skew polynomial ring for an appropriate choice of automorphism $\theta : S^{(n)} \to S^{(n)}$. It follows that
\begin{align*}
h_{R^{(n)}}(t) = \frac{h_{S^{(n)}[\a_0]}(t)}{1-t} = \frac{h_{S^{(n)}}(t)}{(1-t)^2},
\end{align*}
so it suffices to determine the Hilbert series of $S^{(n)}$. \\
\mbox{}\hspace{12pt}By an argument similar to the one used for $R^{(n)}$, a basis for the degree $d$ piece of $S^{(n)}$ is given by
\begin{align*}
\{ x_1^{i_1} \dots x_n^{i_n} \mid 0 \leqslant i_k \leqslant d, \hspace{2pt} i_1 + \dots + i_n = d, \hspace{2pt} x_k \in \{ a_k, \a_k \} \}.
\end{align*}
We can decompose this set as
\begin{align*}
\bigsqcup_{\ell=0}^{n} \{ x_{j_1}^{i_{j_1}} \dots x_{j_\ell}^{i_{j_\ell}} \mid 1 \leqslant i_{j_k} \leqslant d, \hspace{2pt} i_{j_1} + \dots + i_{j_\ell} = d, \hspace{2pt} x_k \in \{ a_k, \a_k \}, \hspace{2pt} 1 \leqslant j_1 < j_2 < \dots < j_\ell \leqslant n \}.
\end{align*}
In each disjoint subset of this union, there are $\binom{n}{\ell}$ choices for $(j_1, \dots, j_\ell)$ and then $2^\ell$ choices for $(x_{j_1}, \dots, x_{j_\ell})$. Therefore the generating function for the $\ell$th stratum is
\begin{align*}
\frac{2^\ell \binom{n}{\ell} t^\ell}{(1-t)^\ell} = \frac{\binom{n}{\ell} (2t)^\ell}{(1-t)^\ell}.
\end{align*}
It follows that
\begin{align*}
h_{S^{(n)}}(t) &= \sum_{\ell=0}^{n} \frac{\binom{n}{\ell} (2t)^\ell}{(1-t)^\ell} = \frac{1}{(1-t)^{n}} \sum_{\ell=0}^{n} \binom{n}{\ell} (2t)^\ell(1-t)^{n-\ell} = \frac{\Big((1-t)+(2t)\Big)^{n}}{(1-t)^{n}} 
= \frac{(1-t^2)^{n}}{(1-t)^{2n}}.
\end{align*}
Therefore,
\begin{align*}
h_{R^{(n)}}(t) = \frac{h_{S^{(n)}}(t)}{(1-t)^2} = \frac{(1-t^2)^{n}}{(1-t)^{2(n+1)}}.
\end{align*}
\item In reduced form, the Hilbert series of $R^{(n)}$ is
\begin{align*}
h_{R^{(n)}}(t) = \frac{(1+t)^{n}}{(1-t)^{n+2}}.
\end{align*}
Since this has a pole at $1$ of order $n+2$, we have $\GKdim R^{(n)} = n+2$.
\item The proof of part (1) shows that $z \coloneqq a_0 \a_0$ is a normal nonzerodivisor in $R^{(n)}$, so we have an inclusion
\begin{align*}
R^{(n)} \hookrightarrow R^{(n)}[z^{-1}].
\end{align*}
It is straightforward to see that there is an isomorphism
\begin{align*}
R^{(n)}[z^{-1}] \cong \frac{\Bbbk[ a_0^{\pm 1}, \dots, a_n^{\pm 1}]\ang{ z^{\pm 1} }}{\langle z a_i = q a_i z \mid 0 \leqslant i \leqslant n \rangle},
\end{align*}
which, being a localisation of a quantum plane, is a domain, and hence the same is true of $R^{(n)}$.
\item Since $R^{(n)}$ is a quotient of a quantum polynomial ring, it is noetherian. \qedhere
\end{enumerate}
\end{proof}

\subsection{A universal property}\label{UNIVERSAL}
Let $R = R^{(n)}_q$.
If $q=1$ then $R$ is the coordinate ring of the representation variety of  $\Lambda_1$.  
It is thus reasonable to ask if  $R$ is ``universal'' for representations of $\Lambda_q$ for arbitrary $q$.  In this section we show that this is true at least when $q$ is not a root of unity, or when we consider $R_q$ as a family parametrised by $q$.

We first give a general definition.

\begin{defn}\label{def:rep}
Let $\KK$ be a commutative ring, let $Q$ be a quiver, and let $\Lambda = \KK Q/I$ be a basic $\KK$-algebra.
We assume here  that $I \subseteq (\KK Q)_{\geq 2}$ is a graded ideal with respect to the path length grading on $\KK Q$.  
    Let $\Gamma$ be an $\NN$-graded $\KK$-algebra,
    and induce a grading on $M_{Q_0}(\Gamma)$ by giving the matrix units degree 0.
    A {\em graded $\Gamma$-representation} of $\Lambda$ 
    is an $\NN$-graded $\KK$-algebra homomorphism
    \[ \iota = \iota_\Gamma: \Lambda \to M_{Q_0}(\Gamma) = \End_\Gamma(\Gamma^{Q_0})\]
    which sends, for each $i \in Q_0$, the vertex idempotent $e_i$ to the matrix unit $e_{ii}$,
    and which is does not send any arrows (length 1 paths in $\Lambda$) to $0$.  

Thus $\iota(e_i \Lambda e_j) \subseteq e_{i} M_{Q_0}(\Gamma) e_j = e_{ij}\Gamma $, which we identify with $\Gamma$.  
We will abuse notation and write $\iota(e_i \Lambda e_j) \subseteq \Gamma$.
\end{defn}

\begin{remark} Technically, this should be called a graded $\Gamma$-representation with   dimension vector $(1, \cdots, 1)$, but we will suppress the dimension vector in the notation.
\end{remark}

Given a graded $\Gamma$-representation of $\Lambda$, then $\iota_\Gamma$ induces an action of $\Lambda$ on $\Gamma^{Q_0}$, which is thus 
a representation of $\Lambda$ in the usual sense:  the arrows in $\Lambda$ take us from one copy of $\Gamma$ to another.
This corresponds to a right action of $\Lambda$ on $\Gamma^{Q_0}$, and so we regard $\Gamma^{Q_0}$ as a row vector.

We now formally define the family of the $R_q$.  
Let $\KK = \kk[q, q^{-1}]$, where $q$ is an indeterminate.
Fix $n$ and let $R_\KK = R^{(n)}_\KK$ be the $\KK$-algebra defined by the presentation
\begin{align*}
R_\KK := \frac{\KK \langle a_i, \a_i \mid 0 \leqslant i \leqslant n \rangle}{\left \langle \hspace{-3pt}
\begin{array}{c|c}
\begin{array}{cc}
a_j a_i = a_i a_j & \a_j \a_i = \a_i \a_j  \\
\a_j a_i = q a_i \a_j & a_0 \a_0 = a_i \a_i
\end{array}
&
0 \leqslant i,j \leqslant n
\end{array}
\right \rangle}.
\end{align*}
Let $\Lambda_\KK$ be the path algebra over $\KK$ of the quiver $Q$ with relations as in \eqref{quiverdiagram}.
Further, let $A_\KK = \KK\ang{u,v}/\ang{vu-quv}$.

\begin{lemma}\label{lem:flat}
The $\KK$-algebras $A_\KK$, $R_\KK$, and $\Lambda_\KK$ are flat over $\KK$. 
\end{lemma}
\begin{proof}
    For $A_\KK$ the statement is obvious.
    For $R_\KK$ and $\Lambda_\KK$ it follows from Lemma~\ref{lem:R-basics}(1) and \cite[Theorem~23.1]{Mats} (see also \cite[Theorem~III.9.9]{Ha}).
\end{proof}

In this section we will prove that $R_\KK$ gives a representation of $\Lambda_\KK$ which is universal for graded $A_\KK$-representations,
and that the same holds when we specialise to let $q\in \kk^\times$ be any non-root of unity.

\begin{theorem}\label{thm:universal}\mbox{}
\begin{enumerate}[{\normalfont (1)},topsep=0pt,itemsep=0pt,leftmargin=*]
\item There is a graded $R_\KK$-representation $\iota_{ R_\KK}$ of $\Lambda_\KK$ given by sending $\alpha_i \to a_i$, $\overline{\alpha}_i \to \overline{a}_i$. 
This specialises to give a graded $R_q$-representation of $\Lambda_q$ for any $q \in \kk^\times$.

\item 
$R_\KK$ is universal for graded $A_\KK$-representations of $\Lambda_\KK$ in the sense that given a graded $A_\KK$-representation $\iota_A$ of $\Lambda_\KK$, there is a unique $\KK$-algebra homomorphism $\phi: R_\KK \to A_\KK$ so that the diagram
\begin{equation*}
\begin{tikzcd}[>=stealth,ampersand replacement=\&,row sep=30pt,every label/.append style={font=\normalsize},column sep=35pt] 
\Lambda_\KK \arrow[dr, "\iota_A" {swap}] \arrow[r, "\iota_R" {above,yshift=1pt}] \& M_{Q_0}(R_\KK) \arrow[d, "M_{Q_0}(\phi)"] \\
\& M_{Q_0}(A_\KK)
\end{tikzcd}
\end{equation*}
commutes.

\item If $q$ is not a root of unity, then $R_q$ is universal for graded $A_q$-representations of $\Lambda_q$ in the sense defined by the obvious modification of \emph{(2)}.
\end{enumerate}
\end{theorem}

\begin{remark}\label{rem:universal1}
As usual, it follows from the universal properties in Theorem~\ref{thm:universal} that $R_\KK$ is unique up to isomorphism, and similarly for $R_q$ when $q$ is not a root of unity.
This justifies calling $R$ {\em the} representation ring of $\Lambda$.
\end{remark}

Before giving the proof of Theorem~\ref{thm:universal}, we give two quick lemmata about multiplication in $A_\KK$ and $A_q$.

\begin{lemma}\label{lone}\mbox{}
\begin{enumerate}[{\normalfont (1)},topsep=0pt,itemsep=0pt,leftmargin=*]
\item 
Let $f,g $ be nonzero elements of $( A_\KK)_{n+1}$. 
Then  $gf = q^{(n+1)^2} fg$ if and only if $g \in \KK v^{n+1}$, $f \in \KK u^{n+1}$.
\item
Suppose that $q$ is not a root of unity, let $A = A_q$, and let $f,g $ be nonzero elements of $ A_{n+1}$. 
Then  $gf = q^{(n+1)^2} fg$ if and only if $g \in \kk v^{n+1}$, $f \in \kk u^{n+1}$.
\end{enumerate}
\end{lemma}
\begin{proof}
The proofs are similar; we prove (2).
We have
\beq\label{bcad} u^a v^b u^c v^d = q^{bc-ad} u^c v^d u^a v^b,\eeq
and sufficiency follows immediately.  
For necessity, first suppose that  $g = u^{n+1-b}v^b, f=u^c v^{n+1-c}$ are monomials.
From \eqref{bcad},  $gf = q^{(n+1)^2} fg$, and the assumption that $q$ is not a root of unity, we have
\[ 
(n+1)^2 = bc - (n+1-b)(n+1-c) = (n+1)(b+c-(n+1));
\]
in other words  $b+c = 2(n+1)$.  
As $0 \leq b,c \leq n+1$, the only way this can happen is if $b = c = n+1$.

For general $f,g$, put a monomial order on monomials in $A$ so that $u>v$, so $v^{n+1}$ is the smallest monomial in $A_{n+1}$.  
The leading monomials of $f,g$ must $q^{(n+1)^2}$-commute, and the monomial case forces $g = \lambda v^{n+1}$.
It is then easy to see that for $f$ to $q^{(n+1)^2}$-commute with $g$ forces $f = \mu u^{n+1}$.
\end{proof}

\begin{lemma}\label{ltwo}\mbox{}
\begin{enumerate}[{\normalfont (1)},topsep=0pt,itemsep=0pt,leftmargin=*]
\item
Let $x_0, \dots, x_n \in (A_\KK)_1$ be nonzero.
Then
$x_0 x_1 \dots x_n \in \KK u^{n+1}$ if and only if $x_0, \dots, x_n \in \KK u$, and
$x_0 x_1 \dots x_n \in \KK v^{n+1}$ if and only if $x_0, \dots, x_n \in \KK v$.
\item 
Let $q \in \kk^\times$ be arbitrary and let $A = A_q$.
Let $x_0, \dots, x_n \in A_1$ be nonzero.
Then
$x_0 x_1 \dots x_n \in \kk u^{n+1}$ if and only if $x_0, \dots, x_n \in \kk u$, and
$x_0 x_1 \dots x_n \in \kk v^{n+1}$ if and only if $x_0, \dots, x_n \in \kk v$.
\end{enumerate}
\end{lemma}

\begin{proof}
For the first claim, consider the leading term of $x_0\dots x_n$ in a monomial order with $v>u$.  For the second claim, use an order with $u>v$.
\end{proof}

\begin{proof}[Proof of Theorem~\ref{thm:universal}]
(1) We wish to define a homomorphism $\iota_R: \Lambda_\KK \to M_{Q_0}(R_\KK)$ by sending
\[ \alpha_i \mapsto e_{i,i+1} a_i , \quad \overline{\alpha}_i \mapsto e_{i+1,i} \a_i,\]
where the subscripts are interpreted modulo $n+1$.
By construction,  $\iota_R$ respects the relations of $\KK Q$. Moreover, we have
\[
\iota_R(\overline{\alpha}_i) \iota_R(\alpha_i) 
= e_{i+1,i+1} \a_i a_i 
= q e_{i+1,i+1} a_{i+1} \a_{i+1}  
= q \iota_R(\alpha_{i+1}) \iota_R(\overline{\alpha}_{i+1}),
\]
so $\iota_R$ respects the relations of $\Lambda_\KK$ and gives a well-defined homomorphism.
It is clear that $\iota_R$ specialises as claimed.

Statements (2) and (3) have similar proofs; we prove (2).
A graded $A_\KK$-representation of $\Lambda_\KK$ is equivalent to a choice of nonzero $a_0, \dots, a_n, \overline{a}_0, \dots, \overline{a}_n \in (A_\KK)_1$ so that
\[ \overline{a}_i a_i = q a_{i+1} \overline{a}_{i+1}\]
for $i \in \{0, \dots, n\}$, where the subscripts are understood modulo $n+1$.
We will show (abusing notation) that $a_0, \dots, a_n, \overline{a}_0, \dots, \overline{a}_n$ satisfy the relations of $R_\KK$.

Let $f = a_0\dots a_n$, $g = \overline{a}_n \dots \overline{a}_0$.
Then
\begin{align*}
gf & = q \overline{a}_n \dots \overline{a}_2 (\overline{a}_1 a_1)^2 a_2 \dots a_n \\
 & = q q^2 \overline{a}_n \dots \overline{a}_3 (\overline{a}_2 a_2)^3 a_3 \dots a_n \\
 & = \dots = q^{1+\dots + n}(\overline{a}_n a_n)^{n+1}= q^{\binom{n+1}{2}}q^{n+1}(a_0 \overline{a}_0)^{n+1}.
 \end{align*}
 Likewise,
 \[ fg = q^{-\binom{n+1}2} (a_0 \overline{a}_0)^{n+1},\]
 so $gf = q^{(n+1)^2} fg$.
 Combining Lemmata~\ref{lone} and \ref{ltwo}, we see that $a_0, \dots, a_n \in \KK u$, and $\overline{a}_0, \dots, \overline{a}_n \in \KK v$.
 It follows immediately that the relations for $R_\KK$ hold for 
 $a_0, \dots, a_n,\overline{a}_0, \dots, \overline{a}_n$.
\end{proof}

\begin{remark}\label{rem:universal}
If $q=1$, then the representation variety of $\Lambda$ parametrises all (commutative) representations of $\Lambda$, and so $R$ is universal for (ungraded) representations of $\Lambda$ into any commutative ring, using the obvious slight generalisation of Definition~\ref{def:rep}.  
Since any graded $A$-representation of $\Lambda$ is a commutative representation, it follows that $R$ is also universal for graded $A$-representations.  The argument here is quite different, however, than the proof of Theorem~\ref{thm:universal}.
\end{remark}

\begin{example}\label{eg:notuniversal}
If $q$ is a root of unity then $R_q$ need not be not universal in the sense of Theorem~\ref{thm:universal} (3), even given the universality of $R_\KK$.
For example, fix $n$ and suppose that $q$ is a primitive $(n+1)$th root of unity. 
Let $A = A_q$, $\Lambda = \Lambda_q$, and define an $A$-representation of $\Lambda$ by
\begin{align*}
\iota_A : \Lambda \to M_{Q_0}(A), \qquad a_i \mapsto v \;e_{i,i+1}, \quad \overline{a}_i \mapsto q^{-2i} u\; e_{i+1,i}.
\end{align*}
This gives an $A$-representation of $\Lambda$ since
\begin{align*}
\iota_A(\overline{a}_i a_i - q a_{i+1} \overline{a}_{i+1}) = \left\{
\begin{array}{ll}
q^{-2i} uv - q^{-2(i+1) + 1}v u = q^{-2i} uv - q^{-2(i+1) + 2}uv = 0 \hspace{5pt} & \text{if } i \neq n, \\
q^{-2n} uv - q vu = q^{-2n} uv - q^2 uv = 0 & \text{if } i = n.
\end{array}\right.
\end{align*}
However, it is easy to see that this representation does not factor through $R_q$.
\end{example}

\subsection{An embedding of \texorpdfstring{$R$}{R}}
Fix $n, q$ and let $R = R^{(n)}_q$.
For later use, we  note that $R$ embeds in an iterated Laurent polynomial ring over $A$.

\begin{lemma}\label{lem:EmbeddingOfR}
There is an embedding
\begin{align*}
\varphi: R 
\hra A[s_0^{\pm 1}, \dots, s_n^{\pm 1}], \qquad
a_i & \mapsto u \frac{s_{i+1}}{s_i}, \quad
\overline{a_i} \mapsto v \frac{s_{i}}{s_{i+1}}.
\end{align*}
\end{lemma}
\begin{proof}
It is straightforward to verify that $\varphi$ respects the relations of $R$.
Let 
\beq\label{Rprime}
R' = \Bbbk \big\langle u \tfrac{s_{i+1}}{s_{i}}, v \tfrac{s_i}{s_{i+1}} \mid 0 \leqslant i \leqslant n \big\rangle \subseteq A[s_0^{\pm 1}, \dots, s_n^{\pm 1}]
\eeq
be the image of $\varphi$. Note that $uv $ is a normal element of $ R'$ and that 
\begin{align*}
R'[(uv)^{-1}] = \Bbbk_{\mathbf{q}}\Big[ \Big( v \tfrac{s_1}{s_0} \Big)^{\pm 1}, \Big( v \tfrac{s_2}{s_1} \Big)^{\pm 1}, \dots, \Big( v \tfrac{s_n}{s_0} \Big)^{\pm 1}, (uv)^{\pm 1} \Big],
\end{align*}
which is a localisation of some quantum polynomial ring in $n+2$ variables and hence has GK-dimension 
$n+2$. It follows that $\GKdim R' = n+2$. We therefore have a surjection $R \twoheadrightarrow R'$ between domains of GK-dimension $n+2$.
By \cite[Proposition~3.15]{lenagan}, this map must be an isomorphism. 
\end{proof}

\section{A geometric resolution: the ring}\label{GEOMRESOLUTION}

In this section we construct an $\NN$-graded ring $T$ which will give our ``geometric'' resolution of $B = B_q$.
Unsurprisingly, $T$ will be a deformation of the semi-invariant ring used to define the minimal resolution of a commutative $\AA_n$ singularity.

\subsection{The ring \texorpdfstring{$\T$}{T}}
We first define $T$, or $T^{(n)}_q$ if we want to make the dependence on $n$ and $q$ explicit.
Again, let $A = A_q$ and $R = R^{(n)}_q$.
If we give $\kk[s_0^{\pm 1}, \dots, s_n^{\pm 1}]$ the standard $\ZZ^{n+1}$-grading with $\deg s_i = e_i$, then there is an induced $\ZZ^{n+1}$-grading on $A[s_0^{\pm 1}, \dots, s_n^{\pm 1}]$ and on $R$ via the embedding  in Lemma~\ref{lem:EmbeddingOfR}.
We will refer to this grading as the {\em weight grading} on $R$, and write $\wt(a)$ for the degree of a weight-homogeneous element $a$.

Following the approach of quiver GIT, set 
\[\chi \coloneqq (-n,1,1, \dots, 1) \in \ZZ^{n+1}\]
and let $T = \T^{(n)}_q$ be the ring of semi-invariants of $R$ corresponding to $\chi$:
\begin{align*}
\T \coloneqq \bigoplus_{k \in \NN} R_{k \chi}.
\end{align*}
That is, $\T$ is the $\NN$-graded ring whose $k$th graded piece consists of elements of $R$ of weight $k \chi$.
We will  refer to this as the \emph{projective grading} on $\T$, to avoid confusion with the other gradings that we can (and will) use.
If $k \in \NN$, then $T_k$ is always the $k$th graded piece under the projective grading.
Later we will use the projective grading on $\T$ to define a category $\rqgr \T$, as in noncommutative projective geometry.

Before doing this, though, we need to study the elementary algebraic properties of $T$.  
 We can view $\T=\T^{(n)}_q$ as a subring of $R=R^{(n)}_q$ as given in \eqref{eqn:DefOfR}, or as a subring of $A_q[s_i^{\pm 1}]$ via 
 Lemma~\ref{lem:EmbeddingOfR}; both perspectives will be useful.
 
 We first record a generating set for $\T$.

\begin{proposition} \label{prop:GeneratorsOfT}
Fix $n$ and let $\T = \T^{(n)}_q$.  
Then $T$ is generated by the elements
\begin{gather*}
x \coloneqq a_0 a_1 \dots a_n, \qquad y \coloneqq \a_0 \a_1 \dots \a_n, \qquad z \coloneqq a_0 \a_0, \\
t_i \coloneqq a_0^{n-i} a_1^{n-i-1} \dots a_{n-i-2}^2 a_{n-i-1} \a_{n-i+1} \a_{n-i+2}^2 \dots \a_{n-1}^{i-1} \a_n^i.
\end{gather*}
Further, $T_0 = B$ and $\GKdim T \geq 3$.
\end{proposition}

\begin{remark}
We will see later that $\GKdim T = 3$.
\end{remark}

\begin{proof}[Proof of Proposition~\ref{prop:GeneratorsOfT}]
Observe that $x,y,z \in T_0$ and $t_0, \dots, t_n \in T_1$.
A general element of $R$ is homogeneous of projective degree $d$ if and only if each monomial in the $a_i$ and $\a_j$ appearing in it with nonzero coefficient has projective degree $d$, so it suffices to show that the claimed generating set generates all monomials in $T_d$ for all $d \in \NN$. 
We will prove this by induction on $d$.  \\
\indent We first show that $x$, $y$, and $z$ generate $\T_0$. Any monomial in $R$ can be written in the form
\begin{align*}
a = a_0^{i_0} \dots a_n^{i_n} \a_0^{j_0} \dots \a_n^{j_n},
\end{align*}
where all of the exponents are non-negative. So suppose a monomial $a$ as above has projective degree $0$. For $0 \leqslant k \leqslant n$, if $i_k$ and $j_k$ are both nonzero, then we can pull out a factor of $z = a_k \a_k$. Therefore we may assume that at least one of $i_k$ and $j_k$ is $0$ for all $0 \leqslant k \leqslant n$. By considering the $n+1$ components comprising the weight of $a$, we obtain $n+1$ linear equations:
\begin{align*}
-i_k + i_{k-1} + j_k - j_{k-1} = 0, \qquad 0 \leqslant k \leqslant n,
\end{align*}
where the subscripts are interpreted modulo $n+1$. First assume $j_0 = 0$, so that $-i_0 + i_n - j_n = 0$. For contradiction, suppose $j_n > 0$. By our earlier assumption, this forces $i_n = 0$, and hence $i_0 + j_n = 0$,
which is absurd. 
Therefore $j_n = 0$, which also gives $i_0 = i_n$.
Repeating this with the other equations shows that $i_0 = i_1 = \dots = i_n \eqqcolon c$ and $j_k = 0$ for $0 \leqslant k \leqslant n$. 
Therefore $a = a_0^c a_1^c \dots a_n^c$, which is a scalar multiple of $x^c$. If instead $i_0 = 0$, we find that $a$ is a multiple of $y^c$.

This shows that $T_0$ is the subring of $A$ generated by $x, y, z$. 
Under the embedding in Lemma~\ref{lem:EmbeddingOfR}, we can identify $x=u^{n+1}$, $y= v^{n+1}$, $z = uv$, and therefore $T_0 = A^G = B$. As $R$ is a domain and $t_0$ clearly $q^\bullet$-commutes with $x,y$, the subring $\kk\ang{x,y,t_0}$ of $T$ is isomorphic to a quantum 3-space and hence $\GKdim T \geqslant 3$.

\indent Now suppose that $m \in \T$ is a monomial of projective degree $d \geqslant 1$.
Write
\begin{align*}
m = a_0^{i_0} a_1^{i_1} \dots a_n^{i_n} \a_0^{j_0} \a_1^{j_1} \dots \a_n^{j_n},
\end{align*}
for some $i_k, j_k \geqslant 0$.
By factoring out a power of $z = a_k \a_k$, up to rescaling $m$ we may assume that $i_k = 0$ or $j_k = 0$ for $0 \leqslant k \leqslant n$.

First assume that at least one $i_k$ is nonzero, and fix
\begin{align*}
k \coloneqq \max \{ \ell \mid i_\ell \neq 0 \}.
\end{align*}
Therefore $i_{k+1} = \dots = i_n = 0$, $i_k \geqslant 1$, and $j_k = 0$. The $k$th coordinate of the equation
\begin{align}
\wt(m) = d(-n,1,1, \dots, 1), \label{eqn:WeightEquation}
\end{align}
(where here we $0$-index this vector), tells us that
\begin{align*}
d = i_{k-1} - j_{k-1} - i_k,
\end{align*}
where $i_{k-1} = 0$ or $j_{k-1} = 0$. If $i_{k-1} = 0$ then we obtain
\begin{align*}
d = - j_{k-1} - i_k \leqslant 0,
\end{align*}
which is absurd, so we must have $j_{k-1} = 0$. It then follows that
\begin{align*}
i_{k-1} = d + i_k \geqslant 1+1 = 2.
\end{align*}
Repeating this argument along the entries of \eqref{eqn:WeightEquation} shows that
\begin{align*}
i_k \geqslant 1, \quad i_{k-1} \geqslant 2, \quad i_{k-2} \geqslant 3, \quad \dots, \quad i_0 \geqslant k+1.
\end{align*}
Similarly, if we look at the $(k+2)$nd coordinate of \eqref{eqn:WeightEquation}, we obtain the equation
\begin{align*}
-j_{k+1} + j_{k+2} = d.
\end{align*}
If $j_{k+2} = 0$ then this equation cannot be satisfied, so we must have $j_{k+2} \geqslant 1$. An argument similar to the one used for the $i_\ell$ shows that 
\begin{align*}
j_{k+2} \geqslant 1, \quad j_{k+3} \geqslant 2, \quad \dots, \quad j_n \geqslant n-k-1. 
\end{align*}
This shows that $a_0^{k+1} a_1^k \dots a_{k-1}^2 a_k \a_{k+2} \a_{k+3}^2 \dots \a_{n-1}^{n-k-2} \a_n^{n-k-1} = t_{n-k-1}$ is a factor of $m$, so we can write   $m =  \lambda t_{n-k-1} a'$ for some $\lambda \in \Bbbk^\times$ and a monomial $a' \in R$.
If instead all of the $i_\ell$ are $0$, a similar argument shows that $m = \lambda t_n a'$. 

By the above, we have  $m = \lambda t_\ell a'$ , where $a' \in R$ is a monomial and $\lambda \in \kk^\times$.
We have
$\wt(a') = \wt(m)-\wt(t_\ell) = (d-1)\chi$, so $a' \in T_{d-1} = R_{(d-1)\chi}$.
The claim follows by induction.
\end{proof}

\subsection{A presentation for \texorpdfstring{$T$}{T}}
We now seek to give a presentation for $\T$. Define
\beq\label{eqn:PresentationForT}
\Bbbk_\mathbf{q}[x,y,z,t_0, \dots, t_n] = 
\frac{\Bbbk \langle x,y,z,t_0, \dots, t_n \rangle}{
\sbox0{$
\hspace{-9pt}
\begin{array}{c}
\begin{array}{ccc}
yx = q^{(n+1)^2} xy & zx = q^{n+1} xz & zy = q^{-(n+1)} yz \\
t_i x = q^{(n+1) \binom{i+1}{2}} x t_i & t_i y = q^{-(n+1) \binom{n-i+1}{2}} yt_i & t_i z = q^{\frac{1}{2}(n+1)(2i-n)} z t_i
\end{array} \\
t_j t_i = q^{\frac{1}{4}(n+1)(i-j)(2ij-n(i+j+1))} t_i t_j
\end{array} \hspace{-9pt}
$}
\mathopen{\resizebox{1.2\width}{\ht0}{$\Bigg\langle$}}
\usebox{0}
\mathclose{\resizebox{1.2\width}{\ht0}{$\Bigg\rangle$}}
},
\eeq
which is a skew polynomial ring, and hence AS regular \cite{AS}. 
We note that the exponent 
\[\frac{1}{4}(n+1)(i-j)(2ij-n(i+j+1))\]
is always an integer, so the last relation is well-defined.
Also set
\begin{align*}
\Tpres \coloneqq \frac{\Bbbk_\mathbf{q}[x,y,z,t_0, \dots, t_n]}{
\sbox0{$
\hspace{-6pt}
\begin{array}{c}
xy = q^{-\binom{n+1}{2}} z^{n+1} \\
\begin{array}{cl}
xt_i = q^{\alpha_i} z^i t_{i-1} & \text{\normalfont{for} } 1 \leqslant i \leqslant n \\
yt_i = q^{\beta_i} z^{n-i} t_{i+1} & \text{\normalfont{for} } 0 \leqslant i \leqslant n-1 \\
t_{i-1} t_{j+1} = q^{\gamma_{ij}} z^{j-i+1} t_i t_j & \text{\normalfont{for} } 1 \leqslant i \leqslant j \leqslant n-1
\end{array}
\end{array}
\hspace{-6pt}
$}
\mathopen{\resizebox{1.2\width}{\ht0}{$\Bigg\langle$}}
\usebox{0}
\mathclose{\resizebox{1.2\width}{\ht0}{$\Bigg\rangle$}}
},
\end{align*}
where
\begin{gather*}
{\alpha_i} = -\binom{i}{2} - i \binom{n-i+2}{2}, \quad {\beta_i} = (i+2) \binom{n-i+1}{2}, \\
{\gamma_{ij}} 
= -(j+1) \binom{n-j+1}{2} - (n-j)\binom{i}{2} - (j-i+1)\binom{n-i+1}{2} - \binom{j-i+1}{2}.
\end{gather*}
As suggested by our choice of notation, we claim that $\T$ and $\Tpres$ are isomorphic. As a first step, we show that the generators of $T$ satisfy the relations in $\Tpres$:

\begin{lemma} \label{lem:RelationsInT}
The elements $x,y,z,t_i \in T$ satisfy the relations in $\Tpres$.
\end{lemma}
\begin{proof}
Let $x,y,z, t_0, \dots, t_n \in \T$ be as in Proposition~\ref{prop:GeneratorsOfT}. We begin the proof by showing that these elements satisfy the skew-commutativity relations given in the presentation for $\Bbbk_\mathbf{q}[x,y,z,t_0, \dots, t_n]$.
It is clear that any pair of generators $\qb$-commute, for  some power $\qb$ of $q$.
To determine the precise power of $q$, we will view the generators as elements of $A[s_0^{\pm 1}, \dots, s_n^{\pm 1}]$ via Lemma~\ref{lem:EmbeddingOfR}. 
If $a \in A[s_0^{\pm 1}, \dots, s_n^{\pm 1}]$, write $a \approx f(u,v)$ for the image of $a$ in
\begin{align*}
A[s_0^{\pm 1}, \dots, s_n^{\pm 1}]/\langle s_i - 1 \mid 0 \leqslant i \leqslant n \rangle \cong A.
\end{align*}
It suffices to work in this quotient ring to determine the power of $q$ in the $\qb$-commutativity relation between two generators. \\
\indent With this notation, we have
\begin{align*}
x \approx u^{n+1}, \quad y \approx v^{n+1}, \quad z \approx uv, \quad t_i \approx u^{\binom{n-i+1}{2}} v^{\binom{i+1}{2}}.
\end{align*}
Direct calculation now gives
\begin{gather*}
yx \approx v^{n+1} u^{n+1} = q^{(n+1)^2} u^{n+1} v^{n+1} \approx q^{(n+1)^2} xy, \\
zx \approx uv \cdot u^{n+1} = q^{n+1} u^{n+1} \cdot uv \approx q^{n+1} xz, \\
zy \approx uv \cdot v^{n+1} = q^{-(n+1)} v^{n+1} \cdot uv \approx q^{-(n+1)} yz.
\end{gather*}
For the commutativity relations involving $t_i$ and one of $x,y,$ or $z$, we have
\begin{gather*}
t_i x \approx u^{\binom{n-i+1}{2}} v^{\binom{i+1}{2}} u^{n+1} = q^{(n+1) \binom{i+1}{2}} u^{n+1} u^{\binom{n-i+1}{2}} v^{\binom{i+1}{2}} \approx q^{(n+1) \binom{i+1}{2}} x t_i, \\
t_i y \approx u^{\binom{n-i+1}{2}} v^{\binom{i+1}{2}} v^{n+1} = q^{-(n+1) \binom{n-i+1}{2}} v^{n+1} u^{\binom{n-i+1}{2}} v^{\binom{i+1}{2}} \approx q^{-(n+1) \binom{n-i+1}{2}} y t_i \\
t_i z \approx u^{\binom{n-i+1}{2}} v^{\binom{i+1}{2}} uv = q^{\binom{i+1}{2} - \binom{n-i+1}{2}} uv u^{\binom{n-i+1}{2}} v^{\binom{i+1}{2}} = q^{\frac{1}{2}(n+1)(2i-n)} z t_i.
\end{gather*}
The only remaining commutativity relation left to determine is that relating $t_i$ and $t_j$:
\begin{align*}
t_j t_i &\approx u^{\binom{n-j+1}{2}} v^{\binom{j+1}{2}} u^{\binom{n-i+1}{2}} v^{\binom{i+1}{2}} \\
&= q^{\binom{j+1}{2} \binom{n-i+1}{2}} u^{\binom{n-j+1}{2}} u^{\binom{n-i+1}{2}} v^{\binom{i+1}{2}} v^{\binom{j+1}{2}} \\
&= q^{\binom{j+1}{2} \binom{n-i+1}{2} - \binom{i+1}{2} \binom{n-j+1}{2}} u^{\binom{n-i+1}{2}} v^{\binom{i+1}{2}} u^{\binom{n-j+1}{2}} v^{\binom{j+1}{2}} \\
&= q^{\frac{1}{4}(n+1)(i-j)(2ij-n(i+j+1))} t_i t_j
\end{align*}
where the equality in the exponent of $q$ is an easy computation. \\
\indent We now show that the generators satisfy the relations by which we factor $\Bbbk_\mathbf{q}[x,y,z,t_0, \dots, t_n]$ to obtain $\Tpres$. The first three are relatively straightforward. First,
\[
xy \approx u^{n+1} v^{n+1} = q^{-\binom{n+1}{2}} (uv)^{n+1} \approx q^{-\binom{n+1}{2}} z^{n+1} .
\]
Since $xy, z^{n+1} \in T_0 \subseteq A$, this suffices to show that $xy=q^{-\binom{n+1}{2}} z^{n+1}$.

Likewise,
\begin{align*}
x t_i &\approx u^{n+1} \cdot u^{\binom{n-i+1}{2}} v^{\binom{i+1}{2}} \\
&= u^i \cdot u^{n-i+1 \binom{n-i+1}{2}} v^{i+\binom{i}{2}} \\
&= q^{-i \binom{n-i+2}{2}} u^i v^i u^{\binom{n-i+2}{2}} v^{\binom{i}{2}} \\
&\approx q^{- \binom{i}{2} - i \binom{n-i+2}{2}} z^i t_{i-1}  .
\end{align*}
Similar to the above, $x t_i,  z^i t_{i-1} \in As_0^{-n} s_1\dots s_n$ and so the above suffices to show that the relation holds.
Finally,
\begin{align*}
y t_i &\approx v^{n+1} \cdot u^{\binom{n-i+1}{2}} v^{\binom{i+1}{2}} \\
&= q^{(i+1)\binom{n-i+1}{2}} v^{n-i} u^{n-i+\binom{n-i}{2}}v^{i+1+\binom{i+1}{2}} \\
&= q^{(i+1)\binom{n-i+1}{2} + \binom{n-i+1}{2}} (uv)^{n-i} u^{\binom{n-i}{2}} v^{\binom{i+2}{2}} \\
&\approx q^{(i+2)\binom{n-i+1}{2}} z^{n-i} t_{i+1},
\end{align*}
and again this is sufficient.

The final relation requires a little more care. 
Again, it suffices to check the relations are satisfied up to $\approx$.
If $1 \leqslant i \leqslant j \leqslant n-1$, then
\begin{align*}
t_{i-1} t_{j+1} 
&\approx u^{\binom{n-i+2}{2}} v^{\binom{i}{2}} u^{\binom{n-j}{2}} v^{\binom{j+2}{2}} \\
&= u^{n-i+1+\binom{n-i+1}{2}} v^{\binom{i}{2}} u^{\binom{n-j}{2}} v^{j+1+\binom{j+1}{2}} \\
&= q^{-(j+1) \binom{n-j}{2}} u^{j-i+1} \, u^{n-j} \, u^{\binom{n-i+1}{2}} v^{j+1} v^{\binom{i}{2}} u^{\binom{n-j}{2}} v^{\binom{j+1}{2}} \\
&= q^{-(j+1) \binom{n-j}{2} - (n-j)\big( (j+1) + \binom{i}{2} \big)} u^{j-i+1} \, u^{\binom{n-i+1}{2}} v^{j+1} v^{\binom{i}{2}} u^{n-j+\binom{n-j}{2}} v^{\binom{j+1}{2}} \\
&= q^{-(j+1) \binom{n-j+1}{2} - (n-j) \binom{i}{2} - (j-i+1)\binom{n-i+1}{2}} u^{j-i+1} v^{j-i+1} u^{\binom{n-i+1}{2}} v^{i+\binom{i}{2}} u^{n-j+\binom{n-j}{2}} v^{\binom{j+1}{2}} \\
&= q^{-(j+1) \binom{n-j+1}{2} - (n-j) \binom{i}{2} - (j-i+1)\binom{n-i+1}{2} - \binom{j-i+1}{2}} z^{j-i+1} t_i t_j.
\end{align*}
Thus the generators of $T$ satisfy all the relations in $\Tpres$.
\end{proof}

By Lemma~\ref{lem:RelationsInT}, there is a surjection $\Tpres \twoheadrightarrow T$. 
To show that there are no more relations between the elements $x,y,z,t_i$ other than those given in $\Tpres$, we first make the following useful observation.

\begin{proposition}\label{prop:green}
$R = R^{(n)}$ has a $\ZZ^3$-grading given by 
\begin{align*}
    \deg(a_i) &= (\delta_{i,n}, 0, \delta_{i,n-1}- \delta_{i,n}) \\
    \deg (\a_i) & = (1-\delta_{i,n}, 1, -\delta_{i,n-1}+ \delta_{i,n}) ,
\end{align*}
where $\delta_{i,j}$ is the Kronecker delta.
Under this grading $T$ is $\NN^3$-graded with 
\begin{gather*}
\deg x = (1,0,0), \quad
\deg y = (n,n+1,0), \quad
\deg z = (1,1,0), \\
\deg t_i = \bigg(\binom{i}{2},\binom{i+1}{2},1\bigg), \quad 0 \leqslant i \leqslant n.
\end{gather*}
\end{proposition}
\begin{proof}
The relations of $R$ are homogeneous with respect to this grading on the generators, giving the first statement. The second statement is a consequence of Proposition~\ref{prop:GeneratorsOfT}.
\end{proof}

We often refer to the grading on $T$ given in Proposition~\ref{prop:green} as the the \emph{toric grading} of $T$. Since the relations in $\Tpres$ are a subset of those of $T$ (and ultimately we will show they satisfy precisely the same relations), this also gives an $\NN^3$-grading on $\Tpres$ which we also call the toric grading. One reason that this grading is useful is because the graded pieces of $\Tpres$ are at most one-dimensional.

\begin{proposition} \label{prop:OneDimensionalPieces}
The graded pieces of $\Tpres$ in  the toric grading satisfy $\dim (\Tpres)_{\vec n} \leqslant 1$ for all $\vec{n} \in \NN^3$.
\end{proposition}
\begin{proof}
Let 
\begin{align}
\label{Sone}
\mc S = \hspace{3pt}&\big\{ x^i z^k \mid i,k \geqslant 0 \big\} \cup 
\big\{ y^j z^k \mid j \geqslant 1, \hspace{3pt} k \geqslant 0 \big\} 
\\ 
\cup \hspace{4pt}
&\big\{ x^i z^k t_0^{\ell_0} \mid i,k,\ell_0 \geqslant 0 \big\} 
\cup
\big\{ y^j z^k t_n^{\ell_n} \mid j,k,\ell_n \geqslant 0 \big\} \label{Stwo}  \\
\cup \hspace{4pt}
&\bigcup_{m=1}^{n-1} \big\{z^k t_m^{\ell_m} \mid k \geqslant 0, \hspace{3pt} \ell_m \geqslant 1\big\} \cup
\bigcup_{m=0}^{n-1} \big\{z^k t_m^{\ell_m} t_{m+1}^{\ell_{m+1}} \mid k \geqslant 0, \hspace{3pt} \ell_m,\ell_{m+1} \geqslant 1\big\}. \label{Sthree}
\end{align}
We first show that $\mc S$ spans $\Tpres$. By applying the relations in $\Bbbk_\mathbf{q}[x,y,z,t_0, \dots, t_n]$, a monomial in $\Tpres$ can be written as $x^i y^j z^k t_0^{\ell_0} \dots t_n^{\ell_n}$, up to multiplication by a scalar, where all the exponents are non-negative. Moreover, repeated application of the relation $t_{i-1} t_{j+1} = \qb z^{j-i+1} t_i t_j$ to the flanking $t_i$ terms in this monomial shows that, up to scaling, we can write it as
\begin{align*}
x^i y^j z^k t_m^{\ell_m} t_{m+1}^{\ell_{m+1}}
\end{align*}
for some $0 \leqslant m \leqslant n-1$ and exponents $i,j,k,\ell$, some of which are possibly $0$. Additionally, the relation $xy = \qb z^{n+1}$ allows us to assume $i=0$ or $j=0$. 

 If $i=0=j$ then this element is in $\mc S$, so assume that $j = 0$ and $i>0$; the case $i=0$ and $j>0$ is similar. 
 We prove by induction on $i$ that the the monomial $x^i z^k t_m^{\ell_m} t_{m+1}^{\ell_{m+1}}$ lies in $\spann{\mc S}$. If $\ell_{m+1} \geqslant 1$, then the relation $xt_{m+1} = q^{\alpha_{m+1}} z^{m+1} t_{m}$ gives
\begin{align*}
x^i z^k t_m^{\ell_m} t_{m+1}^{\ell_{m+1}} = \qb x^{i-1} z^{k'} t_m^{\ell_m+1} t_{m+1}^{\ell_{m+1}-1}.
\end{align*}
By induction, this element lies in $\spann{\mc S}$. 
Now suppose $\ell_{m+1}=0$; we consider two subcases. 
If $m = 0$, or if $m \neq 0$ and $\ell_m = 0$, then this monomial is already an element of $\mc S$.
Otherwise $m \neq 0$ and $\ell_m \geqslant 1$, and we have
\begin{align*}
x^i z^k t_m^{\ell_m} = \qb x^{i-1} z^{k'} t_{m-1} t_m^{\ell_m - 1},
\end{align*}
and this element lies in $\spann{\mc S}$ by induction.

We now show that distinct elements of $\mc S$ have different toric degrees. By considering the third coordinate, it is clear that the sets in \eqref{Sone} are  disjoint from the other sets.
Moreover, 
\begin{align*}
\deg x^i z^k = (i+k,k,0), \quad \deg y^j z^{k'} = (jn+k',j(n+1)+k',0),
\end{align*}
and these are distinct for $i,k,k' \geqslant 0$ and $j \geqslant 1$. \\
\indent It remains to show that the toric degrees of elements in the remaining four sets in \eqref{Stwo} and \eqref{Sthree} are distinct. Call these four sets (a), (b), (c), and (d), in the order given; in principle we need to show that two distinct elements chosen from these four sets (possibly chosen from the same set) have distinct degrees. Most of these are routine, so we give only some of the proofs. 

For example, consider elements $x^i z^k t_0^\ell$ and $y^{j'} z^{k'} t_n^{\ell'}$ from sets (a) and (b), which have the same weight. 
Equating weights gives three equations:
\begin{align*}
i+k = j'n + k' + \ell' \binom{n}{2}, \quad k = j'(n+1) + k' + \ell' \binom{n+1}{2}, \quad \ell = \ell'.
\end{align*}
Subtracting the second equation from the first and using the third equation gives $i = j' - \ell n$, where the left hand side is non-negative and the right hand side is negative, a contradiction.

As another example, consider elements $z^k t_m^{\ell_m}$ and $z^{k'} t_{m'}^{\ell_{m'}} t_{m'+1}^{\ell_{m'+1}}$ from sets (c) and (d). Equating degrees gives three equations
\begin{align*}
k + \ell_m \binom{m}{2} &= k' + \ell_{m'} \binom{m'}{2} + \ell_{m'+1} \binom{m'+1}{2}, \\
k + \ell_m \binom{m+1}{2} &= k' + \ell_{m'} \binom{m'+1}{2} + \ell_{m'+1} \binom{m'+2}{2} \\
\ell_m &= \ell_{m'} + \ell_{m'+1}.
\end{align*}
Subtracting the first equation from the second gives
\begin{align*}
\ell_m m = \ell_{m'} m' + \ell_{m'+1} (m'+1) = (\ell_{m'} + \ell_{m'+1})m' + \ell_{m'+1} = \ell_m m' + \ell_{m'+1}.
\end{align*}
Therefore
\begin{align*}
\ell_m (m-m') = \ell_{m'+1}.
\end{align*}
Now, if $m' \geqslant m$ then $\ell_{m'+1} \leqslant 0$, which is a contradiction, so we must have $m > m'$. But then
\begin{align*}
\ell_{m'+1} = \ell_m (m-m') \geqslant \ell_m = \ell_{m'} + \ell_{m'+1} > \ell_{m'+1},
\end{align*}
which is also a contradiction. The remaining cases are similar and left to the reader. 
\end{proof}

With this result in hand, we are now able to show that $\Tpres$ gives a presentation of $T$:

\begin{theorem}\label{thm:TcongTpres}
There is an isomorphism $T \cong \Tpres$ of $\NN^3$-graded rings. Moreover, these rings are domains of GK-dimension 3.
\end{theorem}
\begin{proof}
We first note that, as $T$ is a subring of the domain $R$, it is also a domain. Now, temporarily denote the generators of $T$ by $\mathbf{x}, \mathbf{y}, \mathbf{z}, \mathbf{t_0}, \dots \mathbf{t_n}$. Let $\nu: \Tpres \to T$ be the obvious ``boldface'' map sending $x \mapsto \mathbf{x}$, etc., which exists by Lemma~\ref{lem:RelationsInT}.

Since we know that $\nu$ is surjective by Proposition~\ref{prop:GeneratorsOfT}, to show that it is an isomorphism it remains to show that it is injective. It suffices to show that its restriction to each (toric) graded piece $(\Tpres)_{\vec n}$ is injective; since each of these are at most one-dimensional, we only need to show that the restriction is not the zero map when $\dim (\Tpres)_{\vec n} = 1$. Indeed, if $s \in \mc S$ then $\nu(s)$ is nonzero, since it is a monomial in the (nonzero) generators of the domain $T$. Thus  $\nu$ is injective, and hence an isomorphism.

We will now make the identification $T = \Tpres$; in particular, we will drop the bold notations $\mathbf{x}$, etc., for generators of $T$. It remains to show that $\GKdim T = 3$.

Let $\Sigma = \{ \vec n \mid T_{\vec n} \neq 0\}$ be the sub-semigroup of $\NN^3$ on which $T$ is supported.  
If $v_1, \dots, v_k$ are nonzero $\NN^3$-homogeneous elements of $T$ and $V = \kk \cdot (v_1, \dots, v_k)$, we write 
\[\Sigma_V = \{ \deg(v_1),\dots, \deg (v_k)\} \subseteq \Sigma.\]
Since $\vec n \in \Sigma $ if and only if $\dim T_{\vec n} = 1$, we see that $\dim V = \#\Sigma_V$.

Now let $W$ be an arbitrary finite-dimensional subspace of $T$, and let $V $ be a finite-dimensional $\NN^3$-graded subspace containing $W$.
Then $\dim W^n \leq \dim V^n = \# (\Sigma_{V})^n$, and so 
 $\GKdim T$ is bounded above by the growth of $\Sigma$.
As $\Sigma$ is a sub-semigroup of $\NN^3$, which has cubic growth, $\GKdim T \leq 3$.
But by Proposition~\ref{prop:GeneratorsOfT}, $\GKdim T  \geq 3$.
Thus $\GKdim T = 3$.
\end{proof}

\subsection{The toric embedding of \texorpdfstring{$T$}{T}} \label{sec:ToricEmbedding}
We now construct a useful embedding of $T$ into a quantum 3-torus which is compatible with the toric grading. 
Let
\begin{align}
D = \frac{\Bbbk \langle \ba^{\pm 1},\bb^{\pm 1},\bs^{\pm 1} \rangle}{\left\langle 
\begin{array}{cc}
\bb\ba = q^{n+1} \ba\bb \\
\bs \ba = \ba\bs \\
\bs \bb = q^{-\binom{n+1}{2}} \bb\bs
\end{array} \right\rangle} \label{eqn:Q}
\end{align}
This has a $\ZZ^3$-grading given by declaring
\begin{align*}
\deg \ba = (1,0,0), \quad \deg \bb = (0,1,0), \quad \deg \bs = (0,0,1).
\end{align*}
Let $\Qgr[\ZZ^3](T)$ be the $\ZZ^3$-graded quotient ring of $T$ obtained by inverting all $\ZZ^3$-homogeneous elements.

We claim that $\Qgr[\ZZ^3](T) \cong D$ as graded rings; we will use this to construct an embedding $T \hookrightarrow D$. We first require a lemma:

\begin{lemma} \label{lem:QgrTOneDimensional}
We have $\dim \Qgr[\ZZ^3](T)_{\vec n} = 1$ for all $\vec{n} \in \ZZ^3$.
\end{lemma}
\begin{proof}
Let $D' = \Qgr[\ZZ^3](T) $.
We first show that $\dim D'_{\vec n} \leqslant 1$. Suppose that $u,v \in D'_{\vec n}$ are nonzero. By multiplying by a suitable homogeneous element $d \in T_{\vec{m}}$ (essentially clearing denominators), we obtain $ud, vd \in T_{\vec{n} + \vec{m}}$. Since $T_{\vec{n} + \vec{m}}$ is one-dimensional by Proposition~\ref{prop:OneDimensionalPieces}, it follows that $ud = \lambda vd$ for some $\lambda \in \Bbbk^\times$. Multiplying by $d^{-1}$ shows that $u = \lambda v$, so that $\dim D'_{\vec n} \leqslant 1$.
Since $D'$ contains the elements $x,t_0, yz^{-n}$  whose respective  degrees are $ (1,0,0)$,  $  (0,0,1)$, and $(0,1,0)$, the reverse inequality also holds.
\end{proof}

Now define a map
\begin{gather*}
\overline{\theta} : \Bbbk \langle \ba,\bb,\bs \rangle \to \Qgr[\ZZ^3](T),\\
\ba \mapsto x, \quad \bb \mapsto yz^{-n}, \quad \bs \mapsto t_0.
\end{gather*}
The following calculations shows that this descends to a map $\theta: D \to \Qgr[\ZZ^3](T)$:
\begin{gather*}
\overline{\theta}(\bb\ba) = yz^{-n} x = y q^{-n(n+1)} x z^{-n} = q^{(n+1)^2-n(n+1)} xy z^{-n} = q^{n+1} \theta(\ba\bb),\\
\overline{\theta}(\bs\ba) = t_0 x = x t_0 = \overline{\theta}(\ba\bs), \\
\overline{\theta}(\bs\bb) = t_0 y z^{-n} = q^{-\frac{1}{2}n (n+1)^2} yt_0z^{-n} = q^{-\frac{1}{2}n (n+1)^2 + \frac{1}{2} n^2(n+1)} y z^{-n} t_0 = q^{-\frac{1}{2}n (n+1)} \overline{\theta}(\bb\bs).
\end{gather*}
Moreover, it is straightforward to check that $\theta : D \to \Qgr[\ZZ^3](T)$ is a graded map.

\begin{proposition}
The map 
\begin{gather*}
\theta : D \to \Qgr[\ZZ^3](T), \qquad \ba \mapsto x, \quad \bb \mapsto yz^{-n}, \quad \bs \mapsto t_0,
\end{gather*}
is an isomorphism of $\ZZ^3$-graded rings.
\end{proposition}
\begin{proof}
We first show that $\theta$ surjects. Indeed, we have
\begin{align}
\begin{array}{rl}
\theta(\ba)\hspace{-7pt} &=  x, \\[5pt]
\theta(\ba^n \bb^{n+1})\hspace{-7pt} &= x^n (yz^{-n})^{n+1} = \qb y (xy z^{-(n+1)})^{n} = \qb y \\[5pt]
\theta(\ba \bb)\hspace{-7pt} &= xy z^{-n} = \qb z (xy z^{-(n+1)}) = \qb z \\[5pt]
\theta\big(\ba^{\binom{i}{2}} \bb^{\binom{i+1}{2}} \bs\big)\hspace{-7pt} &= x^{\binom{i}{2}} (y z^{-n})^{\binom{i+1}{2}} t_0
= \qb (xy)^{\binom{i}{2}} z^{-n\binom{i+1}{2}} y^i t_0 \\[5pt]
\hspace{-7pt} &= \qb z^{(n+1)\binom{i}{2}} z^{-n\binom{i+1}{2}} z^{in - \binom{i}{2}} t_i 
= \qb t_i.
\end{array} \label{ref:ThetaSurjects}
\end{align}
Since $D$ and $\Qgr[\ZZ^3](T)$ are $1$-dimensional in each graded piece (the latter by Lemma~\ref{lem:QgrTOneDimensional}), it follows that $\theta$ is in fact an isomorphism.
\end{proof}

Now define $\eta : T \to D$ to be the map $\eta = \restr{(\theta^{-1})}{T}$, which is necessarily an embedding of $T$ into $D$. From \eqref{ref:ThetaSurjects}, we have
\begin{align}
\eta(x) = \ba, \quad \eta(y) = \qb \ba^{n} \bb^{n+1}, \quad \eta(z) = \qb \ba \bb, \quad \eta(t_i) = \qb \ba^{\binom{i}{2}} \bb^{\binom{i+1}{2}} \bs. \label{eqn:eta}
\end{align}
The precise powers of $q$ in the above are not important for our needs, but can be calculated with a bit of care. We call the map $\eta : T \to D$ the \emph{toric embedding} of $T$, and we write
\begin{align*}
\Ttoric \coloneqq \Bbbk \langle \ba, \ba^n \bb^{n+1}, \ba \bb, \ba^{\binom{i}{2}} \bb^{\binom{i+1}{2}} \bs \mid 0 \leqslant i \leqslant n \rangle
\end{align*}
for the image of $\eta$.
These will play a major role in later calculations.

\section{A geometric resolution: the category}\label{QGR T}

We will now use the ring $T$ and methods of noncommutative algebraic geometry to construct a category $\sX_q$  which we will see is also a resolution of $B_q$.

In this section we consider $\T$ with its projective grading; that is, $\T$ is $\NN$-graded with $\T_d = R_{d \chi}$.
We will 
define a category $\rqgr \T$, a suitable localisation of the graded module category of $T$,  and give many of its  properties.  
Much of this is standard noncommutative geometry as in \cite{AZ}, but we need to give some proofs as the standing assumption of \cite{AZ} is that the degree 0 part of the ring is a finite module over a \emph{commutative} ground ring.

We begin, however, by recalling the basics of local cohomology.

\subsection{Local cohomology}\label{LOCALCOHDEF}
In this paper we will be primarily interested in local cohomology of graded $\T$-modules with respect to $\T_+ = \bigoplus_{s \geq 1} \T_s$.
However, we recall the basic definitions more generally. Thus, in this section let $\Omega$ be a ring, 
$I$ an ideal of $\Omega$, and $M$ an $\Omega$-module.
Recall that the {\em $j$th local cohomology of $M$ with respect to $I$} is
\[ H^j_I(M) = \lim_{s\to \infty} \Ext^j_\Omega(\Omega/I^s, M).\]
Equivalently, the $H^j_I(-)$ are the right derived functors of $H^0_I(-)  = \lim_{s \to \infty} \Hom_\Omega(\Omega/I^s, -)$.
Since $\Omega/I^s$ is an $\Omega$-bimodule, $H^j_I(M)$ is an $\Omega$-module.
Note that $H^0_I(M)$ is a submodule of $M$; we refer to it as the {\em $I$-torsion submodule of $M$}.

\begin{lemma}\label{lem:torsion}
Local cohomology modules are $I$-torsion.  That is, for all $j \in \NN$ and for any $\Omega$-module $M$, we have 
\[ H^0_I(H^j_I(M)) = H^j_I(M).\]
\end{lemma}
\begin{proof}
As $I^s (\Omega/I^s) = 0$,  we have $\Hom_\Omega(\Omega/I^s, M) I^s = 0$, and thus
$H^0_I(H^0_I(M)) = H^0_I(M)$.
The general result follows by dimension shifting.
\end{proof}

If $\Omega$ is commutative then it is well-known that local cohomology may be computed via \v{C}ech cohomology. The same is true if $\Omega$ is noncommutative and $I$ is normally generated. Indeed, suppose now that $I= \langle x_0, \dots, x_n \rangle$, where $x_0, \dots, x_n$ are normal elements of $\Omega$ and let $M$ be a right $\Omega$-module.
The {\em \v{C}ech complex} $\vC(x_0, \dots, x_n ; M)$ of $M$ with respect to $I$ is 
defined as follows, where we use the (slightly nonstandard) notation that $[n] = \{0, \dots, n\}$:
\begin{align*}
\vC(x_0, \dots, x_n ; M) : \quad 0 \to M \xrightarrow{d_0} \bigoplus_{i=0}^n M[x_i^{-1}] \xrightarrow{d_1} \cdots \xrightarrow{d_{s-1}} \bigoplus_{\substack{J \subseteq [n] \\ |J| = s}} M[x_J^{-1}] \xrightarrow{d_s} \cdots \xrightarrow{d_n} M[x_{[n]}^{-1}] \to 0.
\end{align*}
Here, we define  
\[M[x_J^{-1}] = M[x_j^{-1} \mid j \in J].\]
The differential is defined on $m_J \in M[x_J^{-1}]$ by
\begin{align*}
d_{|J|}(m_J) = \sum_{k \not\in J} (-1)^{o_J(k)} m_{J \cup \{k\}},
\end{align*}
where $o_J(k)$ denotes the number of elements of $J$ which are less than $k$, and $m_{J \cup \{k\}}$ is the image of $m_J \in M[x_J^{-1}]$ in $M[x_{J \cup \{k\}}^{-1}] = M[x_J^{-1}][x_k^{-1}]$.

\begin{proposition}\label{prop:LC}
If $x_0, \dots, x_n$ are normal elements of $\Omega$ and $I = \langle x_0, \dots, x_n \rangle$, then the cohomology of the \v{C}ech complex computes the local cohomology of $M$ with respect to $I$: that is,
\[ H^j(\vC(x_0, \dots, x_n ; M) \cong H^j_I(M)\]
for all $j \geq 0$.
\end{proposition}
\begin{proof}
The proof given in \cite[Theorem~5.1.19]{BS} for commutative rings also works in our situation.
\end{proof}

\subsection{Defining the category \texorpdfstring{$\sX_q$}{Xq}} \label{DEFCAT}
We now consider $\T$ with its projective grading.
Let  $M$ be a graded (right) $\T$-module.
The {\em torsion} submodule $\tau(M)$ of $M$ is the $\T_+$-torsion submodule; 
that is, $\tau(M)=H^0_{\T_+}(M)$.
If $M = \tau(M)$ then $M$ is {\em torsion}.
As usual, let $\rTors \T$ be the full subcategory of $\rGr\T$ consisting of torsion modules, and let $\rQgr \T = \rGr \T/\rTors \T$.
Let $\pi:\rGr \T \to \rQgr \T$ be the quotient functor.

More explicitly,  the objects  of $\rQgr \T$ are the same as the objects of $\rGr \T$, and the maps are defined by
\[ \Hom_{\rQgr \T}(\pi M,\pi N) = \hspace{-15pt} \varinjlim_{\substack{ M' \subseteq M\\ M/M' \text{ is torsion}}} \hspace{-15pt} \Hom_{\rGr \T}(M', N/\tau(N)).\]
If $M$ is finitely generated, then
\[ \Hom_{\rQgr \T}(\pi M,\pi N) = \varinjlim_{s} \Hom_{\rGr \T}(M_{\geq s}, N).\]
By definition, the elements of $\Hom_{\rGr \T}(M, N)$ are degree-preserving, and thus the same is true for elements of $\Hom_{\rQgr \T}(\sM, \sN)$ for $\sM, \sN \in \rQgr \T$.
We will also sometimes need
\[ \uHom_{\rQgr \T}(\sM, \sN) = \bigoplus_{s \in\ZZ} \Hom_{\rQgr \T}(\sM, \sN(s)).\]

Recall our convention that if $\Xyz$ is an abelian category, then $\xyz$ is the full subcategory of noetherian objects. 
Since $\T$ is noetherian, $\rqgr \T \simeq \rgr \T/\rtors \T$.

By construction, the functor $\pi: \rGr T \to \rQgr T$ is exact and preserves colimits.
By the adjoint functor theorem, therefore, $\pi$ has a right adjoint, which we denote $\omega$.
If $\sN\in \rQgr \T, s\in \ZZ$ then
\[ (\omega \sN)_s = \Hom_{\rGr \T}(\T, \omega \sN(s)) = \Hom_{\rQgr \T}(\pi \T, \sN(s))\]
and so 
\[ \omega \sN = \uHom_{\rQgr \T}(\pi \T, \sN).\]

If $\sN \in \rQgr \T$ then the $(B,\T)$-bimodule structure on $\T$ coming from the inclusion $B = \T_0 \subseteq \T$ gives $\Hom_{\rQgr \T}(\pi T, \sN)$ the structure of a $B$-module.
We denote the functor 
\[\H^0_{\rQgr \T}(-) \coloneqq \Hom_{\rQgr \T}(\pi T, -): \rQgr \T \to \rMod B\]
by $\phi_*$. As this functor is left exact and preserves limits, the adjoint functor theorem guarantees that it has a left adjoint $\phi^*$.
We thus have a map of noncommutative spaces $\phi: \rQgr T \to \rMod B$ given by 
\begin{equation*}
\begin{tikzcd}[>=stealth,ampersand replacement=\&,row sep=35pt,every label/.append style={font=\normalsize}] 
\rQgr T \arrow[d, "\,\phi_*" right,shift left=2] \\
\rMod B \arrow[u, "\phi^*\," left,shift left=2]
\end{tikzcd}
\end{equation*}
We now show that $\phi^* = \pi(- \otimes_B T)$.

\begin{proposition}\label{prop:adjoint}
Let $\phi_*  = \H^0_{\rQgr \T}(-)$ as defined above and let $\phi^* = \pi(- \otimes_B T)$.  
Then 
$(\phi^*, \phi_*)$ are an adjoint pair.
\end{proposition}
\begin{proof}
Let $M \in \rMod B$, $\mc N \in \rQgr \T$. 
We have
\begin{align*}
\Hom_{\rQgr T}(\phi^* M, \mc N) & = \Hom_{\rQgr T}(\pi(M \otimes_B T), \mc N) && \mbox{by definition}\\
& \cong \Hom_{\rGr T}(M \otimes_B T, \omega \mc N) && \mbox{as $(\pi, \omega)$ are an adjoint pair} \\
& = \uHom_T(M \otimes_B T, \omega \mc N)_0 && 
\\
& \cong \Hom_B(M, (\omega \mc N)_0) && \mbox{by hom-tensor adjointness.} 
\end{align*}
Now, $(\omega \mc N )_0 = \Hom_{\rGr T}(T, \omega \mc N) \cong \Hom_{\rQgr T}(\pi T, \mc N)$, 
using adjointness of $(\pi, \omega)$ again.
But this last  group is $\phi_* \mc N$ by definition, proving that 
\[ \Hom_{\rQgr T}(\phi^* M, \mc N) \cong \Hom_{B}(M, \phi_* \mc N).\]
The isomorphisms used in the proof are all natural in $\mc N$ and (where relevant) $M$.
\end{proof}

It is clear from the definition that $\phi^*(\rmod B) \subseteq \rqgr \T$.
We will see later that $\phi$ is ``proper'': that is, that $\phi_*(\rqgr \T) \subseteq \rmod B$, and that the map $\phi = (\phi^*, \phi_*): \rqgr T \to \rmod B$ is  a resolution of the singularity of $B$.

Similarly to Lemma~\ref{lem:MoritaEquiv}, we now show that the the map $\phi: \rQgr T \to \rMod B$ is an isomorphism above the ``nonsingular locus'' of $B$.

\begin{proposition} \label{prop:june}
Let $w \in \{x,y,z\}$.  
Then $\phi^*, \phi_*$ restrict to inverse equivalences between $\rQgr T[w^{-1}]$ and $\rMod B[w^{-1}]$.
\end{proposition}

\begin{proof}
We give the proof for $w=x$; the other cases are similar.
In $T[x^{-1}]$ we have $t_i = \qb x^{-1} z^i t_{i-1}$ for all $i \geq 1$, so
\[ T[x^{-1}] = B[x^{-1}][t_0; \theta] \]
for a suitable automorphism $\theta$ of $B[x^{-1}]$.

Let $N \in \rGr T[x^{-1}]$ and let $\sN = \pi N$.
Then $\omega \sN = N[t_0^{-1}]$, so 
\beq \label{eq:june2}
\phi_* \sN = N[t_0^{-1}]_0.
\eeq
As $N[t_0^{-1}]_0 \otimes_{B[x^{-1}]} T[x^{-1}] = N[t_0^{-1}]_{\geq 0}$,
which is isomorphic to $N[t_0^{-1}]$ in $\rQgr T$, we see that $\phi^*\phi_* \sN \cong \sN$.

If $M \in \rMod B[x^{-1}]$, it follows from \eqref{eq:june2} that
\[ \phi_* \phi^* M = (M \otimes_{B[x^{-1}]} T[x^{-1}])_0 = M. \qedhere\]
\end{proof}

We now define cohomology in $\rQgr \T$, that is, the derived functors of $\phi_* = \H^0_{\rQgr \T}(-)$.
This is similar to the first few pages of \cite[Section~7]{AZ}, and  we refer the reader there for proofs.

As $\rQgr \T$ is a quotient category of $\rGr \T$, it has enough injectives. In particular, we have:
\begin{lemma}\label{lem:piinj}
Let $I$ be a graded-injective $\T$-module. Then $\pi I$ is an injective object of $\rQgr \T$. As a consequence, $\rQgr \T$ has enough injectives.
\qed
\end{lemma}

It follows that if $\sN \in \rQgr \T$ then we can define $\Ext^j_{\rQgr \T}(\sN, -)$ to be the $j$th right derived functor of $\Hom_{\rQgr \T}(\sN, -)$. We similarly define $\uExt^j_{\rQgr \T}(\sN, -)$ as  right derived functors of $\uHom_{\rqgr T}(\sN, -)$, where we recall that $\uHom_{\rqgr T}(\sN, \sM) = \bigoplus_{s\in\ZZ} \Hom_{\rGr \T}(\sM, \sN(s))$.

We define the {\em $j$th cohomology} of $\sN$ to be
\[ \H^j_{\rQgr \T}(\sN) = \Ext^j_{\rQgr \T}(\pi \T, \sN),\]
and we also define
\[ \underline\H^j_{\rQgr \T}(\sN) = \uExt^j_{\rQgr \T}(\pi \T, \sN).\]
Thus $\underline\H^0_{\rQgr \T}(\sN) = \omega\sN$. 

We can relate cohomology $\underline\H^\bullet_{\rQgr \T}(-)$ on $\rQgr \T$ to local cohomology $H^{\bullet+1}_{\T_+}(-)$, as usual.
\begin{proposition}\label{prop:cohbasics}
Let $M, N$ be graded right $\T$-modules, with $M$ finitely generated.
\begin{enumerate}[{\normalfont (1)},topsep=0pt,itemsep=0pt,leftmargin=*]
\item If $j \geq 0$ then
\[ \uExt^j_{\rQgr \T}(\pi M, \pi N) \cong \lim_{d \to \infty} \uExt^j_\T(M_{\geq d}, N)\]
and 
\[ \underline\H^j_{\rQgr \T}(\pi N) \cong \lim_{d \to \infty}\uExt^j_\T(\T_{\geq d}, N). \]
\item There is an exact sequence
\[ 0 \to \tau(N) \to N \to \underline\H^0_{\rQgr \T}(\pi N) \to H^1_{\T_+}(N)\to 0,\]
and for $j \geq 1$,
\[\underline\H^j_{\rQgr \T}(\pi N) \cong H^{j+1}_{\T_+}( N).\]
\end{enumerate}
\end{proposition}
\begin{proof} This is \cite[Proposition~7.2]{AZ}, and the proof there works in our situation.
\end{proof}

\begin{corollary}\label{cor:tuesday}
Let $N$ be a graded right $\T$-module.
Then $N$ and $\omega(\pi N)$ are isomorphic in $\rQgr \T$.
\end{corollary}
\begin{proof}
This follows from Proposition~\ref{prop:cohbasics} and Lemma~\ref{lem:torsion}.
\end{proof}

\subsection{Properness}\label{PROPER}
The goal of this section is to prove the following ``relative Serre finiteness'' result for $\rqgr \T$.

\begin{theorem}\label{thm:SF}
{\rm(cf.~\cite[Theorem~III.5.2]{Ha}, \cite[Theorem~7.4]{AZ})}
Let $\sN \in \rqgr \T$.
Then:
\begin{enumerate}[{\normalfont (1)},topsep=0pt,itemsep=0pt,leftmargin=*]
\item for every $i \geq 0$, $\H^i_{\rQgr \T}(\sN)$ is a finitely generated right $B$-module;
\item for every $i \geq 1$, $\underline \H^i_{\rQgr \T}(\sN)$ is right bounded:
that is, there is $s_0 \in \NN$ so that $\H^i_{\rQgr \T}(\sN(s)) = 0$ for $s \geq s_0$.
\end{enumerate}
\end{theorem}
Theorem~\ref{thm:SF}(1) may be thought of as establishing the properness of $\phi$, as if $\psi: Y \to X$ is a proper morphism of noetherian schemes, then 
$\RR^i\psi_*(\coh Y) \subseteq \coh X$ for all $i$ by \cite[3.2.1]{EGA-III.1}.
Theorem~\ref{thm:SF}(2) shows that $\T(1) $ 
is relatively ample (over $B$) and so we may think of $\rqgr \T$ as being projective over $\rmod B$.

We will prove Theorem~\ref{thm:SF} by establishing the Artin--Zhang $\chi$ conditions for $\T$.
We give the relevant definitions here, generalised from \cite{AZ}. 
As in Definition~\ref{def2.2.1}, we say that 
a graded vector space $V$ is {\em left bounded} if $V_s = 0 $ for $s \ll 0$, {\em right bounded } if $V_s =0$ for $s \gg 0$, and {\em bounded} if it left and right bounded. 
Let $\Omega$ be a left and right noetherian $\NN$-graded $\kk$-algebra and let $i \in \NN$. 
Then $\Omega$ satisfies $\chi_i^\circ$ if $\uExt^j_\Omega(\Omega/\Omega_+, N)$ is bounded for all $ j \leqslant i$ and for all finitely generated graded $\Omega$-modules $N$. 
Further, $\Omega$ satisfies $\chi_i$ if for all finitely generated graded $\Omega$-modules $N$, for all $d$, and for all $j\leq i$, there exists an integer $s_0$ so that $\uExt_\Omega^j(\Omega/\Omega_{\geq s}, N)_{\geq d}$ is a finitely generated $\Omega$-module for all $s \geq s_0$.
If $\Omega$ satisfies $\chi_i^\circ$ for all $i$, we say that $\Omega$ satisfies $\chi^\circ$, and if $\Omega$ satisfies $\chi_i$ for all $i$, we say that $\Omega$ satisfies $\chi$.

We will see in this section that $T$ satisfies $\chi$, and as a consequence that Theorem~\ref{thm:SF} holds.
To do this, we generalise a number of results from \cite{AZ}, as to use \cite{AZ} directly we would need $T_0$ to be module-finite over a commutative ring.

We first give some general facts about bounds on $\uExt$ groups. Note that if $\Omega$ is an $\NN$-graded $\kk$-algebra and $N$ is a finitely generated graded $\Omega$-module, then all $\uExt^i_\Omega(\Omega/\Omega_{\geq s}, N)$ are automatically left bounded (cf.~\cite[Proposition~3.1(c)]{AZ}). 
If $\Omega$ satisfies $\chi_i$, then since the modules
$\uExt^j_\Omega(\Omega/\Omega_{\geq s}, N)$
are automatically $\Omega_+$-torsion, we see that
$\uExt^j_\Omega(\Omega/\Omega_{\geq s}, N)$ is bounded for $j \leq i$ and $s \gg 0$.

\begin{lemma}\label{lem4}
{\rm (cf.~\cite[Lemma~3.4]{AZ})}
Let $\Omega$ be an $\NN$-graded $\kk$-algebra and let $i \in \NN$.
Let $M, N$ be finitely generated graded $\Omega$-modules, with $M$ bounded with left bound $\ell$ and right bound $r$.
Let $\lambda$ be the minimum left bound and $\rho$ the maximum right bound of $\uExt^j_\Omega(\Omega/\Omega_+, N)$ for $j \leq i $ (note that we may have $|\lambda|, |\rho|= \infty$). 
Then for $j \leq i$, the left bound of $\uExt^j_\Omega(M, N)$ is at least $\lambda-r$, and the right bound is at most $\rho-\ell$.
\end{lemma}
\begin{proof}
The proof of \cite[Lemma~3.4]{AZ} does not use the standing assumption in \cite{AZ} that  $\Omega_0$ is finite-over-commutative, so may be used unchanged.
\end{proof}

\begin{proposition}\label{prop:chiimplieschi}
Let $\Omega$ be a left and right noetherian $\NN$-graded $\kk$-algebra that satisfies $\chi_i$. Then $\Omega$ satisfies $\chi_i^\circ$.
\end{proposition}
\begin{proof}
We follow the proof of \cite[Corollary~3.6(2)]{AZ}.
Let $N$ be a finitely generated graded $\Omega$-module.
By assumption, using the comment before Lemma~\ref{lem4}, there exists $s$ so that $\uExt^j_\Omega(\Omega/\Omega_{\geq s}, N)$ is bounded for $j \leq i$.

Fix $j \leq i$.
By induction, $\uExt^k_\Omega(\Omega/\Omega_+, N)$ is bounded for $k<j$
and so by Lemma~\ref{lem4} $\uExt_\Omega^{j-1}(\Omega_+/\Omega_{\geq s},N)$ is bounded.
From the exact sequence
\[ \uExt_\Omega^{j-1}(\Omega_+/\Omega_{\geq s}, N) \to \uExt_\Omega^j(\Omega/\Omega_+, N) \to \uExt_\Omega^j(\Omega/\Omega_{\geq s}, N),\]
we see that $\uExt_\Omega^j(\Omega/\Omega_+, N)$ is bounded.
\end{proof}

We now consider the $\chi$ condition for the graded ring
\[ S = \kk_{\mathbf{q}}[x,y,z, t_0, \dots, t_n]\]
defined in \eqref{eqn:PresentationForT}.
Here the grading on $S$ is such that 
 $x,y,z$ are in degree 0 and $\deg t_i = 1$ for all $i$.
As $\T \cong \Tpres$ by Theorem~\ref{thm:TcongTpres}, 
the obvious map $S \to \T$ is well-defined.

We define torsion $S$-modules as we did torsion $\T$-modules.

\begin{proposition}\label{prop1}
{\rm (cf.~\cite[Theorem~8.1]{AZ})}
Let $M, N$ be finitely generated graded $S$-modules such that $M$ is torsion.
Then $\uExt^i_S(M, N)$ is bounded for all $i$.
Further, $\uExt^i_S(S/S_{\geq s}, N)$ is a finitely generated $S$-module for all $i$ and $s$, and in particular $S$ satisfies $\chi$.
\end{proposition}

\begin{proof}
As $S$ is noetherian, we have $M S_{\geq s} = 0$ for some $s$.  
From the long exact $\uExt$ sequence, it is  enough to prove the first statement in the case that $MS_+ = 0$, i.e.~that $M$ is an $S/S_+$-module.

First assume that $M= S/S_+$ and $N=S$.
From the Koszul resolution of $S/S_+$ as an $S$-module we immediately compute that 
\[ \uExt^i_S(S/S_+, S) = \left\{
\begin{array}{cl}
S_0(n+1) & \text{if } i = n+1, \\
0 & \mbox{otherwise}.
\end{array}
\right.\]
It follows immediately that the proposition holds if $M_{S/S_+}$ and $N_S$ are projective.
 
Now let $N$ be an arbitrary $S$-module and consider a partial resolution $0 \to N' \to P \to N \to 0$ of $S$-modules, where $P$ is projective.
By induction on $\pd_S N$ and the long exact sequence in $\uExt$ we see that $\uExt^i_S(S/S_+, N) $ is bounded and finitely generated over $S$ for all $i$.  
 
 Using the long exact $\uExt$ sequence and the conclusion of the previous paragraph, we see that $\uExt^i_S(S/S_{\geq s}, N)$ is finitely generated and bounded for all $i, s$, and that if $M$  is an arbitrary finitely generated torsion module than $\uExt^i_S(M, N)$ is bounded for all $i$.
\end{proof}

Before we can prove that $T$ satisfies $\chi$, we will need to consider bounds on the $(S, T)$-bimodules $\Tor^S_i(S/S_{\geq s}, T)$, which have a $\ZZ$-grading induced from the gradings on $S$ and $T$.

\begin{proposition}\label{prop:tensor}
{\rm (cf.~\cite[Proposition~2.4]{AZ})}
Let $i,s \in \NN$.  
Then $\Tor_i^S(S/S_{\geq s}, T)$ is a finitely generated and bounded $T$-module, and if $i \geq 1$ the left bound is at least $s$.
\end{proposition}
\begin{proof}
That $\Tor_i^S(S/S_{\geq s}, T)$ is finitely generated over $T$ follows by resolving $S/S_{\geq s}$ by finitely generated graded $S$-modules.  Resolving $T$ as an $S$-module, we see that $\Tor_i^S(S/S_{\geq s}, T)$ is torsion, and thus bounded since it is finitely generated.
To compute the left bound, we follow \cite{AZ}.
For $i \geq 2$, $\Tor^S_i(S/S_{\geq s}, T) \cong \Tor^S_{i-1}(S_{\geq s}, T)$.
Taking a resolution of $T$ by finitely generated graded $S$-modules, we can see the left bound of $\Tor^S_{i-1}(S_{\geq s}, T)$ is at least $s$.
Further, $\Tor_1^S(S/S_{\geq s}, T) \subseteq S_{\geq s} \otimes_S T$, which has left bound $s$.
\end{proof}

\begin{proposition}\label{prop2}
{\rm (cf.~\cite[Lemma~8.2]{AZ})}
\begin{enumerate}[{\normalfont (1)},topsep=0pt,itemsep=0pt,leftmargin=*]
\item $T$ satisfies $\chi$ and $\chi^\circ$.  
\item Let $M, N$ be finitely generated graded $\T$-modules with $M$ torsion.
Then all $\uExt^i_T(M, N)$ are  bounded.
\end{enumerate}
\end{proposition}
\begin{proof}
We prove (1) following the proof of \cite[Lemma~8.2]{AZ}.
By induction, assume $T$ satisfies $\chi_{i-1}^\circ$.
Let $N \in \rgr T$ and $s\in \NN$, and consider the spectral sequence
\[ E^{p,r}_2 = \uExt^p_T(\Tor_r^S(S/S_{\geq s}, T), N) \Rightarrow \uExt^{p+r}_S(S/S_{\geq s}, N_S)\]
and the edge map
\beq \label{edge}
E^{i,0}_2 = \uExt^i_T(T/T_{\geq s}, N) \to \uExt^{i}_S(S/S_{\geq s}, N_S).
\eeq
Let $r \geq 1$.
By Proposition~\ref{prop:tensor}, the left bound of $\Tor_r^S(S/S_{\geq s}, T)$ goes to $\infty$ as $s \to \infty$.
By the induction hypothesis, $\uExt^{p-r}_T(T/T_+, N)$ is bounded and so by Lemma~\ref{lem4} the right bound of $E^{p-r,r}_2  = \uExt^{p-r}_T(\Tor_r^S(S/S_{\geq s}, T), N)$ goes to $-\infty$ as $s\to \infty$.
For $d\in \ZZ$, therefore, by taking $s \gg 0$ we may assume that $(E^{i-r,r}_2)_{\geq d} = 0$ and thus that \eqref{edge} is an isomorphism in degrees $\geq d$.

By Proposition~\ref{prop1},
\[ \uExt^i_T(T/T_{\geq s}, N)_{\geq d} \cong  \uExt^{i}_S(S/S_{\geq s}, N_S)_{\geq d}\]
is a finitely generated $S$-module and thus a finitely generated $T$-module.
Thus $T$ satisfies $\chi_i$, and by Proposition~\ref{prop:chiimplieschi} it satisfies $\chi_i^\circ$.

Part (2) is a consequence of the fact that $T$ satisfies $\chi^\circ$ and Lemma~\ref{lem4}.
\end{proof}

We need two subsidiary results before proving Serre finiteness for $T$.

\begin{proposition}
\label{prop5}
{\rm (cf.~\cite[Proposition~3.5]{AZ})}
Let $M, N \in \rgr \T$.
\begin{enumerate}[{\normalfont (1)},topsep=0pt,itemsep=0pt,leftmargin=*]
\item For all  $d\in\ZZ$ and $i \in \NN$ there is $s_0 \in \ZZ$ so that for all $s \geq s_0$ we have
\[ \uExt^i_\T(M/M_{\geq s_0}, N)_{\geq d}  \cong  \uExt^i_\T(M/M_{\geq s}, N)_{\geq d}. \]
\item For all $d\in\ZZ$ and $i \in \NN$ there is $s_0 \in \ZZ$ so that for all $s \geq s_0$ we have
\[ \uExt^i_\T(M_{\geq s_0}, N)_{\geq d} \cong \uExt^i_\T(M_{\geq s}, N)_{\geq d}.\]
\item For all $i \in \NN$, there is $\ell \in \ZZ$ so that for all $s \in \ZZ$ we have
\[ \uExt^i_\T(M, N)_{\geq \ell} \cong \uExt^i_\T(M_{\geq s}, N)_{\geq \ell}.\]
\end{enumerate}
\end{proposition}
\begin{proof}
By Proposition~\ref{prop2} $\T$ satisfies $\chi^\circ$.
Then the proof of \cite[Proposition~3.5]{AZ} goes through in our context, using Lemma~\ref{lem4} instead of \cite[Lemma~3.4]{AZ}.
\end{proof}

\begin{proposition}
\label{prop6}
{\rm (cf.~\cite[Proposition~3.13]{AZ})}
Let $N \in \rgr \T$.
\begin{enumerate}[{\normalfont (1)},topsep=0pt,itemsep=0pt,leftmargin=*]
\item For all $d \in \ZZ$ and for all $s \gg 0$, we have
\[\uHom_\T(\T_{\geq s}, N)_{\geq d} \cong (\omega\pi N)_{\geq d}.\]
\item There is $\ell \in \ZZ$ so that
\[ N_{\geq \ell} \cong (\omega\pi N)_{\geq \ell}.\]
\end{enumerate}
\end{proposition}
\begin{proof}
We  follow the proof in \cite{AZ}.
By definition
\beq\label{homdef} \omega\pi N = \lim_{s \to \infty}\uHom_\T(\T_{\geq s}, N).\eeq
Then (1) is \eqref{homdef} combined with Proposition~\ref{prop5}(2).
For (2), by Proposition~\ref{prop5}(3) there is $\ell$ so that $N_{\geq \ell} \cong \uHom_\T(\T_{\geq s}, N)_{\geq \ell}$ for all $s$.
But by (1) this is $(\omega\pi N)_{\geq \ell}$.
\end{proof}

We can now prove relative  Serre finiteness for $\rqgr \T$.

\begin{proof}[Proof of Theorem~\ref{thm:SF}]
Let $N \in \rgr \T$ so that $\pi N = \sN$, and let $i \in \NN$.
By Proposition~\ref{prop5}(1) there is $s_0$ so that for all $ j\leq i$ we have
\[ \underline\H^j_{\rQgr \T} (\sN) = \lim_{s\to \infty} \uExt^j_\T(T_{\geq s}, N) = \uExt^j_\T(T_{\geq s_0}, N).\]
Similarly to Proposition~\ref{prop:cohbasics}(2), if $j \geq 1$ then
\[ \underline\H^j_{\rQgr \T} (\sN) \cong \uExt^{j+1}_\T(T/T_{\geq s_0}, N),\]
and there is an exact sequence
\[ 0 \to \uHom_\T(T/T_{\geq s_0}, N) \to N \to \underline\H^0_{\rQgr \T} (\sN)\to \uExt^1_\T(T/T_{\geq s_0}, N) \to 0.\]
Since $T$ satisfies $\chi$, therefore, $\underline \H^i_{\rQgr \T}(\sN)_{\geq 0}$ is a finitely generated $T$-module, and thus $ \H^i_{\rQgr \T}(\sN)$ is  finitely generated  over $T_0 = B$.
Further, if $i \geq 1$ then $\underline \H^i_{\rQgr \T}(\sN)_{\geq 0} \cong H^{i+1}_{\T_+}(N)_{\geq 0}$ is torsion, and as it is finitely generated it must be right bounded.
\end{proof}


\subsection{The local cohomology of \texorpdfstring{$T$}{T}}\label{LOCALCOHCOMP}
It will clearly be useful to be able to compute some (local) cohomology groups.  
The goal of this section is to show that the \v{C}ech complex $\vC(t_0, \dots, t_n; \T)$ of  $\T$ is exact in degree 0. Recall from \eqref{eqn:eta} that we can embed $\T$ into the quantum polynomial ring $\Bbbk_{\qb}[\ba, \bb, \bs] \subseteq D$, where $D$ is defined in \eqref{eqn:Q}, by setting
\begin{align*}
x = \ba, \quad y = \qb \ba^n \bb^{n+1}, \quad z = \qb \ba\bb, \quad t_i = \qb  \ba^{\binom{i}{2}} \bb^{\binom{i+1}{2}} \bs,
\end{align*}
for suitable scalars $\qb$, and that this induces the toric grading on $\T$, where $\deg \ba^i \bb^j \bs^k = (i,j,k)$. In particular, the third component records the projective grading of $\T$.
Thus, any graded localisation of $\T$ can be viewed as living inside the quantum torus $D$, and the degree $0$ part (with respect to the third component of the grading, i.e.~the projective grading) of any such ring can be viewed as living inside the quantum division ring $Q(\Bbbk_{q^{n+1}}[\ba^{\pm 1}, \bb^{\pm 1}])$.
(Here we recall from \eqref{eqn:Q} that $\ba$ and $\bb$ $q^{n+1}$-commute.)
Therefore we have the following:

\begin{lemma}\label{lem:dag}
\begin{enumerate}[{\normalfont (1)},topsep=0pt,itemsep=0pt,wide=0pt]
\item For $0 \leqslant k \leqslant n$, we have
\begin{align}
\T[t_k^{-1}]_0 =
\left\{
\begin{array}{cl}
\Bbbk_{q^{n+1}} [x, t_1 t_0^{-1}] & \text{if } k = 0, \\
\Bbbk_{q^{n+1}}[ t_{k-1} t_k^{-1}, t_{k+1} t_k^{-1}] & \text{if } 1 \leqslant k \leqslant n-1, \\
\Bbbk_{q^{n+1}}[ t_{n-1} t_n^{-1}, y ] & \text{if } k = n.
\end{array}
\right. \label{zero}
\end{align}
If we identify $T$ with $\Ttoric$, then 
\begin{align*}
\T[t_k^{-1}]_0 = \Bbbk_{q^{n+1}}[\ba^{-(k-1)} \bb^{-k}, \ba^k \bb^{k+1}].
\end{align*}
In particular, the given generators $q^{n+1}$-commute.
\item For $0 \leqslant k \leqslant n-1$, we have
\begin{align*}
\T[t_k^{-1}, t_{k+1}^{-1}]_0 = 
\left\{
\begin{array}{cl}
\Bbbk_{q^{n+1}}[x,(t_0 t_1^{-1})^{\pm 1}] & \text{if } k=0, \\
\Bbbk_{q^{n+1}}[t_{k-1} t_k^{-1}, (t_k t_{k+1}^{-1})^{\pm 1}] & \text{if } 1 \leqslant k \leqslant n.
\end{array}
\right.
\end{align*}
If we identify $T$ with $\Ttoric$, then 
\beq\label{one}
\T[t_k^{-1}, t_{k+1}^{-1}]_0 = \Bbbk_{q^{n+1}}\Big[\ba^{-(k-1)} \bb^{-k}, \big(\ba^k \bb^{k+1}\big)^{\pm 1}\Big].
\eeq
\item For $0 \leqslant j < k \leqslant n$ with $k-j \geqslant 2$ we have
\beq\label{two}
\T[t_j^{-1}, t_k^{-1}]_0 = \Bbbk_{q^{n+1}}[\ba^{\pm 1}, \bb^{\pm 1}].
\eeq
\item If $J \subseteq \{t_0, \dots, t_n\}$ with $\#J \geq 3$, then  
\beq \label{three}
\T[t_J^{-1}] = \Bbbk_{q^{n+1}}[\ba^{\pm 1}, \bb^{\pm 1}].
\eeq
\end{enumerate}
\end{lemma}
\begin{proof}
\begin{enumerate}[{\normalfont (1)},topsep=0pt,itemsep=0pt,wide=0pt]
\item It is clear that any element of $T[t_k^{-1}]_0$ is a linear combination of products of $x,y,z$ and elements of the form $t_j t_k^{-1}$, where $0 \leqslant j \leqslant n$. In each case, the claimed generators of $T[t_k^{-1}]_0$ are monomials in $\ba, \bb,$ and $\bs$, and hence generate a subalgebra of $D$ which is isomorphic to a quantum plane. We write this as $\kk_{q^\bullet}[x, t_1 t_0^{-1}] \subseteq \T[t_0^{-1}]_0$, for example, and have similar containments in the other cases.\\
\mbox{}\hspace{12pt}For the reverse inclusion, we consider the three cases separately, beginning with $k=0$; the case $k=n$ is similar and hence omitted. The relations in $\T$ allow us to write
\begin{align*}
z = q^{-\alpha_1} x t_1 t_0^{-1}, \quad y = q^{\beta_0} z^n t_1 t_0^{-1},
\end{align*}
the first of which clearly lies in $\Bbbk_\qb[x, t_1 t_0^{-1}]$, from which it follows that the same is true of $y$. Moreover, for $2 \leqslant j \leqslant n$, the relation $t_0 t_j = \qb z^{j-1} t_1 t_{j-1}$
implies that
\begin{align*}
t_j t_0^{-1} = q^\bullet z^{j-1} (t_1 t_0^{-1}) (t_{j-1} t_0^{-1}),
\end{align*}
and induction shows that $t_j t_0^{-1} \in \Bbbk_\qb[x, t_1 t_0^{-1}]$.\\
\mbox{}\hspace{12pt}Now suppose that $1 \leqslant k \leqslant n$. The relation $t_{k-1} t_{k+1} = \qb z t_k^2$ in $\T$ gives
\begin{align*}
z = q^\bullet (t_{i-1} t_i^{-1}) (t_{i+1} t_i^{-1})
\end{align*}
in $\T[t_k^{-1}]$, which lies in $\Bbbk_\qb[ t_{k-1} t_k^{-1}, t_{k+1} t_k^{-1}]$. We also have
\begin{align*}
x = \qb z^k t_{k-1} t_k^{-1}, \quad y = \qb z^{n-k} t_{k+1} t_k^{-1}
\end{align*}
from which it follows that $x,y \in \Bbbk_\qb[ t_{k-1} t_k^{-1}, t_{k+1} t_k^{-1}]$. Moreover, if $j \geqslant k+2$, then $t_k t_j = \qb z^{j-k-1} t_{k+1} t_{j-1}$, and hence
\begin{align*}
t_j t_k^{-1} = q^{\bullet} z^{j-k-1} (t_{k+1} t_k^{-1}) (t_{j-1} t_k^{-1}),
\end{align*}
and induction shows that $t_j t_k^{-1} \in \Bbbk_\qb[ t_{k-1} t_k^{-1}, t_{k+1} t_k^{-1}]$. Similarly, one can show that $t_j t_k^{-1} \in\Bbbk_\qb[ t_{k-1} t_k^{-1}, t_{k+1} t_k^{-1}]$ when $j \leqslant k-2$, establishing the claimed equality.\\
\mbox{}\hspace{12pt}Finally, it is straightforward to check that the given generators are identified with $\ba^{-(k-1)} \bb^{-k}$ and $\ba^k \bb^{k+1}$ under the toric embedding, which gives the final equality. Direct calculation shows that these elements $q^{n+1}$-commute:
\begin{align*}
\ba^k \bb^{k+1} \cdot \ba^{-(k-1)} b^{-k} = q^{-(n+1)((k+1)(k-1)-k^2)} \ba^{-(k-1)} b^{-k} \cdot \ba^k \bb^{k+1} = q^{n+1} \ba^{-(k-1)} b^{-k} \cdot \ba^k \bb^{k+1}.
\end{align*}
\item It is clear that, for all values of $k$, the right hand side is contained in the left hand side; we now consider the reverse inclusion. Suppose that $1 \leqslant k \leqslant n-1$. By part (1), we have 
\begin{align*}
T[t_k^{-1}]_0 = \Bbbk_{q^{n+1}}[t_{k-1} t_k^{-1}, t_{k+1} t_k^{-1}] \quad \text{and} \quad T[t_{k+1}^{-1}]_0 = \Bbbk_{q^{n+1}}[t_{k} t_{k+1}^{-1}, t_{k+2} t_{k+1}^{-1}],
\end{align*}
so that $T[t_k^{-1},t_{k+1}^{-1}]_0$ is the subring generated by these elements:
\begin{align*}
T[t_k^{-1},t_{k+1}^{-1}]_0 = \Bbbk \langle t_{k-1} t_k^{-1}, (t_{k} t_{k+1}^{-1})^{\pm 1}, t_{k+2} t_{k+1}^{-1} \rangle,
\end{align*}
It therefore suffices to show that $t_{k+2} t_{k+1}^{-1}$ lies in $\Bbbk_{q^{n+1}} [t_{k-1} t_k^{-1}, (t_{k} t_{k+1}^{-1})^{\pm 1}]$. 
From the relations in $\T$, we have
\begin{align*}
t_{k-1} t_{k+1}^3 = \qb z t_k^2 t_{k+1}^2 = \qb t_k^3 t_{k+2},
\end{align*}
and hence
\begin{align*}
t_{k+2} t_{k+1}^{-1} = \qb (t_{k-1} t_k^{-1}) (t_{k+1} t_k^{-1})^2 \in \Bbbk_{q^{n+1}} [t_{k-1} t_k^{-1}, (t_{k} t_{k+1}^{-1})^{\pm 1}],
\end{align*}
as desired. \\
\mbox{}\hspace{12pt}The cases $k=0$ and $k=n$ are similar; we prove $k=0$.
In this case, it suffices to show that $t_2t_1^{-1}$ lies in $\Bbbk_{q^{n+1}}[x,(t_0 t_1^{-1})^{\pm 1}]$. To see this, note that we have
\begin{align*}
xt_1^3 = \qb z t_0 t_1^2 = \qb t_0^2 t_2,
\end{align*}
and hence 
\begin{align*}
t_2 t_1^{-1} = \qb x (t_1 t_0^{-1})^2 \in \Bbbk_{q^{n+1}}[x,(t_0 t_1^{-1})^{\pm 1}].
\end{align*}
\item It suffices to show that $\ba^{\pm 1}, \bb^{\pm 1} \in \T[t_j^{-1}, t_k^{-1}]_0$ when $0 \leqslant j < k \leqslant n$ with $k-j \geqslant 2$. Since $x \in \T$, we have $\ba \in \T[t_j^{-1}, t_k^{-1}]_0$. We also have
\begin{align*}
\T[t_j^{-1}, t_k^{-1}]_0 \ni z^{k-1} (t_{k-1} t_k^{-1}) = \qb \ba^{k-1} \bb^{k-1} \ba^{-(k-1)} \bb^{-k} = \qb \bb^{-1},
\end{align*}
so that $\bb^{-1} \in \T[t_j^{-1}, t_k^{-1}]_0$. Additionally, we have
\begin{align*}
\T[t_j^{-1}, t_k^{-1}]_0 \ni (t_{j+1} t_j^{-1})^{k-1} (t_{k-1} t_k^{-1})^j = \qb (\ba^j \bb^{j+1})^{k-1} (\ba^{-(k-1)} \bb^{-k})^j = \qb \bb^{k-j-1},
\end{align*}
where $k-j-1 \geqslant 1$. Multiplying by a suitable power of $\bb^{-1}$ shows that $\bb \in \T[t_j^{-1}, t_k^{-1}]_0$. Finally,
\begin{align*}
\T[t_j^{-1}, t_k^{-1}]_0 \ni \ba^{k-2} \bb^{k} (t_{k-1} t_k^{-1}) = \qb \ba^{k-2} \bb^{k} \, \ba^{-(k-1)} \bb^{-k} = \qb \ba^{-1},
\end{align*}
so that $\ba^{-1} \in \T[t_j^{-1}, t_k^{-1}]_0$. Thus $\T[t_j^{-1}, t_{k}^{-1}]_0 = \Bbbk_{q^{n+1}}[\ba^{\pm 1}, \bb^{\pm 1}]$.
\item This follows from (3). \qedhere
\end{enumerate}
\end{proof}

The computations in Lemma~\ref{lem:dag} show that $\rqgr \T$ is ``locally nonsingular'' and so it is natural to expect that it has  finite homological dimension.
We will prove that this is the case in Section~\ref{LOCAL}.

Let $\vC^\bullet = \vC(t_0, \dots, t_n; \T)_0$.
Note that for all $(i,j) \in \ZZ^2$ and $s \in \NN$ we have
\begin{align*}
\dim \vC^s_{(i,j)} \leq \binom{n+1}{s}< \infty
\end{align*} 
from Lemma~\ref{lem:dag}(4).

\begin{lemma}\label{bigtoriclemma}
For all $(i,j) \in \ZZ^2$, the Euler characteristic of $\vC^\bullet_{(i,j)}$ is $0$: that is,
\[ \sum_{k \geq 0} (-1)^k \dim \vC^k_{(i,j)} = 0.\]
 \end{lemma}
\begin{proof}
 If $X \subseteq \kk_{q^{n+1}}[\ba^\pm, \bb^\pm]$ is bigraded, define $\supp(X) = \{(i,j) \mid \ba^i \bb^j \in X\}.$

We now describe a partition of $\ZZ^2$.
Let $ Q_{\,\RomanOne} = \NN\times \NN \ssm \{(0,0)\}$ be the ``first quadrant'' and let $Q_{\MakeUppercase{\,\RomanThree}} = \{ (i,j) \in \ZZ \times \ZZ \mid i,j < 0\}$  be the ``third quadrant''.
For $1 \leqslant k \leqslant n$ let
\[ U_k = \Big\{ (i,j)\in  Q_{\,\RomanOne} \;\Big|\; \frac{k+1}{k} < \frac{j}{i} \leqslant \frac{k}{k-1} \Big\} \]
and also define
\[ W = \Big\{ (i,j) \in Q_{\,\RomanOne} \;\Big|\; \frac{j}{i} \leqslant \frac{n+1}{n} \Big\} \cup \{0\}.\]
Moreover, let 
\[ L_0 = \{ (i,j) \in \ZZ^2 \mid i \geqslant 0, \; j <0\}\]
and, for $1 \leqslant k \leqslant n-1$, let
\[ L_k = \Big\{ (i,j) \in Q_{\,\RomanThree} \;\Big|\; \frac{k+1}{k} \leqslant \frac{j}{i} < \frac{k}{k-1} \Big\}.\]
Finally, set
\[ Z = \Big\{ (i,j) \in \ZZ^2 \;\Big|\; i < 0, \; \frac{j}{i} < \frac{n}{n-1}\Big\}.\]

We observe that as $\T_0 = \kk \ang{\ba, \ba^n \bb^{n+1}, \ba \bb} \subseteq \kk_{q^{n+1}}[\ba, \bb]$ we have
\[ \supp \T_0 = \NN\cdot (1,0) + \NN \cdot (n, n+1) + \NN \cdot (1,1) = W.\]
Let $0 \leq k \leq n$.
From  \eqref{zero} we calculate:
\[ \supp \T[t_k^{-1}]_0 = \NN \cdot (-(k-1), -k) + \NN \cdot(k, k+1) = 
L_{k-1} \sqcup \dots \sqcup L_0 \sqcup W \sqcup U_n \sqcup \dots \sqcup U_{k+1}.
\]
Likewise, \eqref{one} tells us that, if $0 \leq k \leq n-1$, then 
\[ 
\supp \T[t_k^{-1}, t_{k+1}^{-1}]_0 = \NN \cdot (-(k-1), -k) + \ZZ \cdot (k, k+1) = 
L_{k} \sqcup \dots \sqcup L_0 \sqcup W \sqcup U_n \sqcup \dots \sqcup U_{k+1}.
\]
Combining these computations with \eqref{two} and \eqref{three} we obtain that
\[ \dim \vC^s_{(i,j)} = \left\{
\begin{array}{cl}
\binom{n+1}{s} & \mbox{if $s \geq 3$ or  $(i,j) \in W$,} \\[3pt]
\binom{n}{2} & \mbox{if $s=2$ \text{ and } $(i,j) \in Z$,} \\[3pt]
\binom{n}{2} + k & \mbox{if $s=2$ \text{ and } $(i,j) \in U_k$,} \\[3pt]
\binom{n}{2} + n-k & \mbox{if $s=2$ \text{ and } $(i,j) \in L_k$,} \\[3pt]
k & \mbox{if $s=1$ \text{ and } $(i,j) \in U_k$,} \\[3pt]
n-k & \mbox{if $s=1$ \text{ and } $(i,j) \in L_k$,} \\[3pt]
0 & \mbox{otherwise.}
\end{array} \right.\]
The result follows.
\end{proof}

We can now compute the degree 0 local cohomology of $T$.

\begin{proposition}\label{prop:localcohomology}
The \v{C}ech complex $\vC^\bullet$ is exact, and $H^s_{\T_+}(\T)_0 = 0$ for all $s$.
\end{proposition}

\begin{proof}
We first observe that $\dim H^s_{\T_+}(\T)_0$ does not depend on $q$.
To see this, note that for any $J \subseteq [n]$ 
and any $(i,j) \in \ZZ^2$, $\T[t_J^{-1}]_{(i,j)}$ is either $0$ or $\kk \ba^i \bb^j$, and which occurs is independent of $q$. Thus if we define 
\[ m^{ij}_J = \left\{\begin{array}{cl} \ba^i \bb^j & 
\mbox{if $\T[x_J^{-1}]_{(i,j)} \neq 0$,}
\\
0 & \mbox{otherwise,} \end{array}\right.\]
then $d_s: \vC^s \to \vC^{s+1}$ is given by
$d_s(m_J^{(i,j)}) = \sum_{k \not\in J} (-1)^{o_J(k)} m^{(i,j)}_{J \cup \{k\}},$
which again does not depend on $q$.
It thus suffices to compute  $H^s_{\T_+}(\T)_0$ if $q=1$.

If $q=1$ then $\widetilde{X}= \Proj(\T)$ is the minimal resolution of $X = \kk^2/C_{n+1}$, constructed via quiver GIT as explained in \cite{NCCR}. 
The singularity of $X$ is rational by \cite{Artin66}; that is, if $\pi: \widetilde{X} \to X$ is the morphism associated to the resolution then 
\[ H^s(\widetilde{X}, \ms O_{\widetilde X}) = \RR^s \pi_* \ms O_{\widetilde X} =\left\{\begin{array}{cl} 0 & \text{if } s \geq 1, \\ \ms O_X & \text{if } s=0.\end{array}\right.\]
But $\coh \widetilde X \simeq \rqgr \T$ by \cite[\href{https://stacks.math.columbia.edu/tag/0BXD}{Tag 0BXD}]{stacks-project} so by Proposition~\ref{prop:cohbasics}(2) for $s \geq 2$ we have 
$H^s_{\T_+}(\T)_0 \cong H^{s-1}(\widetilde X, \ms O_{\widetilde X})  $,
giving the result for $s \geq 2$.

As $\T$ is a domain, $H^0_{\T_+}(\T) = 0$.
The complex $\vC^\bullet$ is thus exact except possibly at $s=1$.
Applying Lemma~\ref{bigtoriclemma} we see that  $\vC^\bullet$ is exact, as claimed. 
\end{proof}

\begin{corollary}\label{cor:local2}
Let   $\sN$ be a finitely generated $\ZZ^3$-graded right 
$\T$-module
that is locally free in the sense that $\sN[t_i^{-1}]_0$ is  a free $\T[t_i^{-1}]_0$-module for all $0 \leq i \leq n$. 
Then $H^s_{\T_+}(\sN)_0 = 0$ for all $s \geqslant 3$.
\end{corollary}
\begin{proof}
Since $\sN$ is locally free, the \v{C}ech complex for $\sN$ is isomorphic to that of $\T^{\oplus d}$ in cohomological degrees $\geqslant 2$, where $d$ is the uniform rank of $\sN$. 
Thus if $s \geq 3$ then $H^s_{\T_+}(\sN)_0 \cong H^s_{\T_+}(\T)^d_0$.  
The result follows from Proposition~\ref{prop:localcohomology}.
\end{proof}


\section{A system for localising \texorpdfstring{$\T$}{T}-modules}\label{LOCAL}

Let $q \in \Bbbk^\times$ be arbitrary and let $T= T_q$.
In this section we prove two main results: first, that  many computations in $\rqgr \T$ can be done locally, i.e.~on ``charts''; and, second, that $\rqgr \T$ has finite homological dimension.
This is done by establishing an equivalence of categories between $\rqgr \T$ and the category of coherent sheaves over a certain noncommutative  ringed space $\X$.

\subsection{A noncommutative toric variety equivalent to \texorpdfstring{$\rqgr \T$}{qgr-T}} \label{TORICVAR}
By \eqref{eqn:eta}, we can view $\T \subseteq D$ via the toric embedding with, up to scalars,
\[ x= \ba, \quad y = \ba^n\bb^{n+1}, \quad z=\ba\bb,\quad t_i = \ba^{\binom{i}{2}}\bb^{\binom{i+1}{2}} \bs.\]
Any two  of the generators $\{x,y, z, t_0, \dots, t_n\}$ $q^\bullet$-commute, and as a result the generators of $\T$ are all normal.

Define a topological space  $\X$ to be the set of $\ZZ^3$-graded prime ideals of $\T$ which do not contain the irrelevant ideal $\T_+ = (t_0, \dots, t_n)$, in  the Zariski topology.

\begin{lemma}\label{lem:X}
The elements of $\X$ are:
\begin{align*}
0 &\\
\mf{l}_{0}& = \langle y,z, t_1, \dots, t_{n} \rangle\\
\mf{l}_{i} & = \langle x,y,z, t_0, \dots, t_{i-2}, t_{i+1}, \dots t_{n} \rangle \quad \mbox{for $1 \leq i\leq n$}\\
\mf{l}_{{n+1}} & = \langle x,z, t_0, \dots, t_{n-1} \rangle\\
\mf{p}_{i} & = \langle x,y,z, t_0, \dots, t_{i-1}, t_{i+1}, \dots t_{n} \rangle \quad \mbox{for $0 \leq i\leq n$.}
\end{align*}
In particular, $\X$ is a two-dimensional topological space. 
\end{lemma}

\begin{remark}
The reason for the notation here is that we may define cyclic $T$-modules $L_i = T/\mf{l}_{i}$ and $P_i = T/\mf{p}_{i}$.  We will see in Section~\ref{INTTHEORY} that these are, respectively, line and point modules, and that they have the (noncommutative) intersection theory of an $\mathbb{A}_n$ Dynkin diagram.
\end{remark}

\begin{proof}[Proof of Lemma~\ref{lem:X}]
We view $T$ as a subring of $D$ via the toric embedding of Section \ref{sec:ToricEmbedding}. For any $\vec{n} \in \NN^3$ we have $\dim \T_{\vec n} \leq 1$ from the inclusion $\T \subseteq D$.
Thus an ideal  of $\T$ is $\ZZ^3$-graded exactly when it is generated by monomials in $\ba, \bb, \bs$; but an element of $\T$ is a monomial in $\ba,\bb,\bs$ exactly when it is a monomial in $x,y,z, t_0, \dots, t_n$.  
As these elements are all normal, an ideal $\mf n $ of $ \T$ is $\ZZ^3$-graded and prime if and only if $\mf n$ is generated by some subset of $\{x,y,z,t_0, \dots, t_n\}$.

Let $0 \ne \mf p  \in \X$.
If $z \not \in \mf p$ then from the relation $xy = q^\bullet z^{n+1}$ we see that $x,y \not\in \mf p$.
But now some $t_i \in \mf p$ as $\mf p \neq 0$, and then from the relations $t_{i-1}y = q^\bullet t_i z^{n-i+1}$ and $t_{i+1} x = q^\bullet t_i z^{i+1}$ and the assumption that $z\not\in \mf p$ we have $t_0, \dots, t_n \in \mf p$. This contradicts the assumption that $\mf p \not\supseteq \T_+$. It follows that $z \in \mf p$.  Now for all $1 \leq i < j \leq n-1$ we have $t_{i-1} t_{j+1} = q^\bullet z^{j-i+1} t_i t_j \in \mf p$, so if $0 \leq i\leq j \leq n$ with $j \geq i+2$ then $t_i \in \mf p$ or $t_j \in \mf p$.

By hypothesis there is some $t_i \not\in \mf p$; let $i$ be the minimal such index.
We therefore have $\mf p \supseteq \langle  t_0, \dots, t_{i-1}, t_{i+2}, \dots, t_n \rangle$ by the observation in the previous paragraph.
If $i \leq n-1$ then $t_iy = q^\bullet t_{i+1} z^{n-i} \in \mf p$ so $y \in \mf p$ as $t_i \not\in \mf p$.
Likewise if $i \geq 1$ then $x \in \mf p$.
Thus if $1\leq i \leq n-1$ then $\mf p \supseteq \mf l_{i+1}$ and 
\[
\mf p = \left\{\begin{array}{cl} \mf l_{i+1} & \mbox{if $t_{i+1}\not\in \mf p$,}\\
   \mf p_i & \mbox{if $t_{i+1} \in \mf p$.}
\end{array}\right.   
\]
If $i = 0$ then $\mf p \supseteq \langle z,y, t_2, \dots, t_n \rangle$.
Since $x t_1 = q^\bullet zt_0 \in \mf p$, either $x \in \mf p$ or $t_1 \in \mf p$.
If $x \not \in \mf p$ then $\mf p = \mf l_0$.
If $t_1 \not \in \mf p$ then $\mf p = \mf l_1$.
If $x, t_1 \in \mf p$ then $\mf p = \mf p_0$.
Finally, if $i = n$ then $\mf p \supseteq \mf l_{n+1} $ and
\[
\mf p = \left\{\begin{array}{cl} \mf l_{n+1} & \mbox{if $y\not\in \mf p$,}\\
   \mf p_n & \mbox{if $y \in \mf p$.}
\end{array}\right. \qedhere
\]
\end{proof}

We now make $\X$ into a ringed space by defining a sheaf $\ms O_\X$ of noncommutative rings on $\X$.  Our conventions for ringed spaces will in general follow \cite{Ha}, except, of course, that our rings are noncommutative.

For notational purposes let $t_{n+1} = y$ and $t_{-1}=x$.
Let $I \subseteq \{ -1,\dots, n+1\}$, and let
\[ V(I) = V\bigg(\prod_{i\in I} t_i\bigg) = \bigg\{ \mf p \in \mf X \;\;\bigg|\; \prod_{i \in I} t_i \in \mf p \bigg\}.\]
The  subsets $V(I) $ generate a topology on $\X$, where the closed subsets of $\X$ are the (necessarily finite) intersections of the $V(I)$.
If $I \subseteq \{-1, \dots, n+1\}$, define $U_I$ to be $\X \ssm V(I)$.  
We refer to the $U_I$ with $I \neq \varnothing$ as {\em principal open subsets of $\X$} or as {\em principal open sets}.
We will omit curly braces from our notation and write, for example,  $U_i$ for $U_{\{i\}}$ and $U_{ij}$ for $U_{\{i,j\}}$.

Let $U_I$ be a principal open set, and define the {\em coordinate ring of $U_I$} to be
\[ \ms O_\X(U_I) = \T[t_i^{-1} \mid i \in I]_0.\]
Recall that the $\ms O_\X(U_I)$ were computed in Lemma~\ref{lem:dag}.  
In particular, $\ms O_\X(U_i) \cong \kk_{q^{n+1}}[u, v]$ for all $i$.

As an intersection of principal open sets is again principal, any open subset of $\X$ has a cover by principal open sets.
If $V$ is any non-principal open subset of $\X$ write $V = \bigcup U_J$ where the $U_J$ are principal, and put
\[ \ms O_\X(V) = \bigcap_J \ms O_\X(U_J).\]
As all the $\ms O_\X(U_J)$ are subalgebras of $D_0 = \kk_{q^{n+1}}[\ba^{\pm 1}, \bb^{\pm 1}]$, this intersection makes sense.

Now if $V\subseteq U$ are open subsets of $\X$ then there is an inclusion 
$\ms O_\X(U) \subseteq \ms O_\X(V) \subseteq D_0$
so that $\ms O_\X(V)$ is an $\ms O_\X(U)$-module.
These inclusions make $\ms O_\X$ a sheaf of noncommutative algebras on $\X$, and we will refer to $\X$ as a {\em noncommutative ringed space}.
If we need to keep track of $q$, we will refer to $\X$ as $\X_q$.

\begin{remark}\label{rem:Kawamata NC scheme}
In \cite{kawamata2024, kawamata2025}, Kawamata gives a definition which generalises our construction of $\X$: he shows how to construct  a {\em noncommutative scheme} or  {\em NC scheme} from a collection of associative algebras parameterised by a poset.  
The main result of \cite{kawamata2025} is that a titlting bundle $F^\circ$ on a commutative scheme $X$ deforms over an Artin local base to a tilting bundle $F$ on any NC scheme $\X$ that deforms $X$, and that $\End_{\X}(F)$ is derived equivalent to $\X$. 

If we vary $q$, then   $\X_q$ is a NC scheme over $\Spec \kk[q, q^{-1}]$, in Kawamata's sense.  
Further, at $q=1$ we recover the minimal resolution of the Kleinian singularity $X_1$, regarded as a toric variety.
Thus it is not surprising that $\X_q \simeq \rqgr T_q$ has a tilting bundle, although note that \cite{kawamata2025} only considers formal deformations over an Artin local base.  
\end{remark}

We make the expected definition of an $\ms O_\X$-module:  a (left or right) { \em $\ms O_\X$-module} is a sheaf $\ms M$ of abelian groups on $\X$ so that each $\ms M(U)$ is a (left or right) $\ms O_\X(U)$-module and if $V\subseteq U$ are open subsets of $\X$ then the restriction maps $\ms M(U) \to \ms M(V)$ are compatible with the restriction $\ms O_\X(U) \subseteq \ms O_\X(V)$.
Let $\ms M, \ms N$ be $\ms O_\X$-modules.
The notion of a morphism $\ms M \to \ms N$ of $\ms O_\X$-modules and thus of $\Hom_\X(\ms M, \ms N)$ is the obvious generalisation of the standard commutative definition; cf.~\cite[p.~104]{Ha}.
The category of left, respectively right, $\ms O_\X$-modules is called $\ms O_\X \lMOD$, respectively $\rMOD \ms O_\X$.
By restriction, each open subset $U$ of $\X$ is also a noncommutative ringed space and we thus have categories  $\ms O_U \lMOD$ and $\rMOD \ms O_U$ of $\ms O_U$-modules, where $\ms O_U$ is our notation for $\ms O_\X|_U$.

Note that $\X = \bigcup_{i = 0}^n U_i$.
Concretely, we have
\[
U_i = \{ 0, \mf l_i, \mf l_{i+1}, \mf p_i\} \quad \mbox{for } 0\leq i \leq n, \quad U_{-1} = \{0, \mf l_0\}, \quad U_{n+1} = \{0, \mf l_{n+1}\}.
\]

We will see that principal open sets  are affine in the appropriate sense, and use them to define (quasi)coherent left and right $\ms O_\X$-modules.
We will only give definitions on the right as those for the left are the same, mutatis mutandis.

Let us consider $U_i$ for $0 \leq i \leq n$.
Define a functor $\ \widetilde{{}}: \rMod \ms O_X(U_i) \to \rMOD U_i$ as follows:  if $V$ is an open subset of $U_i$ and $M \in \rMod\ms O_X(U_i)$ then set $\widetilde{M}(V) := M \otimes_{\ms O_X(U_i)} \ms O_{U_i}(V)$.
Now, $\widetilde{M}$ is certainly a presheaf, but note that $U_i$ has no proper covers: if $\{ V_j\} $ is a set of open subsets of $U_i$ with $\bigcup_j V_j = U_i$ then some $V_j = U_i$.
This means the glueing and vanishing sheaf conditions are trivial and any presheaf on $U_i$ is automatically a sheaf. 
In particular, $\widetilde M$ is a sheaf.

Fix $I$ and let $U = U_I$ and $C = \ms O_X(U_I)$.
Given $C$-modules $M, N$ and  a $C$-module homomorphism $\rho: M \to N $, we define $\widetilde \rho: \widetilde M \to \widetilde N$ as follows.
Let $V $ be an open subset of $U$.  Then 
$\widetilde \rho(V) : \widetilde M(V) \to \widetilde N(V)$  is given by $1 \otimes\rho$ for left modules and $\rho \otimes 1$ for right modules.
Note that the functor $\ \widetilde {}\ $ is exact and fully faithful, with left inverse given by $\ms M \mapsto \ms M(U)$.

Summarising the construction and discussion of $\ \widetilde{} \ $ we have:
\begin{proposition}\label{newprop:two}
Let $U_I$ be a principal open set.
For  every right $\ms O_\X(U_I)$-module $M$, the assignment $V \mapsto M\otimes_{\ms O_\X(U_I)} \ms O_\X(V)$ defines a sheaf $\widetilde M$ of right $\ms O_{U_I}$-modules.  This assignment induces an equivalence of categories $\rMod \ms O_\X(U_I) \to \rMod \ms O_{U_I}$, with quasi-inverse given by $\ms M \mapsto \ms M(U_I)$.  Analogous results hold on the left. \qed
\end{proposition}

We say that a (left or right) $\ms O_\X$-module $\ms M$ is {\em quasicoherent} if for all $i$ there is an $\ms O_X(U_i)$-module $M_{U_i}$ so that $\ms M|_{U_i} \cong \widetilde{M_{U_i}}$, and {\em coherent} if the $M_{U_i}$ are finitely generated. 
The full subset of $\ms O_\X \lMOD$  of quasicoherent (respectively, coherent) sheaves is denoted $\ms O_\X \lMod$ (respectively $\ms O_\X \lmod$).
Likewise on the right we define $\rMod \ms O_\X$ and $\rmod \ms O_\X$.
These categories are abelian given the standard notions of  sheaf kernel (which is the same as presheaf kernel) and sheaf cokernel of a morphism of sheaves on $\X$.  
Likewise, the image of a morphism of (quasi)coherent $\ms O_\X$-modules is (quasi)coherent.

We now define a functor $\bpi:  \rGr \T \to  \rMod \ms O_\X $ as follows.
Given a graded $\T$-module $M$, the sheaf $\bpi M$ is defined by $\bpi M(U_i) = M [t_i^{-1}]_0$ for $0 \leq i \leq n$.
For any open $V\subseteq U_i$, define $\bpi M(V) = \widetilde{M(U_i)}(V)$ as in Proposition~\ref{newprop:two}.
For any other open set $V$, use the glueing axiom of sheaves to define $\bpi M(V)$.
By construction, $\bpi M$ is quasicoherent.

A homomorphism $f: M \to N$ of graded $\T$-modules induces a morphism of $\ms O_\X$-modules $\bpi M \to \bpi N$ in the obvious way.

\begin{lemma}\label{newlem:three}
$\bpi M = 0$ if and only if $M$ is $\T_+$-torsion.
\end{lemma}
\begin{proof}
By Proposition~\ref{newprop:two}, $\bpi M = 0$ if and only if  $\bpi M(U_i) = 0$ for all $0 \leq i \leq n$.
By definition, this is if and only if all $M[t_i^{-1}]_0 = 0$, or if and only if all $M[t_i^{-1}] =0$ as the rings $\T[t_i^{-1}]$ are strongly graded.
This is if and only if for all homogeneous $m \in M$ there is $r \in \NN$ so that  $m t_i^r =0$ for all $0 \leq i \leq n$.
But the chain of ideals $\{ (t_0^r, \dots, t_n^r)\}_{r \geq 0}$ is cofinal with $\{\T_{\geq r}\}$, so this last condition occurs if and only if $M$ is torsion.
\end{proof}

\begin{corollary}\label{cor:five}
There is an induced functor $ \opi:  \rQgr \T  \to \rMod \ms O_\X $.
\qed
\end{corollary}

We will show that $\opi$ is an equivalence of categories.  
To construct the inverse functor to $\opi$, we need a bit more notation.  

First, if $\ms M$ is a right $\ms O_\X$-module and $\ms N$ is an $\ms O_\X$-bimodule, we can define the right $\ms O_\X$-module $\ms M \otimes_\X \ms N$ in the usual way, as the sheafification of the presheaf given by $\ms M \otimes_\X \ms N(U_I) = \ms M(U_I)\otimes_{\ms O_\X(U_I)} \ms N(U_I)$.
For any $d\in \ZZ$, let $\ms O(d) := \bpi \T(d)$.  
This is an invertible $\ms O_\X$-bimodule under the obvious convention, as $\T(d)$ is an invertible $\T$-bimodule.
The inverse of $\ms O(d)$ is $\ms O(-d)$.
Note that $\{t_0, \dots, t_n\} \subseteq H^0(\X, \ms O(1))$.
Given a  right $\ms O_\X$-module $\ms M$, let $\ms M(d) := \ms M \otimes_\X \ms O(d)$. 

Now, given $\ms M \in \rMod \ms O_\X$, define
\[ \bGamma_*(\ms M) := \bigoplus_{d\in \ZZ} H^0(\X, \ms M(d)),\]
where $H^0(\X, -) $ is the global section functor.
We claim that $\bGamma_* (\ms M)$ is a (graded) right $\T$-module.
To see this, let $m \in \bGamma_*(\ms M)_d$ and let $t \in \T_e$.
Now, $m$ is determined by a compatible choice of
\[ m_i \in \ms M(U_i) \otimes_{\ms O_\X(U_i)} \ms O(d)(U_i) = \ms M(U_i) \otimes_{\ms O(U_i)} \T[t_i^{-1}]_d\]
for $0 \leq i \leq n$.
Thus we may write $m_i = b_i \otimes r_i$, where $r_i \in \T[t_i^{-1}]_d$.
Then 
$mt = (m_i t)_i$,
where $m_i t = b_i \otimes r_i t \in \ms M(U_i) \otimes_{\ms O(U_i)} \T[t_i^{-1}]_{d+e}$.
The compatibility relations on the $m_i$ ensure that the $m_it$ are also compatible, that is that $m t \in H^0(\X, \ms M(d+e))$, as needed.

Before proving that $\bGamma_*$ induces a quasi-inverse equivalence to $\opi$, we need a lemma.

\begin{lemma}\label{newlem:foo}
{\rm (cf.~\cite[Lemma~II.5.14]{Ha})}
Let $\ms M \in  \rMod \ms O_\X$, and let $m_i \in \ms M(U_i)$ for some $0 \leq i \leq n$.
Then for some $d$, $m_i t_i^d \in \ms M(d)(U_i)$ extends to a global section of $\ms M(d)$.
\end{lemma}

\begin{proof}
We will be careful with the restriction maps:  for $0 \leq i, j \leq n$, let $\varphi_{ij}: \ms M(U_i) \to \ms M(U_{ij})$ be the restriction.
For $j \neq i$, define $m_j  = \varphi_{ij}(m_i)$.
As $\ms M(U_{ij}) = \ms M(U_j)[(t_i t_j^{-1})^{-1}]$, for some $d_j \in \NN$ and $n_j \in \ms M(U_j)$ we have
$m_j = n_j(t_i t_j^{-1})^{-d_j}$, so $m_j t_i^{d_j} = n_j t_j^{d_j} \in \ms M(d_j)(U_j)$.
Taking $d = \max_j d_j$ we see that $m_j t_i^d \in \ms M(d)(U_j)$ for all $j$.
These choices are compatible, so give a global section of $\ms M(d)$ extending $m_i t_i^d$.
\end{proof}

\begin{proposition}\label{newprop:inverse1}
Let $\ms M \in \rMod \ms O_\X$.
Then $\bpi \bGamma_* \ms M \cong \ms M$.
\end{proposition}

\begin{proof}
We first define a map $\psi:  \bpi \bGamma_* (\ms M) \to \ms M$.
By the local nature of a sheaf and by Proposition~\ref{newprop:two}, it is enough to define $\psi$ on the open sets $U_i$ for $0 \leq i \leq n$:  that is, to define
\[\psi_i:  \bpi \bGamma_* (\ms M)(U_i) \to \ms M(U_i)\]
for all $0 \leq i \leq n$.
Now, an element of $ \bpi \bGamma_* (\ms M)(U_i)$  has the form $m t_i^{-d}$ where $d \in \ZZ$ and $m \in \ms M(d)(U_i)$.
But $t_i^{-d} \in \ms O(-d)(U_i) = \T[t_i^{-1}]_{-d}$, so we may regard $m t_i^{-d}$ as an element of $\ms M(d) \otimes_\X \ms O(-d)(U_i)$.
Let $\psi_i(m)$ be the image of $m$ under the natural isomorphism
\[ \ms M(d) \otimes_\X \ms O(-d) \to \ms M.\]
Note that $\psi_i(m) \neq 0$ if $m \neq 0$.
Thus $\psi$ is injective. That $\psi $ is surjective follows from Lemma~\ref{newlem:foo}.
\end{proof}

We now use local cohomology to show that $\opi$ and $\bGamma_*$ are quasi-inverse equivalences.  
Recall that $\pi: \rGr \T \to \rQgr \T$ is the natural (quotient) functor.

\begin{proposition}\label{newprop:inverse2}
Let $M \in \rGr \T$. 
Then in $\rGr \T$ we have $\bGamma_* \bpi M \cong \underline \H^0_{\rQgr \T}(\pi M)$,
which is isomorphic to $M$ in $\rQgr \T$.
\end{proposition}

\begin{proof}
That $M$ and $\underline\H^0_{\rQgr \T}(\pi M)$ are isomorphic in $\rQgr \T$ is Corollary~\ref{cor:tuesday}.

To complete the proof, let $\ms M  =\bpi M$.
It suffices to show that we have an isomoprhism $H^0(\X, \ms M) \cong \varinjlim \Hom_{\rGr \T}(\T_{\geq d},  M)$.
Now $H^0(\X, \ms M)$ may be identified with
\beq\label{newbiff}
\Big\{ m = (m_i) \in \bigoplus_{i=0}^n M[t_i^{-1}]_0 \;\Big|\; \varphi_{ij}(m_i) = \varphi_{ji}(m_j) \mbox{ for all } i,j\Big\},
\eeq
where  $\varphi_{ij}: \ms M(U_i) \to \ms M(U_{ij})$ is  restriction.

Let $f \in \Hom_{\rGr \T} (\T_{\geq d}, M)$.
Then $n_i := f(t_i^d) \in M$ for all $i$, and $n_i t_i^{-d} = n_j t_j^{-d}$
as elements of $M[t_i^{-1}, t_j^{-1}]_0 = \ms M(U_{ij})$.
Thus the restrictions of $n_i t_i^{-d}$ and $n_j t_j^{-d}$ give the same element of $\ms M(U_{ij})$, and so the $n_i t_i^{-d}$ glue to give an element of $H^0(\X, \ms M)$.
This element is clearly nonzero if $f \neq 0$.
In other words, we have constructed an injective homomorphism  $\varinjlim  \Hom_{\rGr \T}(\T_{\geq d}, M) \to H^0(\X, \ms M)$.

To see this map is surjective, let $m = (m_i) \in H^0(\X, \ms M)$ as in \eqref{newbiff}.
There are $d \in \NN$ and $n_i \in M_d$ so that we can write $m_i = n_i t_i^{-d}$ for all $i$.
This induces a degree 0 homomorphism $\lambda_i: t_i^d \T \to M$ defined by $\lambda_i(t_i^d r) = n_i r$.

The compatibility \eqref{newbiff} between the $m_i$ means that $n_i t_i^{-d} = n_j t_j^{-d} \in M[t_i^{-1}, t_j^{-1}]_0$, or that 
$n^i t_j^d  = q^\bullet n_j t_i^d$, where $t_i^d t_j^d = q^\bullet t_j^d t_i^d$.
Thus on $t_i^d \T \cap t_j^d \T = t_i^d t_j^d \T$ we have
\[ \lambda_i(t_i^d t_j^d r) = n_i t_j^d r = q^\bullet n_j t_i^d r = \lambda_j(q^\bullet t_j^d t_i^d r) = \lambda_j (t_i^d t_j^d r),\]
and $\lambda_i, \lambda_j$ agree.
They thus glue to give a well-defined map $\lambda\in \Hom_{\rGr \T}( \sum_j t_j^d \T, M)$ defined by $\lambda(\sum t_j^d r_j) = \sum n_j r_j $.
To finish, we just note that $\sum_j t_j^d T \supseteq T_{\geq r}$ for appropriate $r$.
\end{proof}

Combining Propositions~\ref{newprop:inverse1} and \ref{newprop:inverse2}. we have proved:
\begin{theorem}\label{newthm:catequiv}
The functors  $\bGamma_*$ and $\opi$  induce quasi-inverse equivalences between $\rMod \ms O_\X $ and $\rQgr \T$, which restrict to equivalences between $\rmod \ms O_\X$ and $\rqgr \T$.
\end{theorem}
\begin{proof}
All that is left is to prove that if $M \in \rgr \T$, then the image of $\bGamma_* \bpi M$ is in $\rqgr \T$.  
This follows from Proposition~\ref{prop6}, as $M_{\geq \ell}$ is finitely generated for all $\ell$.
\end{proof}

\subsection{Finite homological dimension}

We now turn to proving that $\rqgr \T$ (or equivalently, $\rmod \ms O_\X$) has finite homological dimension.  We will use a version of the standard Grothendieck spectral sequence relating $\Ext$ and $\shExt$;  parts of the commutative proof go through with little change, but there are some important differences. 

As $\rQgr \T \simeq \rMod \ms O_\X$, by Lemma~\ref{lem:piinj} $\rMod \ms O_\X$ has enough injectives.
We wish to understand 
 injective objects in $\rMod \ms O_\X$ in more detail.
 Here we have:
 
\begin{lemma}\label{lem:injstalk}
Let $\ms I$ be a (right or left) $\ms O_\X$-module.  
Then $\ms I$ is injective if and only if the stalk $\ms I_{\mf p}$ is an injective $\ms O_{\X, \mf p}$-module for all $\mf p \in \X$.
\end{lemma}

Before proving Lemma~\ref{lem:injstalk}, we need an auxiliary lemma.
The stalk $\ms O_{\X, \mf p}$ is not a local ring, except in the weak sense of being $\ZZ^3$-graded local.
Instead we have:

\begin{lemma}\label{lem:local}
Let $\mf p \in\X$. Then
\[ \ms O_{\X, \mf p} \cong \left\{\begin{array}{cl}
\kk_{q^\bullet}[u^\pm, v^\pm] & \text{if } \mf p = 0, \\
\kk_{q^\bullet}[u^\pm, v] & \text{if } \mf p = \mf l_i, \\
\kk_{q^\bullet}[u, v] & \text{if } \mf p = \mf p_i.
\end{array}\right.\]
Further, there is an open subset $V_{\mf p}$ of $\X$ so that $\ms O_{\X, \mf p} = \ms O_\X(V_{\mf p})$.
\end{lemma}
\begin{proof}
By definition, the stalk $\ms O_{\X, \mf p}$ is 
\[ \varinjlim_{\{ V \mid \mf p \in V, V \text{open}\}} \hspace{-15pt} \ms O_\X(V) = \sum_{V}\ms O_\X(V).\]
(Here recall that each $\ms O_\X(V)$ is a subalgebra of $\kk_{q^{n+1}}[\ba^\pm, \bb^\pm]$ with the restriction maps all being inclusions.)
But as $\X$ is a finite space, the intersection
\[ V_{\mf p} := \bigcap \;\{ V \mid \mf p \in V,\; V \mbox{open}\}\]
is itself open, so $\ms O_{\X, \mf p} = \ms O_\X(V_{\mf p})$.
Further, we have
\begin{align*}
V_0 & = \{ 0\},\\
V_{\mf l_i} &= \{ \mf l_i, 0\}, \\
V_{\mf p_i} &= \{ \mf p_i, \mf l_i, \mf l_{i+1}, 0\},
\end{align*}
and the result follows, using the computations in Lemma~\ref{lem:dag}.
\end{proof}

\begin{proof}[Proof of Lemma~\ref{lem:injstalk}]
$(\Leftarrow)$
Let $\ms M, \ms N$ be  $\ms O_\X$-modules and let $\gamma: \ms M \to \ms N$ and 
 $\iota: \ms M \to \ms I$ be  morphisms, where $\gamma$ is monic.  
Now,  $\gamma$ being monic means by definition that its sheaf kernel is trivial; that is, that $\gamma$ is injective on stalks.
So for all $\mf p \in \X$ we have $\gamma_{\mf p}: \ms M_{\mf p} \hra \ms N_{\mf p}$ and $\iota_{\mf p}: \ms M_{\mf p} \to \ms I_{\mf p}$.  
As $\ms I$ has injective stalks, $\iota_{\mf p}$ extends to create a commutative diagram
\begin{equation*}
\begin{tikzcd}[>=stealth,ampersand replacement=\&,row sep=30pt,every label/.append style={font=\normalsize},column sep=35pt] 
\ms M_{\mf p} \arrow[d, "\iota_{\mf p}" {swap}] \arrow[r, "\gamma_{\mf p}"] \& \ms N_{\mf p} \arrow[dl, "j_{\mf p}"] \\
\ms I_{\mf p} \&
\end{tikzcd}
\end{equation*}
The $j_{\mf p}$ glue to give a morphism $j: \ms N \to \ms I$ of $\ms O_\X$-modules so that
\begin{equation*}
\begin{tikzcd}[>=stealth,ampersand replacement=\&,row sep=30pt,every label/.append style={font=\normalsize},column sep=35pt] 
\ms M \arrow[d, "\iota" {swap}] \arrow[r, "\gamma"] \& \ms N \arrow[dl, "j"] \\
\ms I \&
\end{tikzcd}
\end{equation*}
commutes, as needed.

$(\Rightarrow)$
The restriction of an injective sheaf to an open subset is injective, so $\ms I|_{V_{\mf p}}$ is injective for all $\mf p \in \X$.  
The result follows from Lemma~\ref{lem:local}.
\end{proof}

\begin{remark}\label{rem:inj}
Note that by the stalk criterion of Lemma~\ref{lem:injstalk}, if $\ms I \in \rMod \ms O_\X$ is injective, then $\ms I$ is also injective considered as an object of $\rMOD \ms{O}_\X$.
\end{remark}

As in the commutative setting, if $\ms M, \ms N$ are quasicoherent sheaves on $\X$ then the presheaf $U \mapsto \Hom_{\ms O_U}(\ms M|_U, \ms N|_U)$ is a sheaf (of abelian groups), which we call $\shHom_\X(\ms M, \ms N)$.
As usual, we have $\Hom_{\X}(\ms M, \ms N) = H^0(\X, \shHom_\X(\ms M, \ms N))$.

For our Grothendieck spectral sequence, we will need that $\shHom_\X(\ms M, -)$ takes injective objects to $\Gamma$-acyclic objects.
\begin{proposition}\label{prop:flasque}
    Let $I$ be a graded injective $\T$-modiule and let $ \ms M $ be a quasicoherent $\ms O_\X$-module. 
    Then $\shHom_\X(\ms M, \bpi I) $ is a flasque sheaf on $\X$.
\end{proposition}
\begin{proof}
Without loss of generality $I$ is graded uniform and $\T_+$-torsionfree.  
There is thus a finitely generated graded uniform submodule $H$ of $I$ and a non-irrelevant graded prime ideal $\mf q$ of $\T$ that is the assassinator of $H$, by \cite[Lemma~5.26]{GW}.
(Note that here we are using the projective $\ZZ$-grading of $\T$ where the $t_i$ have degree 1 and $x,y,z$ have degree 0.)
Let $\ms I= \bpi I$ and $\ms H = \bpi H$. 

Let $V \subseteq U$ be open subsets of $\X$.
We need to show:
\beq\label{NTS}
\Hom_U(\ms M|_U, \ms I|_U) \to \Hom_V(\ms M|_V, \ms I|_V) \text{ is surjective,}
\eeq
where the map is induced by restriction. Writing $Z = \X \ssm V$, we have $Z = V(\mf n)$ where $\mf n$ is a non-irrelevant  $\ZZ^3$-graded ideal of $\T$. We make and prove two claims.

\textbf{Claim 1:} If $\mf q \supseteq \mf n$ then $\ms I|_V = 0$.

To see this, it suffices to show that $\ms I|_{V'} =0$ where $V'$ is an principal open subset of $V$. There is a $\ZZ^3$-homogeneous $w \in \mf n$ so that $w$ is invertible in $\ms O_\X(V')$ and, because $w\in \mf q$ and $H \mf q= 0$, we have $\ms H(V') = 0$ by the equivalence from Proposition~\ref{newprop:two} between $\rMod\ms O_{V'}$ and $\rMod \ms O(V')$. Thus $\ms I(V') = 0$ as $\ms H$ is essential in $\ms I$, establishing the claim.

\textbf{Claim 2:} If $\mf q \not\supseteq \mf n$ and if $\ms F \in \rMod \ms O_\X$ is supported on $Z$ in the sense that
\[ 
Z \supseteq \{ \mf p \in \X \mid \ms F_{\mf p} \neq 0\},
\]
then $\Hom_U(\ms F|_U, \ms I|_U) = 0$. 

To establish this, we may assume without loss of generality that $U$ is principal, by the local nature of a sheaf, so we seek to show that $\Hom_U(\ms F(U), \ms I(U))= 0$, using Proposition~\ref{newprop:two}. Now, $\mf q$ corresponds to a prime ideal $P$ of $C = \ms O_\X(U)$, which is the assassinator of the essential submodule $\ms H(U) $ of $\ms I(U)$, and $\mf n$ corresponds to a $\ZZ^2$-graded ideal $N$ of $C$ with $P \not\supseteq N$. 
Also, $\ms F(U)$ is $w$-torsion for any $\ZZ^2$-homogenous $w \in N$. Choose some such $w \in N \ssm P$ and let $\gamma \in \Hom_C(\ms F(U), \ms I(U)) = \Hom_U(\ms F(U), \ms I(U))$. If  $\gamma \neq 0$ then $P$ is the annihilator of the nonzero module $\im \gamma \cap \ms H(U)$.  
This contradicts the fact that $\im \gamma$ is also $w$-torsion and $w \not\in P$, and thus $\Hom_U(\ms F|_U, \ms I|_U) = 0$, as claimed.

We return to the proof of the proposition.  We must establish \eqref{NTS}.
By Claim 1, we may assume without loss of generality that $\mf n \not\subseteq \mf q$.
We will show that \eqref{NTS} is an equality.

Let $j: V\to U$ be the inclusion and let $\ms K$ be the kernel of the natural map $\ms I|_U \to j_* (\ms I|_V)$.
Now $\ms K$ is supported on $Z$ by definition and by Claim 2 the inclusion $\ms K \subseteq \ms I|_U$ must be the zero map; that is, $\ms K =0$.
So $\ms I|_U \subseteq j_*(\ms I|_V)$ and as $\ms I|_U$ is injective the two must be equal.

Now consider the equivalent restriction map $\varphi: \ms M|_U \to j_*(\ms M|_V)$ for $\ms M$.
Let $\ms A = \ker \varphi$ and $\ms B = \coker \varphi$.
Both are supported on $Z$ and so, by Claim 2, $\Hom_U(\ms A, \ms I|_U) = \Hom_U(\ms B, \ms I|_U) = 0$.
Using injectivity of $\ms I|_U$ we see that
\[ \Hom_U(\ms M|_U, \ms I|_U) = \Hom_U(j_*(\ms M|_V), \ms I|_U) =  \Hom_U(j_*(\ms M|_V), j_*(\ms I|_V)).\]
This last is easily seen to be equal to $\Hom_V(\ms M|_V, \ms I|_V)$.
\end{proof}

As $\rMod \X$ has enough injectives, we may construct the right derived functors of $\Hom_\X(\ms M, -)$ and $\shHom_\X(\ms M, -)$, which we call $\Ext^i_\X(\ms M, -)$ and $\shExt^i_\X(\ms M, -)$ respectively.
We thus define the groups $\Ext^i_\X(\ms M, \ms N)$ and the sheaves $\shExt^i_\X(\ms M, \ms N)$.  Note that we can do this entirely with resolutions using injectives of the form $\bpi I$ for $I$ a graded-injective $\T$-module, by Lemma~\ref{lem:piinj} and the equivalence of categories Theorem~\ref{newthm:catequiv}.

\begin{remark}\label{rem:foo}
Note that $\rMOD \X$ also has enough injectives, using the construction of \cite[Proposition III.2.2]{Ha}.
Thus we can derive $\Hom_\X(\ms M, -) $ either as a functor from $\rMOD \X \to \text{Ab}$ or as a functor from $\rMod \X \to \text{Ab}$.
By Remark~\ref{rem:inj} the two agree on $\rMod \X$, and similarly for $\shHom$.

If we regard $\ms M$ simply as a sheaf of abelian groups on the topological space $\X$ then we have the usual sheaf cohomology $H^i(\X, \ms M)$ by deriving the global section functor $\Gamma = H^0(\X, -)$.
Of course, $\Gamma(\ms M) = H^0(\X, \ms M) = \Hom_{\X}(\ms O_\X, \ms M)$.
As injective objects in $\rMod \X$ are flasque,  by  Proposition~\ref{prop:flasque} with $\ms M = \ms O_\X$, they are $\Gamma$-acyclic by \cite[Proposition~III.2.5]{Ha}.
We can thus compute $H^i(\X, \ms M)$ by taking an injective resolution of $\ms M$ in $\rMod \X$: in other words,
\[ H^i(\X, \ms M) = \Ext^i_\X(\ms O_\X, \ms M).\]
\end{remark}

We can now prove our desired nonsingularity result for $\X$.

\begin{theorem}\label{newthm:smooth}
The categories $\rQgr \T$ and $\rMod \X$ have finite homological dimension.
In particular, if $s>4$ then $\shExt^s_\X(\ms F, \ms G) = 0$ for all quasicoherent sheaves $\ms F, \ms G$ on $\X$.
\end{theorem}

\begin{proof}
    This is a standard Grothendieck spectral sequence.
    We have 
    \[\Hom_\X(\ms F, \ms G) = \Gamma(\shHom_\X(\ms F, \ms G)).\]
    By Proposition~\ref{prop:flasque}, $\shHom_\X(\ms F, -) $ sends injective  objects in $\rMod \X$ to flasque, and thus $\Gamma$-acyclic, sheaves.
    By \cite[Theorem~5.8.3]{Weibel}, there is a convergent spectral sequence
    \beq\label{localglobal} H^p(\X, \shExt^r_\X(\ms F, \ms G)) \Rightarrow \Ext^{p+r}_\X(\ms F, \ms G).\eeq
    Now, by Grothendieck vanishing \cite[Theorem III.2.7]{Ha} and our observation that $\dim \X = 2$, we have $H^p(\X, \ms E) =0$ for any $p \geq 3$ and any sheaf $\ms E$ of abelian groups on $\X$, and so in particular for $\ms E = \shExt^r_\X(\ms F, \ms G)$.
    And since  by Lemma~\ref{lem:dag} $\ms O_\X(U_j)$ has global dimension 2 for all $0 \leq j \leq n$, for $r\geq 3$ we have $\shExt^r_\X(\ms F, \ms G)(U_j) = 0$ for all $0 \leq j \leq n$ and thus $\shExt^r_\X(\ms F, \ms G) =0$.
The result follows.
\end{proof}

Theorem~\ref{newthm:smooth} and Proposition~\ref{prop:june} establish that the morphism $\phi = (\phi^*, \phi_*): \rQgr T \to \rMod B$ can legitimately be considered a resolution of singularities of $B$.
It turns out that $\phi$ is a weakly crepant categorical resolution of $B$ in the sense discussed in Section~\ref{NCCR}.  
This can be proved directly, but it is easier to first prove that $\rqgr T$ and $\rmod \Lambda$ are derived equivalent and then use the corresponding facts about $\psi:  \rMod \Lambda \to \rmod B$, proven in Proposition \ref{prop:PsiCrepantCategoricalRes}.



\section{A derived equivalence of resolutions}\label{DERIVEDEQUIV}
In this section, we give our noncommutative version of the derived geometric McKay correspondence:  that is, we show that  $\rqgr \T$ is derived equivalent to $\rmod \Lambda$.
We will do this by constructing a tilting object in $\rqgr T$.
We then use this derived equivalence to show that $\rqgr T$ is also a weakly crepant categorical resolution of $\rmod B$.

\subsection{The modules of interest and their first properties} \label{sec:OurModules}

We begin by defining some modules that will play an important role for us.

\begin{defn}
For $1 \leqslant i \leqslant n$, define
\begin{gather*}
\sMi \coloneqq x \T + z^i \T, \\
\sNi \coloneqq \sum_{j=1}^n z^{-m_{ij}} t_j \T, \qquad \text{where } m_{ij} \coloneqq \min\{i,j\},
\end{gather*}
viewed as objects in $\rgr \T$ or $\rqgr \T$. Note that, since $\sMi$ and $\sNi$ are generated by normal elements, they are both graded $\T$-bimodules. It will also be convenient to write $\sM^{(0)} = \T = \sN^{(0)}$ and to set
\begin{align*}
\sM \coloneqq \bigoplus_{i=0}^n \sMi.
\end{align*}
\end{defn}

Our first goal is to show that the modules $\sMi$ and $\sNi$ are mutually inverse in $\rqgr T$, in an appropriate sense. We begin by proving some local behaviour of these modules.

\begin{lemma}\label{lem:localiseMi}
We have
\begin{align*}
\sMi[t_k^{-1}]_0 = \left\{
\begin{array}{cl}
x \, \msOk & \text{if } 0 \leqslant k \leqslant i-1, \\
z^i \, \msOk & \text{if } i \leqslant k \leqslant n.
\end{array}
\right.
\end{align*}
\end{lemma}
\begin{proof}
It is clear that $\sMi[t_k^{-1}]_0 = x \msOk + z^i \msOk$. If $0 \leqslant k \leqslant i-1$, then
\begin{align*}
z^i = z^{k+1} t_k t_k^{-1} z^{i-k-1} = q^{\bullet} x t_{k+1} t_k^{-1} z^{i-k-1} \in x \, \msOk.  
\end{align*}
If instead $i \leqslant k \leqslant n$, we have
\begin{align*}
x = x t_k t_k^{-1} = q^\bullet z^k t_{k-1} t_k^{-1} = q^\bullet z^i z^{k-i} t_{k-1} t_k^{-1} \in z^i \, \msOk.
\end{align*}
In either case, the result now follows.
\end{proof}

\begin{lemma}\label{lem:localiseNi}
We have
\begin{align*}
\sNi[t_k^{-1}]_0 = \left\{
\begin{array}{cl}
x^{-1} \, \msOk & \text{if } 0 \leqslant k \leqslant i-1, \\
z^{-i} \, \msOk & \text{if } i \leqslant k \leqslant n.
\end{array}
\right.
\end{align*}
\end{lemma}
\begin{proof}
As in the previous lemma, it is clear that $\sNi[t_k^{-1}]_0 = \sum_{j=1}^n z^{-m_{ij}} t_j t_k^{-1} \msOk$. 
First let $0 \leqslant k \leqslant i-1$, so that $m_{i,k+1} = k+1$. Then
\begin{align*}
x^{-1} = q^{\bullet} z^{-(k+1)} t_{k+1} t_k^{-1} 
= q^{\bullet} z^{-m_{i,k+1}}  t_{k+1} t_k^{-1} \in \sNi[t_k^{-1}]_0.
\end{align*}
Conversely, for $1 \leqslant j \leqslant n$ we have
\begin{align*}
z^{-m_{ij}} t_j t_k^{-1} = q^\bullet x^{-1} z^{-m_{ij}} x t_j t_k^{-1}
= q^\bullet x^{-1} z^{j-m_{ij}} t_{j-1} t_k^{-1} \in x^{-1} \msOk,
\end{align*}
so that $\sNi[t_k^{-1}]_0 = x^{-1} \msOk$. 

If instead $i \leqslant k \leqslant n$, then $m_{ik} = i$, so that
\begin{align*}
z^{-i} = z^{-m_{ik}} t_k t_k^{-1} \in \sNi[t_k^{-1}]_0.
\end{align*}
On the other hand, for $1 \leqslant j \leqslant n$ we have
\begin{align*}
z^{-m_{ij}} t_j t_k^{-1} = q^\bullet z^{-i} z^{i-m_{ij}} t_j t_k^{-1} \in z^{-i} \, \msOk,
\end{align*}
from which it follows that $\sNi[t_k^{-1}]_0 = z^{-i} \, \msOk$.
\end{proof}

\begin{corollary}\label{cor:invertible}
For $1 \leqslant i \leqslant n$, $\sMi$ and $\sNi$ are  are mutually inverse in $\rqgr \T$ in the sense that $\sMi\otimes_\T \sNi \cong \sNi \otimes_\T \sMi \cong \T$ in $\rqgr T$.
As a result, $-\otimes_T \sM^{(i)}$ and $-\otimes_T \sN^{(i)}$ are (quasi-inverse) autoequivalences of $\rqgr \T$, and similarly on the left.
\end{corollary}
\begin{proof}
Since there is an equivalence $\rqgr \T \simeq \rmod \ms O_\X$, it suffices to prove these isomorphisms locally; this follows quickly from Lemmata \ref{lem:localiseMi} and \ref{lem:localiseNi}.
\end{proof}

We next aim to show that the tensor product 
\[ \sM_1 \otimes_T \cdots \otimes_T \sM_n\]
is isomorphic to $\T(1)$ in $\rqgr \T$; geometrically, this means that this object is ``ample''. We first require an easy lemma:

\begin{lemma}
For $0 \leq k \leq n$, we have 
$\T[t_k^{-1}]_1 = t_k \msOk$.
\end{lemma}
\begin{proof}
It is clear that $t_k \msOk \subseteq \T[t_k^{-1}]_1$. For the reverse inclusion, it suffices to show that $t_0, \dots, t_n \in t_k \msOk$, but this is clear, since $t_i = t_k (t_k^{-1} t_i) \in t_k \msOk$ for $0 \leqslant i \leqslant n$.  
\end{proof}

Recall that the {\em graded quotient ring of $\T$} (in the projective grading) is
\[\Qgr(T) = \Qgr[\ZZ](\T) = \T[h^{-1} \mid 0 \neq h \in \T_d \mbox{ for some $d\in \NN$}] .\]
\begin{lemma} \label{lem:BetaHom}
The graded module homomorphism 
\begin{gather*}
\beta : \sM^{(1)} \otimes_\T \sM^{(2)} \otimes_\T \dots \otimes_\T \sM^{(n)} \to \Qgr(\T)(1), \quad \beta(m_1 \otimes \dots \otimes m_n) = m_1 m_2 \dots m_n t_0 x^{-n}
\end{gather*}
has image in $\T(1)$.
\end{lemma}
\begin{proof}
It suffices to show that the generators of the domain of $\beta$ get sent to an element of $\T$. Suppose that $m_1 \otimes \dots \otimes m_n$ is such that each $m_j$ is either $x$ or $z^j$, and suppose that exactly $n-i$ of $m_1, \dots, m_n$ are equal to $x$. Then the product of the remaining $m_j$ is equal to $z^r$ for some $r$, where necessarily $r \geqslant 1+2+\dots+(i-1)+i = \frac{1}{2}i(i+1)$. Then
\begin{align*}
\beta(m_1 \otimes \dots \otimes m_n) &= q^\bullet z^r x^{n-i}t_0 x^{-n} \\
&= q^\bullet z^{r - \frac{1}{2}i(i+1)} \big(z^i x^{-1}\big) \big(z^{i-1} x^{-1}\big) \dots \big(z^{2} x^{-1}\big) \big(z^1 x^{-1}\big) t_0 \\
&= q^\bullet z^{r - \frac{1}{2}i(i+1)} \big(z^i x^{-1}\big) \big(z^{i-1} x^{-1}\big) \dots \big(z^{2} x^{-1}\big) t_1 \\
&= q^\bullet z^{r - \frac{1}{2}i(i+1)} \big(z^i x^{-1}\big) \big(z^{i-1} x^{-1}\big) \dots \big(z^{3} x^{-1}\big) t_2 \\
& \hspace{25pt} \vdots \\
&= q^\bullet z^{r - \frac{1}{2}i(i+1)} z^{i} x^{-1} t_{i-1} \\
&= q^\bullet z^{r - \frac{1}{2}i(i+1)} t_i,
\end{align*}
which lies in $\T$.
\end{proof}

\begin{proposition}\label{prop:ProductOfMi}
In $\rqgr \T$, we have
\begin{align*}
\sM^{(1)} \otimes_\T \sM^{(2)} \otimes_\T \dots \otimes_\T \sM^{(n)} \cong \T(1).
\end{align*}
\end{proposition}
\begin{proof}
Let $\beta$ be as in Lemma~\ref{lem:BetaHom} and fix $0 \leqslant k \leqslant n$. 
From the equivalence $\rqgr \T \simeq \rmod \ms O_\X$  it suffices to show that $\beta$ is locally an isomorphism onto $\T(1)$. 
Localising $\beta$ at $t_k$, by Lemma~\ref{lem:localiseMi} in degree $0$ we obtain
\begin{align*}
\widetilde{\beta} : \underbrace{\left( z \, \msOk \otimes z^2 \, \msOk \otimes \dots z^k \, \msOk \right)}_{k \text{ terms}} \otimes \underbrace{\left( x \, \msOk \otimes \dots \otimes x \, \msOk \right)}_{n-k \text{ terms}} \to t_k \, \msOk,
\end{align*}
where the tensor products are over $\msOk$. 
The domain of this map is free of rank $1$, with generator $z^{\tfrac{1}{2}k(k+1)} x^{n-k}$, and we can identify it with the map
\begin{align*}
\widetilde{\beta} : z^{\tfrac{1}{2}k(k+1)} x^{n-k} \, \msOk \to t_k \, \msOk, \quad m \mapsto t_0 x^{-n} m.
\end{align*}
By the proof of Lemma~\ref{lem:BetaHom}, we have
\begin{align*}
\widetilde{\beta}\big(z^{\tfrac{1}{2}k(k+1)} x^{n-k}\big) = t_0 x^{-n} z^{\tfrac{1}{2}k(k+1)} x^{n-k} = q^\bullet \big( z^k x^{-1} \big) \big( z^{k-1} x^{-1} \big) \dots \big( z x^{-1} \big) t_0 = q^\bullet t_k,
\end{align*}
which shows that $\widetilde{\beta}$ is an isomorphism. Thus $\beta$ is locally an isomorphism, and hence an isomorphism in $\rqgr \T$.
\end{proof}

We end this section by identifying some short exact sequences of $\rqgr T$-modules in which the modules $\sMi$ and $\sNi$ sit.

\begin{proposition}
Write
\begin{align*}
\gamma : \T \to \T, \quad \gamma(t) = xtx^{-1},
\end{align*}
which is an automorphism of $\T$, and which restricts to an automorphism of any right (or left) ideal of $\T$. Then, for any $1 \leqslant i, j \leqslant n$, the sequence
\begin{align}
0 \to \T \xrightarrow{t \mapsto \begin{pmatrix} \gamma^{-1}(z^i)t \\ -xt \end{pmatrix}} \sM^{(i)} \oplus \sM^{(j)} \xrightarrow{\begin{pmatrix}m \\ m' \end{pmatrix} \mapsto \gamma(m) \otimes x + z^i \otimes m'} \sMi \otimes_\T \sM^{(j)} \to 0 \label{BijSequence}
\end{align}
is exact in $\rqgr \T$.
\end{proposition}
\begin{proof}
It is not immediately obvious, but nevertheless straightforward to check, that the second map in \eqref{BijSequence} is $\T$-linear and that this sequence forms a complex. 
It remains to verify exactness, and it suffices to check this locally by the equivalence $\rqgr \T \simeq \rmod \ms O_\X$. 
Assume first that $i \leqslant j$. Localising at $t_k$, we obtain the following sequences:
\begin{align*}
\begin{array}{cll}
0 \to \msOk \to x \msOk \oplus x \msOk \to x^2 \msOk \to 0 & & \text{if } 0 \leqslant k \leqslant i-1  \\[2pt]
0 \to \msOk \to z^i \msOk \oplus  x \msOk \to xz^i \msOk \to 0 & & \text{if } i \leqslant k \leqslant j-1 \\[2pt]
0 \to \msOk \to z^i \msOk \oplus z^j \msOk \to z^{i+j} \msOk \to 0 & & \text{if } j \leqslant k \leqslant n,
\end{array}
\end{align*}
each of which are exact. We obtain similar exact sequences when $i > j$. The result now follows.
\end{proof}

\begin{corollary}
For any $1 \leqslant i , j \leqslant n$, the sequence
\begin{align}
0 \to \sN^{(j)} \xrightarrow{n \mapsto \begin{psmallmatrix} \gamma^{-1}(z^i) \otimes n \\ -xn \end{psmallmatrix}} \big(\sM^{(i)} \otimes_\T \sN^{(j)}\big) \oplus \T \xrightarrow{\begin{psmallmatrix}m \otimes n \\ r \end{psmallmatrix} \mapsto \gamma(m) xn + z^i r} \sMi \to 0 \label{BijSequence2}
\end{align}
is exact in $\rqgr \T$. In particular, for any $1 \leqslant i \leqslant n$, the sequence
\begin{align}
0 \to \sNi \xrightarrow{\begin{pmatrix} \gamma^{-1}(z^i) \\ -x \end{pmatrix}} \T^{\oplus 2} \xrightarrow{\begin{pmatrix} x & z^i \end{pmatrix}} \sMi \to 0 \label{BiiSequence}
\end{align}
is exact in $\rqgr \T$, where the maps are given by multiplication from the left by the indicated matrices.
\end{corollary}
\begin{proof}
There is a commutative diagram
\begin{equation*}
\begin{tikzcd}[>=stealth,ampersand replacement=\&,column sep=40pt] 
0 \arrow{r} 
\&[-30pt] \T \otimes \sN^{(j)} \arrow{r}[above,inner sep=10pt]{t \otimes n \mapsto \begin{pmatrix} \gamma^{-1}(x^i) t \otimes n \\ -xt \otimes n \end{pmatrix}} \arrow{d}[description]{t \otimes n \mapsto tn} 
\& \big(\sMi \otimes \sN^{(j)}\big) \oplus \big(\sMj \otimes \sN^{(j)}\big) \arrow{r}[above,inner sep=10pt]{\begin{pmatrix} m \otimes n \\ m' \otimes n' \end{pmatrix} \mapsto \gamma(m) \otimes x \otimes n + z^i \otimes m' \otimes n'} \arrow{d}[description]{\begin{pmatrix}m \otimes n \\ m' \otimes n'\end{pmatrix} \mapsto \begin{pmatrix}m \otimes n \\ m'n' \end{pmatrix}} 
\& \sMi \otimes \sMj \otimes \sN^{(j)} \arrow{r} \arrow{d}[description]{m \otimes m' \otimes n \mapsto m (m'n')} 
\&[-30pt] 0 \\[40pt]
0 \arrow{r} 
\& \sNj \arrow{r}[below,inner sep=5pt]{n \mapsto \begin{pmatrix} \gamma^{-1}(z^i) \otimes n \\ -xn \end{pmatrix}} 
\& \big( \sMi \otimes \sNj \big) \oplus \T \arrow{r}[below,inner sep=5pt]{\begin{pmatrix} m \otimes n \\ t \end{pmatrix} \mapsto \gamma(m) xn + z^i r }
\& \sMi \arrow{r} 
\& 0
\end{tikzcd}
\end{equation*}
where the top row is obtained by applying the exact functor $- \otimes_\T \sN^{(j)}$ to \eqref{BijSequence}, and hence is exact, and whose vertical arrows are isomorphisms. This implies exactness of the bottom row, which is precisely \eqref{BijSequence2}. Finally, exactness of \eqref{BiiSequence} follows by specialising to the case $j=i$.
\end{proof}

\subsection{\texorpdfstring{$\sM$}{M} is a tilting object} 
Recall our notation
\begin{align*}
M \coloneqq \bigoplus_{i=0}^n M^{(i)} \qquad \text{and} \qquad \sM \coloneqq \bigoplus_{i=0}^n \sMi
\end{align*}
from Sections \ref{sec:OurSingularity} and \ref{sec:OurModules}. 
The goal of this subsection is to determine $\Ext_{\rQgr \T}^i(\sM,\sM)$ for all $i \geqslant 1$. We first consider the case $i=0$, i.e.~we determine the endomorphism ring of $\sM$. We begin with some preparatory lemmata.

\begin{lemma}\label{lem7.14}
As right $\T$-modules,  $\sMi \cong (M^{(i)} \otimes_B \T)/\operatorname{tors}(M^{(i)} \otimes_B \T)$. In particular, there is a short exact sequence of $T$-modules
\begin{align}
0 \to \operatorname{tors}(M^{(i)} \otimes_B \T) \to M^{(i)} \otimes_B \T \to \sMi \to 0. \label{eqn:TorsMi}
\end{align}
\end{lemma}
\begin{proof}
We show that
\beq\label{tensorprod} M^{(i)} \otimes_B \T[t_j^{-1}]_0 \cong \sM^{(i)} [t_j^{-1}]_0,\eeq
and the statement will then follow from the category equivalence of Theorem~\ref{newthm:catequiv}. To see this, first note from Lemma \ref{lem:MiSES} that $M^{(i)}$ is freely presented via the exact sequence
\begin{align*}
B^2 \xrightarrow{\begin{pmatrix}q^{i(n+1)} z^i & -y \\ -x & q^{-\binom{n+1}{2}} z^{n+1-i} \end{pmatrix}} B^2 \xrightarrow{\begin{pmatrix}x & z^i \end{pmatrix}} M^{(i)} \to 0.
\end{align*}
Applying the right exact functor $- \otimes_B \T[t_j^{-1}]_0$ gives a commutative diagram:
\begin{equation*}
\begin{tikzcd}[>=stealth,ampersand replacement=\&,column sep=25pt] 
B^2 \otimes_B \T[t_j^{-1}]_0 \arrow{r} \arrow{d}[left]{\cong} \& B^2 \otimes_B \T[t_j^{-1}]_0 \arrow{r} \arrow{d}[left]{\cong} \& M^{(i)} \otimes \T[t_j^{-1}]_0 \arrow[d,dashed] \arrow{r} \& 0 \\
(\T[t_j^{-1}]_0)^2 \arrow{r} \& (\T[t_j^{-1}]_0)^2 \arrow{r} \& X \arrow{r} \& 0
\end{tikzcd}
\end{equation*}
where the top row is exact and the module $X$ is a suitable cokernel. 
The bottom row is simply a free presentation of $\sM^{(i)}[t_j^{-1}]_0$, and the dashed arrow exists and is an isomorphism by the Five Lemma. Thus $M^{(i)} \otimes \T[t_j^{-1}]_0 \cong \sM^{(i)}[t_j^{-1}]_0$, as desired.
\end{proof}

\begin{lemma}\label{lem:simonloc}
We have $\Qgr[\ZZ](\T)_0 = Q(B).$
\end{lemma}
\begin{proof}
Since $B = T_0$, the inclusion $Q(B) \subseteq Q_{\gr}(\T)_0$ is clear, so we consider only the reverse inclusion. 
It suffices to show that any $t_i t_j^{-1}$ lies in $Q(B)$. We prove this by induction on $|j-i|$, with the case $|j-i|=0$ being clear. If $|j-i|=1$, then either $j=i-1$ or $j=i+1$. In the former case, the relation $x t_i = \qb z^i t_{i-1}$ implies $t_i t_{i-1}^{-1} = \qb x^{-1} z^i \in Q(B)$. Similarly, in the latter case we have $t_i t_{i+1}^{-1} = \qb y^{-1} z^{n-i} \in Q(B)$. \\
\indent Now assume that $|j-i| \geqslant 2$; again, we consider two cases. If $j > i$ then
\begin{align*}
t_i t_j^{-1} = t_i t_{j-2} t_{j-2}^{-1} t_j^{-1} = \qb t_i t_{j-2} z^{-1} t_{j-1}^{-2} = \qb z^{-1} \;\, t_i t_{j-1}^{-1} \;\, t_{j-2}t_{j-1}^{-1},
\end{align*}
which lies in $Q(B)$ by the induction hypothesis. Similarly, if $i > j$ then
\begin{align*}
t_i t_j^{-1} = t_i t_{i-2} t_{i-2}^{-1} t_j^{-1} = \qb z t_{i-1}^2 t_{i-2}^{-1} t_j^{-1} = \qb z \;\, t_{i-1} t_{i-2}^{-1} \;\, t_{i-1}t_j^{-1},
\end{align*}
which again lies in $Q(B)$ by the induction hypothesis.
\end{proof}

\begin{lemma} \label{lem:H1MVanish}
For $0 \leq i \leq n$ we have:
\begin{enumerate}[{\normalfont (1)},topsep=0pt,itemsep=0pt,leftmargin=*]
\item $\Hom_{\rQgr \T}(\T, \sMi) = M^{(i)}$; and
\item $H^1_{\T_+}(\sM)_0 = 0$.
\end{enumerate}
\end{lemma}
\begin{proof}
By Proposition~\ref{prop:cohbasics}(2) it suffices to prove that $H^1_{\T_+}(\sMi)_0=0$ for $0 \leqslant i \leqslant n$. We begin by showing that $H^1_{\T_+}(\sMi)_0$ is an $xB + yB$-torsion $B$-module. Since local cohomology commutes with localisation (see, for example, \cite[Theorem 10.67]{rotman}, the proof of which holds in our situation with minor modifications) we obtain
\begin{align*}
H^1_{\T_+}(\sMi)_0[x^{-1}] \cong H^1_{\T_+}(\T)_0[x^{-1}],
\end{align*}
since it is clear from the definition of $\sMi$ that  $\sMi[x^{-1}] = \T[x^{-1}]$.
The right hand side is $0$ by Proposition~\ref{prop:localcohomology}, so 
 $H^1_{\T_+}(\sMi)_0$ has $x$-torsion and, similarly, it has $y$-torsion. In particular, $\GKdim H^1_{\T_+}(\sMi)_0 = 0$.\\
\indent Now, since $M^{(i)}_B$ is maximal Cohen-Macaulay and $B$ is Auslander--Gorenstein of dimension $2$, $M^{(i)}_B$ is reflexive. By \cite[Lemma~4.11(1)]{RSS}, if $M'$ is a $B$-module satisfying $M^{(i)} \subseteq M' \subseteq Q(B)$ with $\GKdim M'/M^{(i)} = 0$, then $M' = M^{(i)}$.
We seek to apply this result when $M' \coloneqq \Hom_{\rQgr \T}(\T, \sMi)$. 
Note that  we have
\[ M' \subseteq \Hom_{\rQgr \T}(\T, \Qgr[\ZZ](\T)) = \End_{\rQgr \T}(\Qgr[\ZZ](\T)) = \Qgr[\ZZ](\T)_0 = Q(B),\]
where the last equality is Lemma~\ref{lem:simonloc}.

\indent We now show that $\GKdim M'/M^{(i)} = 0$. 
Consider the exact sequence in Proposition~\ref{prop:cohbasics}(2) with $N = \sMi$ and take degree 0.
Since $\sMi$ is torsion-free and $\sMi_0 = M^{(i)}$ we obtain an exact sequence 
\begin{align}
0 \to M^{(i)} \to M' \to H^1_{\T_+}(\sM^{(i)})_0 \to 0. \label{eqn:LCSES}
\end{align}
Thus
\begin{align*}
\GKdim M'/M^{(i)} \leqslant 0,
\end{align*}
by the first paragraph of the proof. The second paragraph implies $M'= M^{(i)}$, establishing the result.
\end{proof}

\begin{corollary} \label{cor:HomQgrTM}
$\Hom_{\rQgr \T}(\T,\sM) \cong M$.
\end{corollary}
\begin{proof}
This follows from Lemma~\ref{lem:H1MVanish} (1).
\end{proof}

\begin{lemma} \label{lem:LocalCohOfMN}
We have $H^i_{\T_+}(\sMj) = 0 = H^i_{\T_+}(\sNj)$ for all $i \geqslant 3$ and all $0 \leqslant j \leqslant n$.
\end{lemma}
\begin{proof}
This follows from Corollary~\ref{cor:local2} and Corollary~\ref{cor:invertible}.
\end{proof}

With these in hand, we are able to determine the desired endomorphism ring.

\begin{proposition} \label{prop:EndomorphismRing}
$\End_{\rQgr \T}(\sM) \cong \End_B(M) = \Lambda$.
\end{proposition}
\begin{proof}
Since $\sM$ is torsion-free, by applying $\Hom_{\rGr \T}(-, \sM)$ to the sum of the short exact sequences \eqref{eqn:TorsMi}, we obtain
\begin{align*}
\End_{\rGr \T}(\sM) \cong \Hom_{\rGr \T}(M \otimes_B \T,\sM).
\end{align*}
We then have a chain of isomorphisms
\begin{align*}
\End_{\rQgr \T}(\sM) 
&\cong \lim_{n \to \infty} \Hom_{\rGr \T}(\sM_{\geqslant n},\sM) \\
&\cong \lim_{n \to \infty} \Hom_{\rGr \T}(M \otimes_B \T_{\geqslant n},\sM) \\
&\cong \lim_{n \to \infty} \Hom_B(M, \Hom_{\rGr \T}(\T_{\geqslant n},\sM)) \\
&\cong \Hom_B(M, \lim_{n \to \infty} \Hom_{\rGr \T}(\T_{\geqslant n},\sM)) \\
&\cong \Hom_B(M, \Hom_{\rQgr \T}(\T,\sM)) \\
&\cong \End_B(M),
\end{align*}
where we have used Corollary \ref{cor:HomQgrTM} and the fact that $\Hom_B(M,-)$ commutes with direct limits since $M$ is finitely presented.
\end{proof}

Finally, we are able to show that the higher Ext groups of $\sM$ in $\rQgr T$ vanish.

\begin{proposition} \label{prop:VanishingExt}
For $s \geqslant 1$, we have
\begin{align*}
\Ext_{\rQgr \T}^s(\sM,\sM) = 0.
\end{align*}
\end{proposition}
\begin{proof}
We have 
\begin{align*}
\Ext^s_{\rQgr \T}(\sM,\sM) &= \bigoplus_{j,k=0}^n \Ext^s_{\rQgr \T}(\sMj,\sM^{(k)}) \\
&= \bigoplus_{j,k=0}^n \Ext^s_{\rQgr \T}(\T, \sM^{(k)} \otimes_\T  \sNj) \\
&= \bigoplus_{j,k=0}^n \H^s_{\rQgr \T}(\sM^{(k)} \otimes_\T  \sNj), 
\end{align*}
and from \eqref{BijSequence2} and Proposition~\ref{prop:localcohomology} it suffices to prove that $\H^s_{\rQgr T}(\sMi) = \H^s_{\rQgr T}(\sNi) = 0$ for $i, s \geq 1$.
Taking cohomology of \eqref{BiiSequence} and using Proposition~\ref{prop:localcohomology} we obtain  an exact sequence
\beq \label{eq:june3}
0 \to \H^0_{\rQgr T}(\sNi) \to B^{\oplus 2} \to \H^0_{\rQgr T}(\sMi) \to 
\H^1_{\rQgr T}(\sNi) \to 0
\eeq
and an isomorphism
\[ \H^1_{\rQgr T}(\sMi) \cong \H^2_{\rQgr T}(\sNi). \]
This last is 0 by Corollary~\ref{cor:invertible} and Lemma~\ref{lem:LocalCohOfMN}, so we consider \eqref{eq:june3}.
Now, $\H^0_{\rQgr T}(\sMi) = M^{(i)}$,
and applying the proof of Proposition~\ref{prop:EndomorphismRing} to $\Hom_{\rQgr T}(\sMi, T)$ shows that
\[ \H^0_{\rQgr T}(\sNi) \cong \Hom_{\rQgr T}(\sMi, T) \cong \Hom_B(M^{(i)}, T).\]
By Lemma \ref{lem:HomMi}, this is $M^{(n+1-i)}$ and so \eqref{eq:june3} becomes
\[
0 \to M^{(n+1-i)} \to B^2 \to M^{(i)} \to \H^1_{\rQgr T}(\sNi) \to 0 .
\]
The first three terms of this are from the short exact sequence of Lemma \ref{lem:MiSES} and one can check that the  maps between them are also those of Lemma \ref{lem:MiSES}, so $\H^1_{\rQgr T}(\sNi) = 0$.
\end{proof}


\subsection{The proof of the derived equivalence}\label{TILTING}
In this subsection we prove that the two resolutions $\rqgr \T$ and $\rmod \Lambda$ of the singularity $B$ are derived equivalent.

Let $\ang{\sM}$ denote the subcategory of $\rqgr \T$ given by the closure of $\sM$ under taking finite direct sums, direct summands, and kernels of epimorphisms. (We caution that this is nonstandard notation.)

\begin{proposition} \label{prop:SetOfGenerators}
\mbox{}
\begin{enumerate}[{\normalfont (1)},topsep=0pt,itemsep=0pt,leftmargin=*]
\item Let $k \in \NN$ and let $i_1, \dots, i_k \in \{ 0, \dots, n \}$. Then
\begin{align*}
\sN^{(i_1)} \otimes_\T \dots \otimes_\T \sN^{(i_k)} \in \ang{\sM}.
\end{align*}
\item If $s > 0$ then $\T(-s) \in \ang{\sM}$.
\end{enumerate}
\end{proposition}
\begin{proof}\mbox{}
\begin{enumerate}[{\normalfont (1)},wide=0pt,topsep=0pt,itemsep=0pt]
\item We prove this by induction on $k$. By \eqref{BiiSequence}, $\sNi \in \ang{\sM}$ for $0 \leqslant i \leqslant n$. Now suppose that $k \geqslant 2$, and consider the short exact sequence in $\rqgr \T$
\begin{align*}
0 \to \T \to \sM^{(i_{k-1})} \oplus \sM^{(i_k)} \to \sM^{(i_{k-1})} \otimes_\T \sM^{(i_k)} \to 0
\end{align*}
from \eqref{BijSequence}. Applying the functors $\sN^{(i_1)} \otimes_\T \dots \otimes_\T \sN^{(i_{k-1})} \otimes_\T - $ and $ - \otimes_\T \sN^{(i_k)}$, both of which are exact since the $\sNi$ are invertible objects in $\rqgr \T$, yields an exact sequence
\begin{align*}
0 \to \sN^{(i_1)} \otimes_\T \dots \otimes_\T \sN^{(i_k)}
\to \big( \sN^{(i_1)} \otimes_\T \dots \otimes_\T \sN^{(i_{k-2})} \otimes_\T \sN^{(i_{k})} \big) \oplus \big( \sN^{(i_1)} \otimes_\T \dots \otimes_\T \sN^{(i_{k-1})} \big) \hspace{40pt}\\
\to \sN^{(i_1)} \otimes_\T \dots \otimes_\T \sN^{(i_{k-2})} \to 0. \hspace{33pt}
\end{align*}
The middle and right-hand terms both lie in $\ang{\sM}$ by the induction hypothesis, and so the same is true of $\sN^{(i_1)} \otimes_\T \dots \otimes_\T \sN^{(i_k)}$.
\item By Proposition~\ref{prop:ProductOfMi}, we have $\sM^{(1)} \otimes_\T \dots \otimes_\T \sM^{(n)} \cong \T(1)$ in $\rqgr \T$, and so $\sN^{(n)} \otimes_\T \dots \otimes_\T \sN^{(1)} \cong \T(-1)$. Hence in $\rqgr \T$ we have 
\begin{align*}
\big(\sN^{(n)} \otimes_\T \dots \otimes_\T \sN^{(1)} \big)^{\otimes_\T s} \cong \T(-s),
\end{align*}
and this object lies in $\ang{\sM}$ by part (1). \qedhere
\end{enumerate}
\end{proof}

A set of objects $S$ in an abelian category $\mathcal{A}$ is said to \emph{generate} $\mathcal{A}$ if, for every object $A \in \mathcal{A}$, there exist objects $B_1, \dots, B_k \in \mathcal{A}$ and an epimorphism
\begin{align*}
B_1 \oplus \dots \oplus B_k \twoheadrightarrow A.
\end{align*}

\begin{lemma}\label{lem:Generators}
{\rm (cf.~\cite[Proposition 4.4]{AZ})}
$\rqgr T$ is generated by $\{ \T(-s) \mid s > 0 \}$.
\end{lemma}
\begin{proof}
The proof of \cite[Proposition 4.4(1)]{AZ} works in our situation but, as it is short, we give it here. Let $\sF = \pi(F)$ be an object of $\rqgr \T$, where $F \in \rgr T$.
Then $\sF = \pi(F_{\geqslant 1})$ as $F/F_{\geqslant 1}$ is torsion. Taking a set of homogeneous generators for $F_{\geqslant 1}$, we obtain a surjection 
\begin{align*}
\T(-s_1) \oplus \dots \oplus \T(-s_k) \to F_{\geqslant 1} 
\end{align*}
of graded right $\T$-modules for some positive integers $s_1, \dots, s_k$. Since the quotient functor $\pi : \rgr \T \to \rqgr \T$ is exact this gives a surjection
\begin{align*}
\T(-s_1) \oplus \dots \oplus \T(-s_k) \to \mathcal{F}
\end{align*}
in $\rqgr \T$, as desired.
\end{proof}

We are now in a position to prove our noncommutative McKay correspondence.

\begin{theorem}\label{thm:tilting}
$\sM$ is a tilting object for $D^b(\rqgr T)$ and the functor $\RHom_{\rQgr \T}(\sM, -)$ induces a triangle equivalence
\begin{align*}
D^b(\rqgr \T) \to D^b(\rmod \Lambda) = \mathscr{D}(\Lambda).
\end{align*}
\end{theorem}
\begin{proof}
This is a direct consequence of \cite[Proposition 4.6]{keller}, once we verify the hypotheses of that result. We need to show the following:
\begin{enumerate}[{\normalfont (1)},topsep=0pt,itemsep=0pt,leftmargin=30pt]
\item $\rQgr \T$ is locally noetherian and $\sM$ is a noetherian object in $\rQgr \T$;
\item $\rQgr \T$ has finite homological dimension;
\item $\Ext^i_{\rQgr \T}(\sM,\sM) = 0$ for all $i > 0$;
\item $\ang{\sM}$ contains a set of generators of $\rQgr \T$;  
\item $\End_{\rQgr \T}(\sM) \cong \Lambda$; and
\item $\Lambda$ is noetherian.
\end{enumerate}
As $\T$ is noetherian, $\rQgr \T$ is a locally noetherian category, and $\mc M$ is noetherian since it is a finitely generated $\T$-module, establishing (1). Properties (2) and (3) are proved in Theorem~\ref{newthm:smooth} and Proposition~\ref{prop:VanishingExt}, respectively. Combining Proposition~\ref{prop:SetOfGenerators} and Lemma~\ref{lem:Generators} shows that (4) is satisfied. Finally, (5) is established in Proposition~\ref{prop:EndomorphismRing}. The result now follows since $\Lambda \cong A \hash  G$ is noetherian. 
\end{proof}

\begin{remark}\label{rem:epsilon}
It is possible to define a triangulated functor
\[ - \otimes^{\LL}_\Lambda \sM: \mc D^b(\Lambda) \to D^b(\rqgr T)\]
using $\otimes_\Lambda$ on charts of the noncommutative toric variety $\X$ 
and the equivalence of Theorem~\ref{newthm:catequiv} between $\rmod \ms O_\X$ and $\rqgr T$.
One can then show that $- \otimes^{\LL}_\Lambda \sM $ is  quasi-inverse to the functor $\RHom_{\rQgr T}(\sM, -)$ from Theorem~\ref{thm:tilting}.
Since we will not need the quasi-inverse explicitly, we omit the details.
\end{remark}

Let us now give the promised proof that the map $\phi: \rqgr T \to \rmod B$ is a weakly crepant categorical resolution.

Recall that $\phi_* = \H^0_{\rQgr \T}(-): \rqgr T \to \rmod B$ and $\phi^* = \pi(-\otimes_B T): \rmod B \to \rqgr T$, and that $(\phi^*, \phi_*)$ are an adjoint pair.
We will denote the associated right derived functor $D^b(\rqgr T) \to D(B)$ also by $\phi_*$.
Likewise, we will use  $\phi^*:  D(B) \to D(\rqgr T)$ to denote the triangulated functor associated to the abelian functor $\phi^*$, which we note is right exact as a consequence of Proposition~\ref{prop:adjoint}. It will be helpful to have more compact notation for the triangle equivalence between $D^b(\rqgr T)$ and $\ms D(\Lambda)$, so temporarily let $\Theta$ be the functor $\RHom_{\rQgr T}(\sM, -)$ from Theorem~\ref{thm:tilting}.

The following diagram records the action of these functors:

\begin{equation}\label{diagram}
\begin{tikzcd}[>=stealth,ampersand replacement=\&,row sep=55pt,every label/.append style={font=\normalsize},column sep=-5pt] 
D(\rMod \Lambda) \arrow[dr, "\psi_*" {right,xshift=0pt,yshift=4pt},shift left=2]  \& \& \arrow[ll, "\Theta" {above,yshift=1pt}, "\cong" {below,yshift=-1pt}] D(\rQgr T)
\arrow[dl, "\phi_*" {right,xshift=0pt,yshift=-3pt},shift left=3,yshift=7pt] \\
\& \arrow[ul, "\psi^*" {left,xshift=-1pt,yshift=-3pt},shift left=3,yshift=7pt] D(\rMod B) \arrow[ur, "\phi^*" {left,xshift=2pt,yshift=4pt},shift left=2] \&
\end{tikzcd}
\end{equation}
Note that \eqref{diagram} is not commutative: for example, $\Theta(\sM) = \Lambda \neq \psi^*\phi_* (\sM) = \Lambda e_0 \Lambda$.
The quotient $\Lambda/\Lambda e_0 \Lambda $ is the type $\AA_{n}$ preprojective algebra, by \cite[Theorem~8.2]{crawfordthesis}, which has dimension $\binom{n+2}{3}$ and in particular is nonzero.

However, there are some nice relations between these functors, which we give in the next two results.

\begin{lemma}\label{lem:ii}
The functors $\Theta \phi^*$ and $\psi^*$ from $ D(B) \to \ms D(\Lambda)$ agree on $\Perf(B)$.
\end{lemma}
\begin{proof}
We use the criteria of Lemma~\ref{lem:star}.
First,
\[ \Theta \phi^* B = \Theta(\pi T) = \Hom_{\rqgr T}(\sM, \pi T) = M^* = \psi^* B, \]
where we have used Proposition~\ref{prop:EndomorphismRing} for the isomorphism between $M^*$ and $\Hom_{\rQgr T}(\sM, \pi T)$.
If $b \in B$, let $\lambda_b = (b \cdot -)$ denote the corresponding endomorphism of $B_B$.
Note that $\lambda_b$ also gives an element of $\Lambda = \End_\Lambda(\Lambda)$, from the inclusion  $B = e_0 \Lambda e_0 \subseteq \Lambda$.
Then 
\[ \Theta \phi^*(\lambda_b) : \Hom_{\rQgr T}(\sM, \pi T) \to \Hom_{\rQgr T}(\sM, \pi T) \]
is given by $f \mapsto \lambda_b \circ f$.
On the other hand, $\psi^*(\lambda_b) = \lambda_b \otimes 1: B \otimes_B M^* \to B \otimes_B M^*$ is also given by left multiplication by $b$.
The result follows by Lemma~\ref{lem:star}.
\end{proof}

\begin{lemma}\label{lem:iii}
The functors $\phi_* $ and $\psi_* \Theta$ from $D(\rQgr T) \to  D(R)$
agree on $D^b(\rqgr T) $.
\end{lemma}
\begin{proof}
Let us first check that these functors agree on $\sM$ and on endomorphisms of $\sM$.
We have $\phi_* \sM = H^0_{\rQgr T}(\sM) = M$, and $\psi_* \Theta \sM = \psi_* \Lambda = \Lambda\otimes_\Lambda M = M$.
Since $\psi_* \Lambda = M = \phi_* \sM$ and $\Lambda = \Theta(\sM)$, there are homomorphisms 
\begin{equation*}
\begin{tikzcd}[>=stealth,ampersand replacement=\&,row sep=40pt,every label/.append style={font=\normalsize},column sep=-5pt] 
\Lambda = \End_\Lambda(\Lambda) \arrow[dr, "\psi_*" {swap,xshift=-3pt,yshift=4pt}]  \& \& \arrow[ll, "\Theta" {above,yshift=1pt}] \End_{\rQgr T}(\sM)
\arrow[dl, "\phi_*" {xshift=2pt,yshift=4pt}] \\
\& \End_B(M) \&
\end{tikzcd}
\end{equation*}
These maps induce the canonical identifications, where we have used Proposition~\ref{prop:EndomorphismRing} to identify $\End_B(M)$ with $\Lambda$.
It follows that the functors agree on endomorphisms of $\sM$.

Since $\Lambda$ has finite global dimension, we obtain
$\ms D(\Lambda) = \Perf( \Lambda)$,
which is the image under $\Theta$ of $\Perf(D(\rQgr T))=D^b(\rqgr T)$.
Further, $\Lambda = \Theta(\sM)$.
The result follows from applying Lemma~\ref{lem:star} to $\Perf(D(\rQgr T))$, using that $\Theta$ is a triangle equivalence.
\end{proof}

We can now prove that $\phi$ is a weakly crepant categorical resolution.

\begin{theorem}\label{thm:iv}
The map $\phi = (\phi^*, \phi_*)$ of noncommutative spaces is a weakly crepant categorical resolution of $ D(B)$.
\end{theorem}

\begin{proof}
We have seen in Proposition~\ref{prop:adjoint} that $(\phi^*, \phi_*)$ are an adjoint pair and so do give a map of noncommutative spaces.  
It follows from Theorem~\ref{thm:tilting} and Remark~\ref{rem:CY2} that $D^b(\rqgr T)$ is smooth.
To show that $\phi$ is a categorical resolution, it remains to verify that $\phi_*\phi^* \cong \id_{\Perf(B)}$.
This can be checked directly, or using Lemmata~\ref{lem:ii} and \ref{lem:iii}:
we have
\[ \phi_* \phi^* \cong \psi_* \Theta \phi^* \cong \psi_* \psi^* \] 
as functors from $\Perf(B)$ to $ D(B)$.  
This last is naturally isomorphic to $\id_{\Perf(B)}$ by Proposition~\ref{prop:PsiCrepantCategoricalRes}.

To show that $\phi$ is weakly crepant, we need that $\phi^* : \Perf(B) \to D(\rQgr T)$ is also a right adjoint to $\phi_*$, that is, that there are natural isomorphisms 
\beq\label{sunny} \RHom_B(\phi_* \sF, N) \cong \RHom_{\rQgr T}(\sF, \phi^*N)\eeq
for all $N \in \Perf( B)  $ and $\sF \in D(\rQgr T)$.
But by Proposition~\ref{prop:PsiCrepantCategoricalRes} $\psi$ is weakly crepant, and Lemmata~\ref{lem:ii} and \ref{lem:iii} show that $\Theta$ and quasi-inverse transport the corresponding adjunction to \eqref{sunny}.
\end{proof}

We believe that $D(\rQgr T)$ is also a strongly crepant categorical resolution of $\ms D(B)$, as with $ D(\Lambda)$, but again we have not been able to prove this. 


\section{Intersection theory}\label{INTTHEORY}

\numberwithin{equation}{section}

Recall that, given suitable finiteness on $\Hom$ and $\Ext$ groups, one may define an ``intersection product'' on an abelian category $\mathbf{C}$ with enough injectives by
\[ \sA \cdot \sB = \sum_{i \geq 0} (-1)^{i+1} \dim_\kk \Ext^i_{\mathbf{C}}(\sA, \sB).\]
(This is, for example, the intersection product defined in \cite{MS}.)
We say that $\sA \cdot \sB$ is {\em well-defined} if  $\dim_\kk \Ext^i_{\mathbf{C}}(\sA, \sB) < \infty$ for all $i$ and all but finitely many $\Ext^i_{\mathbf{C}}(\sA, \sB)$ are zero.
In this section, we consider the intersection theory of the ``exceptional locus'' of the map $\phi: \rqgr T \to \rmod B$ defined in Section~\ref{DEFCAT} and show that it consists of a chain of $n$ lines, each with self-intersection $-2$: in other words, the classic resolution graph of an $\AA_n$ singularity.  

In this section, for $1 \leq i \leq n$ let $L_i = T/\mf l_i$, where $\mf l_i$ is defined in Lemma~\ref{lem:X}.
We comment briefly on the term ``exceptional locus of $\phi$''.
Each of the $L_i$ is annihilated by the ideal $\langle x,y,z \rangle$ of $B$ and so is supported on the singularity of $B$.
By Proposition~\ref{prop:june} this is, morally,  the locus where $\phi:  \rqgr T \to B$ is not an isomorphism. 

Our analysis of the $L_i$ begins with the observation that the $L_i$ are line modules.
\begin{lemma}\label{lem:linehilb}
The composition $\kk_{\qb}[t_{i-1}, t_i] \to T \to L_i$ is an isomorphism of graded rings, and thus $L_i$ has Hilbert series 
\begin{align*}
h_{L_i}(t) = (1-t)^{-2}. \tag*{\qed} 
\end{align*}
\end{lemma}

\begin{remark}\label{rem:P1}
The homomorphisms $T \to L_i$ give an inclusion of $\rqgr L_i $ in $\rqgr T$.
Since, as is well-known, $\rqgr \kk_{\qb}[x,y] \simeq \coh \PP^1$ \cite{AV}, it is reasonable to think of the $L_i$ as being projective lines inside $\rqgr T$.
\end{remark}

The main goal of this section is to prove that the line modules $L_i$ have the intersection theory of an $\mathbb{A}_n$ singularity. More specifically:

\begin{theorem}\label{thm:inttheory}
Let  $1 \leq i , j \leq n$.  
\begin{enumerate}[{\normalfont (1)},topsep=0pt,itemsep=0pt,leftmargin=*]
\item The intersection product $L_i \cdot L_j$ is well-defined.
\item We have
\[ L_i \cdot L_j = \left\{\begin{array}{cl} -2 & \text{if } i=j, \\
    1 & \text{if } |i-j| = 1, \\
    0 & \mbox{otherwise}. \end{array}\right.\]
   \end{enumerate} 
\end{theorem}    

The proof of this result will take up most of the section. Our first step will be to prove that the line ideals $\mf l_j$ are invertible in $\rqgr T$, in the sense of Corollary~\ref{cor:invertible}. Our candidate for the inverse of $\mf l_j$ will be 
\begin{align*}
\sJ_j = \T + z^{-1} t_{j-1} t_j \T, \qquad 1 \leqslant j \leqslant n.
\end{align*}
(Note that $\sJ_j \subseteq Q_{\gr}(T)$.) We first require some preparatory lemmata.

\begin{lemma}\label{sublem:one}
The multiplication map $\sJ_j \otimes_\T \mf l_j \to Q_\text{\normalfont{gr}}(\T)$ has image in $\T$.
\end{lemma}
\begin{proof}
It suffices to show that multiplication by $z^{-1} t_{j-1} t_j$ maps the generators of $\mf l_j$ into $\T$. This is clear for $z$. Using the relations in $\T$, we have
\begin{gather*}
z^{-1} t_{j-1} t_j \cdot x = \qb z^{-1} t_{j-1} z^j t_{j-1} = \qb z^{j-1} t_{j-1}^2, \\
z^{-1} t_{j-1} t_j \cdot y = \qb z^{-1} t_j z^{n-j+1} t_j = \qb z^{n-j} t_j^2,
\end{gather*}
both of which lie in $\T$ since $1 \leqslant j \leqslant n$. Moreover, for $k \neq j,j-1$, we have
\begin{align*}
z^{-1} t_{j-1} t_j \cdot t_k = \left\{
\begin{array}{ll}
\qb z^{-1} t_j z^{k-j} t_j t_{k-1} = \qb z^{k-j-1} t_j^2 t_{k-1} & \text{if } k \geqslant j+1, \\
\qb z^{-1} t_{j-1} z^{j-k-1} t_{k+1} t_{j-1} = \qb z^{j-k-2} t_{j-1}^2 t_{k+1} & \text{if } k \leqslant j-2,
\end{array}
\right.
\end{align*}
and this lies in $\T$ in both cases.
\end{proof}

Recall our notation $\msOk = \T[t_k^{-1}]_0$.

\begin{lemma}\label{sublem:two}
We have
\begin{align*}
\mf l_j[t_k^{-1}]_0 =
\left\{
\begin{array}{cl}
\msOk & \text{if } k \neq j,j-1, \\
\msOk t_{k+1} t_k^{-1} & \text{if } k = j,\\
\msOk t_{k-1} t_k^{-1} &\text{if } k = j-1,
\end{array}
\right.
\end{align*}
with the understanding that $t_{-1}t_0^{-1} \coloneqq x$ and $t_n t_{n+1}^{-1} \coloneqq y$.
\end{lemma}
We warn the reader that the convention on $t_{-1}$ and $t_n$ in the lemma is slightly different from the convention in Section~\ref{TORICVAR}.
\begin{proof}
We necessarily have $\mf l_j[t_k^{-1}]_0 \subseteq \msOk$. If $k \neq j,j-1$, then $1 = t_k t_k^{-1} \in \mf l_j[t_k^{-1}]_0$, so we have equality. \\
\indent Now suppose that $k = j$; in particular, we have $k \geqslant 1$. It suffices to show that 
\begin{align*}
x,y,z, t_\ell t_k^{-1} \in \msOk t_{k+1} t_k^{-1} \quad \text{for } \ell \in \{0, 1, \dots, k-2,k+1, \dots n\}.
\end{align*}
We need to consider the cases $k \neq n$ and $k = n$ separately. \\
\indent First suppose that $k \neq n$. Then, using the relations in $\T$ and the fact that $1 \leqslant k \leqslant n-1$, we have
\begin{align*}
z &= \qb t_{k-1} t_k^{-1} t_{k+1} t_k^{-1} \in \msOk t_{k+1} t_k^{-1}.
\end{align*}
This also implies that $x,y \in \msOk t_{k+1} t_k^{-1}$, since
\begin{align*}
x = \qb z^k t_{k-1} t_k^{-1} \quad \text{and} \quad y = \qb z^{n-k} t_{k+1} t_k^{-1}
\end{align*}
again using that $1 \leqslant k \leqslant n-1$. To show that $t_\ell t_k^{-1} \in \msOk t_{k+1} t_k^{-1}$, we consider two cases:
\begin{align*}
\ell \geqslant k+1:& \quad t_k t_\ell = \qb z^{\ell-k-1} t_{k+1} t_{\ell-1} \quad \Rightarrow \quad t_{\ell} t_k^{-1} = \qb z^{\ell-k-1} t_{k+1} t_k^{-1} t_{\ell-1} t_k^{-1}, \\
\ell \leqslant k-2:& \quad t_\ell t_k = \qb z^{k-\ell-1} t_{\ell+1} t_{k-1} \quad \Rightarrow \quad t_{\ell} t_k^{-1} = \qb z^{k-\ell-1} t_{\ell+1} t_k^{-1} t_{k-1} t_k^{-1};
\end{align*}
in the first case, this element plainly lies in $\msOk t_{k+1} t_k^{-1}$, while the same is true in the latter case since $z \in \msOk t_{k+1} t_k^{-1}$ and $k-\ell-1 \geqslant 1$. \\ 
\indent Now suppose that $k=n$. We need to show that 
$\mf l_n[t_n^{-1}]_0 = \msOn y$. We have $z = \qb y t_{n-1}t_n^{-1} \in \msOn y$ and
\begin{align*}
x = \qb y^{-1} z^{n+1} = \qb y^n \big(t_{n-1}t_n^{-1}\big)^{n+1}\in \msOn y. 
\end{align*}
Finally, for $\ell \leqslant n-2$:
\begin{align*}
t_\ell t_n = \qb z^{n-\ell-1} t_{\ell+1} t_{n-1} \quad \Rightarrow \quad t_\ell t_n^{-1} = \qb z^{n-\ell-1} t_{\ell+1}t_n^{-1} t_{n-1}t_n^{-1} \in \msOn,
\end{align*}
since $z^{n-\ell-1} \in \msOn$. \\
\indent The case $k = j-1$ is similar, except one must consider the subcases $k=0$ and $k \neq 0$ separately; we omit the proof.
\end{proof}

\begin{lemma}\label{sublem:three}
We have
\begin{align*}
\sJ_j[t_i^{-1}]_0 =
\left\{
\begin{array}{cl}
\msOi & \text{if } i \neq j,j-1, \\
\msOi {t_i}t_{i+1}^{-1} & \text{if } i = j,\\
\msOi {t_i}t_{i-1}^{-1} &\text{if } i = j-1,
\end{array}
\right.
\end{align*}
again with the understanding that $t_{-1}t_0^{-1} \coloneqq x$ and ${t_n}{t_{n+1}^{-1}} \coloneqq y$.
\end{lemma}
\begin{proof}
First suppose that $i \neq j,j-1$; we need to show that $\sJ_j[t_i^{-1}]_0 = \msOi$. Since $\T \subseteq \sJ$, the inclusion $ \msOi \subseteq \sJ_j[t_i^{-1}]_0$ is clear. For the reverse inclusion, it suffices to show that $z^{-1} t_{j-1}t_j t_i^{-2} \in \msOi$. Since $i \neq j,j-1$, we have $j-i \geqslant 2$ or $i-j \geqslant 1$. If $j-i \geqslant 2$, then the relations in $\T$ give
\begin{align*}
t_i t_j = \qb z^{j-i-1} t_{i+1} t_{j-1} \quad &\Rightarrow \quad t_j t_i^{-1} = \qb z^{j-i-1} t_{i+1} t_{j-1}t_i^{-2} \\
&\Rightarrow \quad z^{-1} t_{j-1}t_j t_i^{-2} = \qb z^{j-i-2} t_{j-1}^2 t_{i+1}t_i^{-3} \in \msOi.
\end{align*}
If instead $i-j \geqslant 1$, then 
\begin{align*}
t_{j-1} t_i = \qb z^{i-j} t_j t_{i-1} \quad &\Rightarrow \quad t_{j-1} t_i^{-1} = \qb z^{i-j} t_j t_{i-1}t_i^{-2} \\
&\Rightarrow \quad z^{-1} t_{j-1}t_j t_i^{-2} = \qb z^{i-j-1} t_j^2 t_{i-1}t_i^{-3} \in \msOi.
\end{align*}
\indent Now suppose $i=j$, so that we need to show $\sJ_i[t_i^{-1}]_0 = \msOi {t_i}t_{i+1}^{-1}$. We consider the cases $1 \leqslant i \leqslant n-1$ and $i=n$ separately. \\
\indent First suppose that $1 \leqslant i \leqslant n-1$. In this case, the claim essentially follows after observing that
\begin{align*}
1 = t_{i+1} t_i^{-1} \cdot {t_i}t_{i+1}^{-1} \quad \text{and} \quad \qb z^{-1} t_{i-1} t_i t_i^{-2} = {t_i}t_{i+1}^{-1}.
\end{align*}
If instead $i=n$, then we need to show that $\sJ_i[t_i^{-1}]_0 = \msOi y^{-1}$. This follows since
\begin{align*}
1 = y \cdot y^{-1} \quad \text{and} \quad z^{-1} t_{n-1} t_n t_n^{-2} = z^{-1} t_{n-1}t_n^{-1} = \qb y^{-1}.
\end{align*}
\indent The case $i=j-1$ is similar and therefore omitted; we remark that the proof in this case requires analysing the subcases $j=1$ and $2 \leqslant j \leqslant n$ separately.
\end{proof}

We can now establish invertibility of the line ideals $\mf l_j$:

\begin{proposition}\label{prop:InvertibleLineIdeal}
The $\mf l_j$ are invertible $T$-bimodules in $\rqgr T$, in the sense of Corollary~\ref{cor:invertible}.
\end{proposition}
\begin{proof}
As before, Theorem~\ref{newthm:catequiv} allows us to prove this locally, and this is precisely the content of Lemma~\ref{sublem:three}.
\end{proof}

Our basic strategy to prove Theorem~\ref{thm:inttheory} is to use the local-to-global spectral sequence \eqref{localglobal}.
We thus need to compute $\shExt^s_{\X}(\bpi L_i, \bpi L_j)$.  
Since $L_i, L_j$ are $T$-bimodules, this will be an $\ms{O}_\X$-bimodule.
We first compute:
\begin{lemma}\label{intfour}
Identify $\msOk $ with $\kk_{\qb}[t_{k-1}t_k^{-1}, t_{k+1} t_k^{-1}]$, using Lemma~\ref{lem:dag}\emph{(1)}.
Then 
\[ (\bpi L_i)(U_k)  = \left\{ \begin{array}{cl}
\kk_{\qb}[t_{i-2}t_{i-1}^{-1}, t_{i} t_{i-1}^{-1}]/(t_{i-2}t_{i-1}^{-1}) & \text{if } k = i-1, \\
\kk_{\qb}[t_{i-1}t_i^{-1}, t_{i+1} t_i^{-1}]/(t_{i+1} t_i^{-1}) & \text{if } k =i, \\
0 & \mbox{otherwise.}
\end{array} \right.
\]
Further, $(\bpi L_i)(U_k)$ is a GK 1-critical $\msOk$-module provided that it is nonzero, i.e.~it has GK-dimension 1 and any proper quotient is finite-dimensional.
\end{lemma}
\begin{proof} This follows directly from Lemma~\ref{sublem:two}.
\end{proof}

\begin{corollary}\label{cor:intsix}
If $|j-i| \geq 2$ then $\Ext^s_\X(\bpi L_i, \bpi L_j) = 0$ for all $s$ and so $L_i \cdot L_j = 0$ is well-defined.
\end{corollary}
\begin{proof} This follows from \eqref{localglobal} and the observation that on all charts $U_k$ either $\bpi L_i(U_k) = 0$ or $\bpi L_j(U_k) = 0$.
\end{proof}

We further have the following computation, which we give without proof.

\begin{lemma}\label{lem:intfive}
Let $C = \kk_q[u,v]$ be a quantum polynomial ring, let  $L, L'$ be $C/(u)$ and $C/(v)$, respectively, and let $P = \kk_q[u,v]/\langle u,v \rangle$.
    Then 
    \[ \Hom_C(L,L) \cong L \cong \Ext^1_C(L,L), \quad \Ext^1_C(L,L')\cong P,\]
    and all other $\Ext^s_C(L,L)$ and $\Ext^s_C(L, L')$ vanish. 
    Thus \[ \shExt^s_\X(\bpi L_i, \bpi L_{i+1}) = \left\{ \begin{array}{cl} \bpi P_i & \text{if } s =1 \\
 0 & \mbox{otherwise,} 
 \end{array} \right.
\] 
where $P_i   = T/(\mf l_i + \mf l_{i+1})$.\qed
\end{lemma}

\begin{lemma}\label{lem:inteight}
We have
\[ \Ext^s_\X(\bpi L_i, \bpi L_{i+1}) = \left\{\begin{array}{cl}
\kk & \text{if } s=1, \\
0 & \mbox{otherwise,} \end{array} \right.\]
so $L_i \cdot L_{i+1} = 1$ is well-defined.
\end{lemma}
\begin{proof}
By Lemma~\ref{lem:intfive}, the local-to-global spectral sequence \eqref{localglobal}
collapses to give 
\[ \Ext^s_\X(\bpi L_i, \bpi L_{i+1}) = H^{s-1} (\X,\shExt^1_\X(\bpi L_i, \bpi L_{i+1})) \cong \H^{s-1}_{\rQgr T}(P_i).
\]
A local cohomology computation gives the result.
\end{proof}

We now complete the proof of Theorem~\ref{thm:inttheory}.

\begin{proposition}\label{prop:intnine}
\[\shExt^s_\X(\bpi L_i, \bpi L_i) \cong \left\{
\begin{array}{cl}
\bpi L_i & \text{if } s=0, \\
\bpi L_i (-2) & \text{if } s=1, \\
0 & \mbox{otherwise,} \end{array} \right.
\]
and $L_i \cdot L_i = -2$ is well-defined.
\end{proposition}
\begin{proof}
Consider the exact sequence $0 \to \mf l_i \to T \to L_i \to 0$
and the associated long exact sequence
\beq\label{LES}
0 \to \shHom_\X(\bpi L_i, \bpi L_i) 
\stackrel{f}{\to} \shHom_\X(\bpi T, \bpi L_i) 
\stackrel{g}{\to} \shHom_\X(\bpi \mf l_i, \bpi L_i) 
\stackrel{h}{\to} \shExt^1_\X(\bpi L_i, \bpi L_i) \to 0,
\eeq 
where all other terms in the long exact sequence vanish.
On each open set $U_k$, by Lemma~\ref{lem:intfive} we see, first that $f,h$ are isomorphisms and $g=0$, and second that
\beq\label{extstar}
\shExt^1_\X(\bpi L_i, \bpi L_i)(U_k) \cong (\bpi L_i)(U_k) \quad \mbox{for all $k$.}
\eeq
Thus $f$ is a global isomorphism and 
\[\shHom_\X(\bpi L_i, \bpi L_i) \cong \shHom_\X(\bpi T, \bpi L_i) \cong \bpi L_i.
\]

We can therefore compute $\shExt^1_\X(\bpi L_i, \bpi L_i)$ as $\shHom_\X(\bpi \mf l_i, \bpi L_i)$.
This object is isomorphic to $\bpi(\sJ_i \otimes_T L_i)$ by 
Proposition~\ref{prop:InvertibleLineIdeal} and the equivalence $\rqgr \sT \simeq \rmod \ms O_\X$ in Theorem~\ref{newthm:catequiv}.
Again using Proposition~\ref{prop:InvertibleLineIdeal}
$\sJ_i \otimes_T L_i \cong \sJ_i/\sJ_i \mf l_i$ is isomorphic in $\rqgr T$ to
\[ \sJ_i/T \cong \frac{z^{-1}t_{i-1} t_i T}{z^{-1}t_{i-1} t_i T\cap T}.\]
Since $z^{-1}t_{i-1} t_i \mf l_i \subseteq T$, there is a surjection
\beq\label{tues1}
\frac{z^{-1}t_{i-1} t_i T}{z^{-1}t_{i-1} t_i \mf l_i} \twoheadrightarrow \frac{z^{-1}t_{i-1} t_i T}{z^{-1}t_{i-1} t_i T\cap T}.
\eeq
Applying $\bpi(-)(U_k)$, both terms in \eqref{tues1} become isomorphic to $(\bpi L_i)(U_k)$.
By the GK-criticality in Lemma~\ref{intfour},  the surjection \eqref{tues1} must be an isomorphism on all charts and thus an isomorphism in $\rqgr \T$.
Thus
\[ \shExt^1_\X(\bpi L_i, \bpi L_i) \cong \bpi\Big(\frac{z^{-1}t_{i-1} t_i T}{z^{-1}t_{i-1} t_i \mf l_i}\Big)\cong \bpi(T(-2)/\mf l_i(-2)) \cong \bpi(L_i(-2)).\]

By the above computations, the $E_2$-page of \eqref{localglobal} for $\Ext^s_\X(\bpi L_i, \bpi L_i)$ is
\[ \begin{array}{cc} 
H^0(\X, \bpi L_i(-2)) \hspace{10pt}&\hspace{10pt} H^1(\X, \bpi L_i(-2)) \\[20pt]
H^0(\X, \bpi L_i) & H^1(\X, \bpi L_i) \end{array}\]
where all the maps are zero as we are on the $E_2$-page.
Thus
\begin{align*}
\Hom_\X(\bpi L_i, \bpi L_i) & \cong H^0(\X, \bpi L_i), \\
\Ext^1_\X(\bpi L_i, \bpi L_i) & \cong H^0(\X, \bpi L_i(-2)) \oplus H^1(\X, \bpi L_i), \\
\Ext^2_\X(\bpi L_i, \bpi L_i ) & \cong H^1(\X, \bpi L_i(-2)).
\end{align*}
We see that 
\[ L_i \cdot L_i = -\chi(L_i) + \chi(L_i(-2)), \]
where for a graded $T$-module $M$ we have 
\[
\chi(M) = \sum_{i \geq 0} (-1)^i \dim_\kk H^i(\X, \bpi M) =  \sum_{i \geq 0} (-1)^i \dim_\kk \H^i_{\rQgr T}( \pi M),
\]
where the second equality follows from the  equivalence between $\rmod \ms{O}_\X$ and $\rqgr T$.

We will now show that $\chi(L_i(n)) = n+1$ for all $n\in \ZZ$, from which the fact that $L_i \cdot L_i = -2$ will follow. We calculate this via a local cohomology computation.
If $k \not\in\{i-1, i\}$ then $L_i[t_k^{-1}] = 0$. 
The \v{C}ech complex for $L_i$ is thus
\[ 0 \to L_i \to L_i[t^{-1}_{i-1}] \oplus L_i[t_i^{-1}] \to L_i[t_{i-1}^{-1}, t_i^{-1}] \to 0.\]
From Lemma~\ref{lem:linehilb}, this is the complex
\[ 0 \to \kk[t_{i-1}, t_i] \to \kk[t_{i-1}^{\pm 1}, t_i] \oplus \kk[t_{i-1}, t_i^{\pm 1}] \to \kk[t_{i-1}^{\pm 1}, t_i^{\pm 1}]\to 0\]
of $\kk$-vector spaces.
Thus $H^s_{T_+}(L_i) = 0$ for $ s \neq 2$, and 
\beq\label{tues2}
\dim_\kk H^2_{T_+}(L_i)_n = \left\{\begin{array}{cl} 0 & \text{if } n \geq 0, \\  -(n+1) & \text{if } n \leq -1. \end{array}\right.
\eeq
That $\chi(L_i(n)) = n+1$ follows from  \eqref{tues2} and Lemma~\ref{lem:linehilb}.
\end{proof}

\begin{remark}\label{rem:kernel}
The computations in Proposition~\ref{prop:intnine} show that 
\[\H^\bullet_{\rQgr \T}(L_i(-1)) = \RR \phi_* (L_i(-1))=0\] for $1 \leq i \leq n$.
Thus the $L_i(-1)$ are contained in the kernel of $\phi_*$.  
We conjecture that the modules $L_1(-1), \dots, L_n(-1)$ generate the kernel of $\phi_*$.  
This is the subject of ongoing research.  
\end{remark}

\bibliographystyle{amsalpha}
\bibliography{bibliography}
\end{document}